\documentclass[11pt]{amsart}
\usepackage{graphicx}
\usepackage[utf8]{inputenc}
\usepackage[T1]{fontenc}
\usepackage{textcomp}
\usepackage[english]{babel}
\usepackage{url}                
\usepackage{amsmath,amssymb,amsthm} 
\usepackage{mathrsfs}
\usepackage{enumitem}   
\usepackage{csquotes}
\usepackage{import}
\usepackage[dvipsnames]{xcolor}
\usepackage[colorlinks=true,linktoc=page]{hyperref}
\hypersetup{urlcolor=NavyBlue, citecolor=red, linkcolor=RubineRed}
\usepackage{tcolorbox}
\usepackage[all,cmtip]{xy}
\usepackage{manfnt}

\usepackage[backend=biber, style=alphabetic,
natbib=true,hyperref=true]{biblatex}

\numberwithin{equation}{section}

\hypersetup{
pdfauthor = {Tinhinane Amina AZZOUZ},
pdfsubject = {p-adic differential equations},
pdftitle = {Spectrum of p-adic linear differential equations
  II},
pdfkeywords = {formal differential equations, spectral theory},
}

\theoremstyle{plain}
\newtheorem{Theo}{Theorem}[section] 
\newtheorem{Pro}[Theo]{Proposition}        
\newtheorem{Lem}[Theo]{Lemma}            
\newtheorem{cor}[Theo]{Corollary}

\theoremstyle{definition}
\newtheorem{Defi}[Theo]{Definition}
\newtheorem{exam}[Theo]{Example}
\newtheorem{exams}[Theo]{Examples}
\newtheorem{nota}[Theo]{Notation}

\newtheorem{conv}[Theo]{Convention}

\theoremstyle{remark}
\newtheorem{rem}[Theo]{Remark}

\def\ogg~{{\rm \og}}   

\def\nn{\noindent}

\def\qq{\nn\quad}
\def\qqq{\nn\quad\quad}

\def\emptyset{\varnothing}

\def\NN{{\mathbb N}}    
\def\ZZ{{\mathbb Z}}     
\def\RR{{\mathbb R}}    
\def\AA{{\mathbb A}}    
\def\PP{{\mathbb P}}

\def\cA{{\mathcal A}}   \def\cM{{\mathcal M}} \def\cS{{\mathcal S}}      \def\cC{{\mathcal C}}   \def\cO{{\mathcal O}} \def\cU{{\mathcal U}} \def\cD{{\mathcal D}}     \def\cE{{\mathcal E}}  \def\cK{{\mathcal K}}     \def\cL{{\mathcal L}} \def\cR{{\mathcal R}}    

\renewcommand{\hom}{\operatorname{Hom}}
\newcommand\ct{\operatorname{cotan}}         

\newcommand{\Ker}{\operatorname{Ker}}

\newcommand{\res}{\operatorname{Res}}
\newcommand{\car}{\operatorname{Card}}

\renewcommand{\dim}{\operatorname{dim}}
\newcommand{\rank}{\operatorname{rank}}

\newcommand{\disf}[2]{D^+(#1,#2)}
\newcommand{\diso}[2]{D^-(#1,#2)}

\newcommand{\h}[1]{\mathscr{H} (#1)}
\newcommand{\A}[1]{\AA^{1,\mathrm{an}}_{#1}}

\newcommand{\Po}[1]{\PP^{1,\mathrm{an}}_{#1}}
\newcommand{\D}[1]{\frac{\mathrm{d}}{\mathrm{d#1}}}

\newcommand{\LL}[2]{\cL_{#1}(#2)}

\newcommand{\Lk}[1]{\LL{F}{#1}}
\newcommand{\nsp}[1]{\rVert #1\rVert_{\mathrm{Sp}}}

\newcommand{\piro}[2]{\pi_{#1/#2}}
\newcommand{\pik}[1]{\piro{#1}{F}}

\newcommand{\nor}[1]{\rVert #1\rVert}

\newcommand{\discf}[3]{D^+_{#1} (#2,#3)}
\newcommand{\disco}[3]{D^-_{#1} (#2,#3)}

\newcommand{\couf}[3]{A^+ (#1,#2,#3)}
\newcommand{\couo}[3]{A^- (#1,#2,#3)}

\newcommand{\Couf}[4]{A^+_{#1} (#2,#3,#4)}
\newcommand{\Couo}[4]{A^-_{#1} (#2,#3,#4)}

\newcommand{\DI}[2]{\mathbb{D}_{#1,#2}}
\newcommand{\di}[1]{\mathbb{D}_{#1}}

\newcommand{\Fonction}[4]{\begin{array}[c]{rcl} 
                            #1&\longrightarrow&#2\\ #3&\mapsto&
                                                                #4\\ \end{array}}

\newcommand{\fra}[1]{\frac{1}{#1}}

\renewcommand\phi{\varphi}
\renewcommand\epsilon{\varepsilon}
\def\ct{\hat{\otimes}}

\def \ot{\otimes}
\def \<{\langle}
\def \>{\rangle}

\def\|{\rVert}
\def\*{\blacklozenge}

\def\mr{\mathrm}

\def\dT{\frac{\mr{d}}{\mr{dT}}}

\def \R+{\RR_{+}}
\def\AF{\A{F}}

\def \Pf{\Po{F}}

\def \-1{^{-1}}

\def \Hx{\mathscr{H}(x)}
\def \Hy{\mathscr{H}(y)}

\def \rk{\widetilde{F}}
\def \crk{\mathrm{char}(\rk)}

\def\(({(\!(}
\def\)){)\!)}
\def\DD{\mathscr{D}}

\def \dI{\di{I}}

\def\kac{\widehat{F^{alg}}}
\def\wac{\widehat{\Omega^{alg}}}  
\def\EE{\mathscr{E}}

\def\Af{\mathsf{A}}
\def\F{\mathscr{F}}
\def\mscr{\mathscr}
\def\RS{\cR^{\mr{Sp}}}
\def\Xh{X^{\rm{hyp}}}
\def\Spe{\Sigma_{\F,c,f}}
\newcommand{\spe}[1]{\Sigma_{#1,c,f}}

\renewcommand{\geq}{\geqslant}
\renewcommand{\leq}{\leqslant}

\begin{document}

\title{Spectrum of p-adic linear differential equations
  II: Variation of the spectrum}

\author{Tinhinane A. AZZOUZ}
\address{Tsinghua University, Yau Mathematical Sciences Center,
  Yanqi Lake Beijing Institute of Mathematical Sciences and applications,
  Beijing, China}
\email{azzouzta@bimsa.cn}
\thanks{The author was supported by NSFC Grant \#12250410239.}



\begin{abstract}
The primary objective of this paper is to generalize the results of
\cite{Azz24} to the case of quasi-smooth Berkovich curves by
establishing a connection between the spectrum and the radii of
convergence. To achieve this, we investigate the continuity and
variation of the spectrum of a $p$-adic linear differential
equation. Our findings demonstrate that the spectrum can be governed
by any controlling graph of the radii of convergence. Furthermore, by
analyzing the shape of the spectrum, we prove that approximating the
connection enables an accurate estimation of spectral radii of
convergence. As a result, we obtain a decomposition theorem with
respect to the spectrum for differential equations defined over a
quasi-smooth curve at a generic point $x$. This decomposition refines
the one provided by the spectral radii of convergence. Previous works
\cite{DR77,np3} have shown that decomposition with respect to the
spectral radii can be extended to a neighborhood of $x$. We prove that
this can also hold for the decomposition with respect to the spectrum.
\end{abstract}
\maketitle

\section{Introduction}

Considering a linear differential equation in the ultrametric setting,
the radii of convergence of the solutions are among the most powerful
invariants. Indeed, they exhibit features not present in the complex
field. In particular, for a linear differential equation defined over
an open disk, the solutions may not converge everywhere; furthermore,
the radii of convergence may not be constant. However, based on the
recent works of F. Baldassarri, K.S. Kedlaya, J. Poineau, and
A. Pulita \cite{bal10, BDV08, ked13, np2, np3, np4}, it is now
well-established that the radii of convergence exhibit continuity and
are effectively governed by a locally finite graph. This controlling
graph exclusively dictates their variations. These findings have
profound implications for the finite dimensionality of the de Rham
cohomology. Particularly noteworthy are the results obtained under
specific assumptions, wherein the finiteness of the controlling graph
for the radii directly implies the finite dimensionality of the de
Rham cohomology. Such implications hold significant importance due to
the deep interconnections between $p$-adic linear differential equations
and other fields in number theory, including $p$-adic cohomology, $p$-adic
Hodge theory, and Galois representations.

Moreover, in a recent paper \cite{BP20}, V. Bojkovi\'c and J. Poineau
prove that knowledge of the radii of convergence can be used to
describe the ramification of finite étale morphisms. In particular, they provide the link with the upper ramification jumps of the extension of the residue field of the point given by a finite morphism.

The aim of our work is to provide a dictionary between the spectral
radii of convergence and the spectrum of a linear differential
equation. The spectrum of a differential equation is a geometric
invariant lying in the Berkovich analytic line, which we introduce in
our previous works \cite{Azz20, Azz21, Azz24}. This paper is a sequel
to \cite{Azz24}, where we determine the general shape of the spectrum
of a linear differential equation; additionally, for differential
equations defined on an analytic domain of the analytic line, we
establish the link between the spectral radii of convergence and the
spectrum and provide a decomposition theorem with respect to the spectrum.

Before summarizing the main ideas of the paper, let us fix the
settings. We consider an algebraically closed complete ultrametric
field $(F, |.|)$ with mixed characteristic. We set $\omega =
|p|^{\frac{1}{p-1}}$,  where $\crk = p > 0$. We consider analytic spaces in
the sense of Berkovich. We denote $\AF$ (resp. $\Pf$) as the analytic
affine (resp. projective) line, $\disf{c}{r}$ (resp. $\diso{c}{r}$) as
the closed (resp. open) disk of $\AF$ centered at $c$ with radius equal to $r$, and $x_{c,r}$ as the Shilov boundary
of the closed disk $\disf{c}{r}$. We fix $T$ to be a coordinate
function on $\AF$. For a complete extension $(\Omega, |.|)$ of $(F,
|.|)$, we set $\pi_{\Omega/F} : \A{\Omega} \to \AF$ as the canonical
projection. By a differential equation $(\F, \nabla)$ over a
quasi-smooth analytic curve $(X,\cO_X)$, we mean a locally free $\cO_X$-module $\F$ of
finite type with an integrable connection $\nabla$. For $x \in \Xh = X
\setminus X(F)$, we set $\RS_1(x, (\F, \nabla)) \leq \cdots \leq
\RS_n(x, (\F, \nabla))$ to be the spectral radii of convergence of
$(\F, \nabla)$ at $x$, which are the normalized radii of convergence of the solutions on a generic
disk of $x$ (the radii of convergence as defined on \cite{np3} on the
generic disk $D(x)$ of $x$). For a precise
definition, see Section~\ref{sec:radii-conv-line-3}. For any $x \in
X$, we denote by $\Hx$ the associated complete residue field of
$x$. For $x \in \Xh$, let $d$ be a nonzero bounded $F$-linear derivation, defined on an affinoid neighborhood
of $x$. We set $\F(x) = \F_x \otimes_{\cO_{X,x}} \Hx$. Since $\F_x$ is a
free $\cO_{X,x}$-module, by choosing a basis, $\F(x)$ acquires a
natural structure of $\Hx$-Banach space, fortunately, this
does not depend on the choosing of the basis. Thus, we obtain an
$F$-Banach space structure on $\F(x)$. Let
$\nabla_{d}:\F(x)\to\F(x)$ be the bounded connection on $\F(x)$
induced by $\nabla$ after choosing $d$. The spectrum of $(\F, \nabla)$
at $x$ with respect to $d$
is the spectrum, in the sense of Berkovich, of $\nabla_d$ as a bounded
$F$-linear operator on $\F(x)$. 

This paper deals with the following main issues:
\begin{itemize}
\item Studying the continuity and variation of the spectrum of a differential equation when $x$ varies on $\Xh$.
\item Establishing the link between the spectral radii of convergence
  and the spectrum, as in \cite{Azz24}, for the case of a quasi-smooth
  curve.
\item Provide a decomposition with respect to the spectrum, as in \cite{Azz24}, for the case of a quasi-smooth
  curve. 
\item Approximating the spectral radii of convergence by approximating
  the connection.
\item Providing a decomposition of a differential equation at a point 
$x$, with respect to the spectra corresponding to branches out of 
$x$. Then extend this decomposition to a neighborhood of $x$.
\end{itemize}

We begin by explaining what we mean by studying the variation of the
spectrum. Let $X$ be a connected affinoid domain of $\AF$, and let
$(\F,\nabla)$ be a differential equation over $X$ of rank $n$. Let
$x\in \Xh$ be not of type (4), and let $c\in F$ such that $x=x_{c,r}$
with $r>0$. Let $f$ be a coordinate function on $X$, in the sense where the induced
map $X\to \AF$ by $F[T]\to \cO(X)$, $T-c\mapsto f$, is a closed
immersion and $f(X)=X$. We assume that for all $y$, we
have $f(y)=y$. We set $d_f=f\D{f}=\frac{f}{\dT(f)}\dT$. From
\cite{Azz24}, for all $y\in (c,x]\cap X$ there exist
$\omega_1,\cdots,\omega_\mu\in\AF$, such that the spectrum has the form
$\Sigma_{\nabla_{d_f},F}(\Lk{\F(y)})=\{\omega_1,\cdots,
\omega_\mu\}+\ZZ_p$. However, we should precise that given two such
coordinate function $f$ and $g$ with $f\ne g$, the spectra
$\Sigma_{\nabla_{d_f},F}(\Lk{\F(y)})$ and
$\Sigma_{\nabla_{d_g},F}(\Lk{\F(y)})$ may not be equal. This due to
the fact that if $\frac{f}{\dT(f)}\ne
\frac{g}{\dT(g)}$, then $\nabla_f$ and $\nabla_g$ are different as
bounded $F$-linear operators. This point is quite important, because the
resulting decomposition with respect to the spectrum may not be the
same. Let us clarify this by an example:

\begin{exam}\label{sec:introduction-4}
  Let $X=\disf{0}{R}\setminus\diso{0,r}$, with $R>1$ and $r<1$. Let
 $c, a_1, a_2, a_3\in F$ with $|c|<r$, $a_i-a_j\not\in\ZZ_p$ for
 $i\ne j$ and
 $|a_1|=|a_2|=|a_3|=1$. We consider the differential module
 $(M,\nabla)$ on $(\h{x_{0,1}},\nabla)$ defined as follows:
 \begin{equation}
  \label{eq:13}
  \nabla\begin{pmatrix}f_1\\f_2\\f_3 \end{pmatrix}=\begin{pmatrix}\dT
   f_1\\ \dT f_2\\ \dT f_3 \end{pmatrix}+
  \begin{pmatrix}
   \frac{a_1}{T} &1& 0\\ 0& \frac{a_2}{T-c}& 1\\
   0&0&\frac{a_3}{T-1} 
  \end{pmatrix} \begin{pmatrix}f_1\\f_2\\f_3 \end{pmatrix}.
 \end{equation}
 Then we have $\Sigma_{\nabla_{d_T},F}(\Lk{M})=(\{a_1\}+\ZZ_p)\cup
 \{x_{0,1}\}$ and $\Sigma_{\nabla_{d_{T-c}},F}(\Lk{M})=(\{a_2\}+\ZZ_p)\cup
 \{x_{0,1}\}$. Let $(M,\nabla)=(M_1,\nabla_1)\oplus (M_2,\nabla_2)$
 (resp. $(M,\nabla)=(M'_1,\nabla'_1)\oplus (M'_2,\nabla'_2)$) the
 decomposition with respect $\Sigma_{\nabla_{d_T},F}(\Lk{M})$
 (resp. $\Sigma_{\nabla_{d_{T-c}},F}(\Lk{M})$) with
 $\Sigma_{(\nabla_1)_{d_T},F}(\Lk{M_1})=\{a_1\}+\ZZ_p$ and
 $\Sigma_{(\nabla'_1)_{d_{T-c}},F}(\Lk{M_1'})=\{a_2\}+\ZZ_p$. We have $\dim
 M_1=\dim M'_1$, but $(M_1,\nabla_1)\not\simeq (M'_1,\nabla'_1)$,
 because $(M_1,\nabla_1)\simeq (\Hx,\dT+\frac{a_1}{T})$ and
 $(M'_1,\nabla'_1)\simeq (\Hx,\dT+\frac{a_2}{T-c})$.
 For the computation of the spectrum we use \cite[Theorem~4.23]{Azz24}.
\end{exam}
We emphasize that the proposition below holds even for a coordinate function defined only on an affinoid domain of $\AF$. On the other hand,we explain later how, by selecting multiple coordinate functions, we can refine the decomposition with respect to the spectrum. 

By equipping $\cK(\AF)$, the set
of compact sets of $\AF$, with the exponential topology (see
Section~\ref{sec:cont-spectr-an} for a precise definition), we can study the continuity of the following map:
\begin{equation}
\Fonction{\Sigma_{\F,c,f}: (c,x]\cap X}{\cK(\AF)}{y}{\Sigma_{\nabla_{d_f},F}(\Lk{\F(y)})},
\end{equation}
where $\Sigma_{\nabla_{d_f},F}(\Lk{\F(y)})$ is the spectrum of
$\nabla_{d_f}$ as an element of $\Lk{\F(y)}$, the $F$-Banach algebra
of bounded $F$-linear endomorphisms of $\F(y)$. We prove the following
result.

\begin{Pro}[Proposition~\ref{sec:vari-spectr-case-4}]\label{sec:introduction}
Let $y_0\in (c,x]\cap X$. Let $z_1,\cdots, z_\nu\in \AF\setminus F$,
$a_1,\cdots,a_\mu\in F$, with $(\nu,\mu) \ne (0,0)$, such that
$\Sigma_{\F,c,f}(y_0)=\{z_1,\cdots,z_\nu,a_1,\cdots,a_\mu\}+\ZZ_p$,
and $\{z_1,\cdots,z_\nu,a_1,\cdots, a_\mu\}$ has minimal
cardinality. Then in
$(c,y_0]\cap X$
(resp. $[y_0,x]$ if $y_0\ne x$) there exists an interval neighborhood $I$ of $y_0$ such that
\begin{enumerate}
 
\item for each $z_i$ there exist $c_{i,1},\cdots, c_{i,{n_i}}\in F$
 ($c_{i,j}$ are not necessarily distinct)
 and $\phi_{i,1},\cdots, \phi_{i,{n_i}}:I\to \R+$ piecewise log-affine on $I$
 such that $z_i=x_{c_{i,j}, \phi_{i,j}(y_0)}$,
\item for each $a_i$ there exist
 $\psi_{i,1},\cdots,\psi_{i,m_i}:I\to\R+$ continuous on $I$ and piecewise log-affine on $I\setminus\{y_0\}$ such that
 $\psi_{i,j}(y_0)=0$ (i.e. $a_i=x_{a_i,\psi_{i,j}(y_0)}$),
\item for all $y\in I$ we have
 \[\Sigma_{\F,c,f}(y)=\{x_{c_{1,1},\phi_{1,1}(y)},\cdots,
  x_{c_{\nu,n_\nu},\phi_{\nu,n_\nu}(y)},
  x_{a_1,\psi_{1,1}(y)},\cdots,x_{a_\mu,\psi_{\mu,m_\mu}(y)}\}+\ZZ_p,\]
 \item the set $\{x_{c_{1,1},\phi_{1,1}(y)},\cdots,
  x_{c_{\nu,n_\nu},\phi_{\nu,n_\nu}(y)},
  x_{a_1,\psi_{1,1}(y)},\cdots,x_{a_\mu,\psi_{\mu,m_\mu}(y)}\}$
  has minimal cardinality.
 \end{enumerate}
\end{Pro}

In particular, Proposition~\ref{sec:introduction} implies the continuity of
the spectrum on $(c,x]\cap X$. Recall that by choosing different derivations for each branch out of $x$ (i.e
each $c$ and $f$), we may obtain different spectra $\Sigma_{\F,c,f}(x)$ (see
Example~\ref{sec:introduction-3}). Hence, in such circumstances,
discussing the continuity of the spectrum across an entire
neighborhood of $x$ in $X$ becomes challenging. One might presume that
fixing the same derivation on an entire neighborhood of $x$ in $X$ is
necessary to achieve continuity throughout this
neighborhood. However, the spectrum is generally not continuous, even
in this scenario. To illustrate, consider the trivial differential
equation on $\AF\setminus\{0\}$ with the derivation $T\dT$ fixed. For
any point $x$ of type (2) within $\AF\setminus[0,\infty)$, the
spectrum exhibits discontinuity at $x$. This phenomenon is further
elaborated in \cite[Theorem~3.20]{Azz24}.

We may ask whether we can choose the same $c_{i,j}$ in both sides out of
$y_0$ in $(c,x]\cap X$. In general it is not possible; consider, for example, the equation of rank one
 $\dT+\frac{a}{T(c-T)}$ defined over $\disf{0}{1}\setminus\{0,c\}$,
with $y_0=x_{0,|c|}$, $|c|<1$ and $|a|>1$. Indeed, in the branch $(0,x_{0,1}]$,
letting $\epsilon>0$ be as small as necessary, 
for $|c|-\epsilon< r\leq|c|$ we have
$\Sigma_{\F,0,T}(x_{0,r})=\{x_{{a}{c\-1},{|ac^{-2}|r}}\}+\ZZ_p=\{x_{{a}{c\-1},{|ac^{-2}|r}}\}$, and for
$|c|\leq r<|c|+\epsilon$, we have
$\Sigma_{\F,0,T}(x_{0,r})=\{x_{0,|a|r\-1}\}+\ZZ_p=\{x_{0,|a|r\-1}\}$. A possible explanation
of this phenomenon is that the spectrum provides the
formal exponent in the first neighborhood around $0$ and in the second one around
$\infty$. Unfortunately, we cannot uniquely determine
$\phi_{i,1},\cdots, \phi_{i,{n_i}}$
(resp. $\psi_{i,1},\cdots,\psi_{i,m_i}$) on the whole $(c,x]\cap X$.

As we pointed out before, we need to choose a derivation to determine the
spectrum, and the resulting spectrum depends on the choice of this derivation. For the case of the analytic line, the derivation $d_f$
makes possible to determine the shape of the spectrum, and also it allows us to establish the link between the radii of
convergence, and provide a decomposition theorem with respect to the
spectrum (\cite{Azz24}).

Assume now that $X$ is a general connected quasi-smooth curve. In this case, a
result of A. Ducros \cite{Duc}, based on the semi-stable reduction
theorem, ensure that there exists a locally finite set $\cS$ of points of type
(2) or (3), for which $X\setminus \cS$ is a disjoint union of open disks and annuli. The set
$\cS$ is called a weak triangulation, and the associated graph
$\Gamma_{\cS}$ is the union of $\cS$ and the skeleton of the annuli of
$X\setminus\cS$. Therefore,
for any point of $X\setminus \cS$ we can choose a derivation of the
kind $d_f$. However, a type (2) point $x\in \cS$ may not have a neighborhood
isomorphic to an analytic domain of the affine analytic line. For this issue,
we proceed as in \cite{np2}, where J. Poineau and A. Pulita provide an étale
morphism between a neighborhood of $x$ and an analytic domain of
$\Pf$ and use it to reduce to the analytic line case. In this
paper, we use these kind of morphisms to pull back derivations of
the form $d_f$.

Let us explain this in more details. Let $X$ be a quasi-smooth
connected analytic curve, let $x\in X$ be
a point of type (2) or (3). Let $\mathscr{C}$ be a quasi-smooth algebraic
curve such that an affinoid neighborhood of $x$ is isomorphic to an
affinoid domain of $\mathscr{C}$ and $x$ is in the interior of
$\mathscr{C}$. Let $T_x\mathscr{C}$ be the set of branches
out of $x$ in $\mathscr{C}$, and we call the set
$T_x\mathscr{C}\setminus T_xX$ the set of missing
branch from $X$. Let $b\in T_x\mathscr{C}$. We
choose the bounded derivation $d_b$ as follows: If $x$ is of type (2),
there exists a finite étale morphism $\Psi_b:
Y_b\to W_b$ as in Theorem~\ref{sec:spectr-radii-diff}
(cf. \cite[Theorem~3.12]{np2}), where $Y_b$ is an affinoid neighborhood
of $x$ in $\mathscr{C}$, $W_b\subset
\AF$, ${\rm deg}(\Psi_b)$ is coprime to $p$, and $\Psi_b\-1(\Psi_b(b))=\{b\}$. If $x$ is of
type (3) we can choose $\Psi_b: Y_b\to W_b$ to be an isomorphism, where
$Y_b$ is an affinoid neighborhood of $x$ in $\mathscr{C}$ and $W_b$
an affinoid domain of $
\AF$. Let $\Psi^{\#}_b: \cO_{W_b}(W_b)\to \cO_{Y_b}(Y_b)$ be the
morphism of $F$-Banach algebras induced by $\Psi_b$.

Let $C_b$ (resp. $C_{\Psi(b)}$) be an open annulus whose skeleton contains
$b$ (resp. $\Psi_b(b)$). Let $c_b\in F$ such that
$C_{\Psi(b)}=\couo{c_b}{r_1}{r_2}$. Let $f_b=\Psi^{\#}_b(T-c_b)$; then
we consider
\begin{equation}\label{eq:A82}
 d_b={\rm deg}(\Psi_b)\cdot f_b\cdot\Psi_b^*(\dT).
\end{equation}

By construction $d_b$ is a nonzero bounded $F$-linear derivation well defined on
$\cO_X(Y_b)$, which extends without any problem to $\h{y}$ for all
$y\in Y_b$. Let $\Omega$ be a complete extension of $(F,|.|)$, we set
\begin{equation}\label{eq:93}
 \Fonction{\delta: \Omega}{\R+}{z}{\inf_{n\in\ZZ}|z-n|}
\end{equation}
and for $\rho\in\R+$ we set 
\begin{equation}\label{eq:94}
 \cR(\rho)=
 \begin{cases}
  1& \text{ if } \rho=0,\\
  \left(\frac{|p|^l\omega}{\rho}\right)^{\fra{p^l}}& \text{ if }
                            |p|^l\leq
                            \rho\leq
                            |p|^{l-1},\;
                            l\in\NN\setminus\{0\},\\
  \frac{\omega}{\rho}& \text{ if } \rho\geq 1.\\ 
 \end{cases}
\end{equation}

Then we prove a generalization of \cite[Theorem~1.1]{Azz24}:
\begin{Theo}[Theorem~\ref{sec:link-betw-spectr}]\label{sec:introduction-1}
Let $(M,\nabla)$ be a differential module $(\Hx, d_b)$. Let
$\phi_b:\AF\to \AF$, $T-c_b\mapsto (T-c_b)^p$, and we set $y^{p^l}$ to
be $\phi^l_b(y)$. We set $N=\deg(\Psi_b)$.
\begin{enumerate}

\item There exist $z_1,\cdots,z_{\nu}\in \AF\setminus F$ and
 $a_1,\cdots, a_\mu\in F$, such that
 \begin{equation}
  \Sigma_{\nabla,F}(\Lk{M})=\{z_1,\cdots,z_\nu, a_1,\cdots,
  a_\mu\}+\ZZ_p,
 \end{equation}
 where $z_i$ has the same type as $x$, and $(\nu,\mu)$ is not
 equal to $(0,0)$. 
\item We can choose $z_i$ and $a_j$ so that the set $\{z_1,\cdots,z_\nu, a_1,\cdots,
  a_\mu\}$ has minimal cardinality. Indeed it is enough to keep only
  $z_i$ and $a_j$ for which we have $\{z_i\}+\ZZ_p\cap
  \{z_{i'}\}+\ZZ_p=\emptyset$ and $\{a_j\}+\ZZ_p\cap \{a_{j'}\}+\ZZ_p=\emptyset$
  for $i\ne i'$ and $j\ne j'$.
 
 \item We choose $\{z_1,\cdots,z_\nu, a_1,\cdots,
  a_\mu\}$ to be minimal. Then we have a unique (up to an isomorphism) decomposition:
  \begin{equation}\label{eq:92}
   (M,\nabla)=\bigoplus_{i=1}^{\nu}(M_{z_i},\nabla_{z_i})\oplus \bigoplus_{j=1}^{\mu}(M_{a_j},\nabla_{a_j}),
  \end{equation}
  where $(M_{z_i},\nabla_{z_i})$ and $(M_{a_j},\nabla_{a_j})$ are
  differential modules over $(\Hx,d_b)$, such that $\Sigma_{\nabla_{z_i},F}(\Lk{M_{z_i}})=\{z_i\}+\ZZ_p$
  and $\Sigma_{\nabla_{a_j},F}(\Lk{M_{a_j}})=\{a_j\}+\ZZ_p$.
 \item For all $a\in F$, the differential module
  $(M_{z_i},\nabla_{z_i}-a)$ (resp. $(M_{a_j},\nabla_{a_j}-a)$) is
  pure, i.e all its spectral radii coincide with each other.
  \item Let $c_i\in F$ and $r_i>0$ be such that $z_i=x_{c_i,r_i}$. If
  $|p|^{l}\leq r_i< |p|^{l-1}$ with $l\in \NN\setminus\{0\}$, then $\car(\{z_i\}+\ZZ_p)=p^l$ and $\{z_i\}+\ZZ_p=\{x_{c_i,r_i},
  x_{c_i+1,r_i},\cdots, x_{c_i+p^l-1,r_i}\}$. If
  $r_i\geq1$, we have $\car(\{z_i\}+\ZZ_p)=1$ and
  $\{z_i\}+\ZZ_p=\{x_{c_i,r_i}\}$.
 \item If $r_i>1$, let $P_{z_i}(N(T-c_b)\dT)$ be a differential polynomial
  associated to $(\Psi_b)_*(M_{z_i},\nabla_{z_i})$ (as a differential
  module over $(\h{\Psi_b(x)},N(T-c_b)\dT)$). Then the image by
  $\pi_{\widehat{\h{\Psi_b(x)}^{alg}}/F}$ of all roots of $P_{z_i}(S)$ (as a
  commutative polynomial) is equal to $z_i$. 
 \item If $|p|^l<r_i\leq|p|^{l-1}$, let $P_{z_i}(Np^l(T-c_b)\dT)$ be a differential
  polynomial associated to $(\phi_b^l\circ\Psi_b)_*(M_{z_i},\nabla_{z_i})$ (as a differential
  module over $(\h{\Psi_b(x)^{p^l}},Np^l(T-c_b)\dT)$). Then the image by
  $\pi_{\widehat{\h{\Psi_b(x)^{p^l}}^{alg}}/F}$ of all roots of $P_{z_i}(S)$ (as a
  commutative polynomial) is equal to \[\{x_{c_i,r_i},
   x_{c_i+1,r_i},\cdots, x_{c_i+p^l-1,r_i}\}.\]
  In the special case
  where $r_i=|p|^{l-1}$ we have \[\{x_{c_i,r_i},
  x_{c_i+1,r_i},\cdots, x_{c_i+p^l-1,r_i}\}=\{x_{c_i,r_i},
  x_{c_i+1,r_i},\cdots, x_{c_i+p^{l-1}-1,r_i}\}.\]

 \item If $r_i\geq 1$, for all $a\in \disf{c_i}{r_i}\cap F$ we have
  \begin{equation}
   \RS_1(x,(M_{z_i},\nabla_{z_i}-a))=\frac{\omega}{r_i}
  \end{equation}
  and for all $a\in F\setminus \disf{c_i}{r_i}$
  \begin{equation}
   \RS_1(x,(M_{z_i},\nabla_{z_i}-a))=\cR(\delta({a-c_i}))=\frac{\omega}{|a-c_i|}.
  \end{equation}
 \item If $|p|^l\leq r_i<|p|^{l-1}$, for all $a\in
  \bigcup_{j=0}^{p^l-1} \disf{c_i+j}{r_i}\cap F$ we have
  \begin{equation}
   \RS_1(x,(M_{z_i},\nabla_{z_i}-a))=\left(\frac{|p|^l\omega}{r_i}\right)^{\fra{p^l}}
  \end{equation}
  and for all $a\in F\setminus \bigcup_{j=0}^{p^l-1} \disf{c_i+j}{r_i}$
  \begin{equation}
   \RS_1(x,(M_{z_i},\nabla_{z_i}-a))=\cR(\delta({a-c_i})).
  \end{equation}
 \item For all $a\in
  \{a_j\}+\ZZ_p$, $(M_{a_j},\nabla_{a_j}-a)$ is solvable (all its
  spectral radii are maximal), and for all $a\in
  F\setminus \{a_j\}+\ZZ_p$, we have $\RS_1(x,(M_{a_j},\nabla_{a_j}-a))=\cR(\delta({a-a_j}))$. 
\end{enumerate}
\end{Theo}

Even though, at a first look, the variation of the spectrum might
seems relatively uncontrolled, Proposition~\ref{sec:introduction}
allows to prove a kind of controllability through a locally finite graph. Let us clarify in which meaning the spectrum can be controlled by a
graph. We consider a
differential equation $(\F,\nabla)$ over $X$, let $x$ be a point of
type (2) or (3). Let $\mathscr{C}$ be a quasi-smooth algebraic
curve such that an affinoid neighborhood of $x$ is isomorphic to an
affinoid domain of $\mathscr{C}$ and $x$ is in the interior of
$\mathscr{C}$. For $b\in T_x\mathscr{C}$, let $d_b$ as \eqref{eq:A82},
and $\nabla_{d_b}:\F(x)\to\F(x)$ be the resulting connection after
choosing $d_b$. As previously mentioned, for distinct $b, b' \in T_x\mathscr{C}$, the spectra $\Sigma_{\nabla_{d_b},F}(\Lk{\F(x)})$ and $\Sigma_{\nabla_{d_{b'}},F}(\Lk{\F(x)})$ may differ. However, we prove that this
can happen
only on a locally finite graph of $X$. Indeed, let $\cS$ be a weak
triangulation on $X$. We set
$\cR_{\cS,1}(y,(\F,\nabla))\leq\cdots\leq \cR_{\cS,n}(y,(\F,\nabla))$
to be the radii of convergence of $(\F,\nabla)$, with respect to
$\cS$, at $y$. It just mean that they are the normalized radii of
convergence of $(\F,\nabla)$ at $y$ on $D(y,\cS)$, where $D(y,\cS)$ is the biggest open
disk containing $y$ in $X\setminus \Gamma_{\cS}$. Let $\Gamma$ be a
subgraph of $X$. We say that
 $\Gamma$ is a {\it controlling graph of the spectrum} of $(\F,\nabla)$
 with respect $\cS$, if and only if the
 following conditions are satisfied:
 \begin{enumerate}
 \item the points of $\Gamma$ are of type (2) or (3);
 \item $\Gamma_{\cS}\subset \Gamma$;
 \item $X\setminus\Gamma$ is a disjoint union of open disks;
 \item for each $D\subset X\setminus\Gamma$, let $x\in \Gamma$ with
  $\overline{D}=D\cup \{x\}$ (the topological closure), for all $y\in
  D$ not of type (1) or (4), we have
  \begin{equation}\label{eq:35}
\Sigma_{\nabla_{d_b},F}(\Lk{\F(y)})=\bigcup_{i=1}^\nu\{x_{0,\phi_i(y)}\}+\ZZ_p,
  \end{equation}
 with
  \begin{equation}\label{eq:36}
\phi_i(y)=
   \begin{cases}
    0& \text{ if } r_{\cS,F}(y)\leq\cR_{\cS,i}(x,(\F,\nabla)),\\
    \frac{|p|^l\omega\cdot r_{\cS,F}(y)^{p^l}}{\cR_{\cS,i}(x,(\F,\nabla))^{p^l}} & \text{ if }
                         \omega^{\fra{p^{l-1}}}\cdot
                         r_{\cS,F} (y)\leq
                        \cR_{\cS,i}(x,(\F,\nabla))\leq
                         \omega^{\fra{p^{l}}}\cdot
                         r_{\cS,F} (y),
                         \text{ with }
                         l\in\NN\setminus\{0\},\\
    \frac{\omega\cdot r_{\cS,F} (y)}{\cR_{\cS,i}(x,(\F,\nabla))} &
    \cR_{\cS,i}(x,(\F,\nabla))\leq \omega \cdot r_{\cS,F} (y).\\
   \end{cases}
  \end{equation}
 \end{enumerate}
Where $r_{\cS,F}(y)$ is defined as follows, if $y\in \Gamma_{\cS}$ we
have $r_{\cS,F}(y)=1$, and
if $y\not\in \Gamma_{\cS}$ we have $r_{\cS,F}(y)=\frac{r(y)}{R(y,\cS)}$, where $r(y)$ is the radius of the point $y$ as an
 element of $D(y,\cS)$, and $R(y,\cS)$ is the radius of $D(y,\cS)$.

This definition means in particular, on $D\subset X\setminus\Gamma$,
for all $y\in D$, the spectrum does not depends on $b\in
T_y\mathscr{C}$ nor on the étale morphism $\Psi_b$. Moreover, the spectrum
is continuous on $y$, even if $d_b\ne d_{b'}$ for $b\ne b'$. It also
means that, if we know the radii of convergence on $x$, where
$\overline{D}=D\cup \{x\}$, then we can deduce the value of the
spectrum on each point of type (2) or (3) of $D$.

Since formulas \eqref{eq:35} and \eqref{eq:36}, and Theorem~\ref{sec:introduction-1} imply that the radii
of convergence $\cR_{\cS,i}(.,(\F,\nabla))$ are constant on
$D\subset X\setminus \Gamma$, the controlling graph
$\Gamma(\cR_{\cS,\F})$ of the radii of
convergence $\cR_{\cS,i}(. ,(\F,\nabla))$, with respect to $\cS$,
is the smallest controlling graph of the spectrum with respect to
$\cS$. We precise that the graph $\Gamma(\cR_{\cS,\F})$ satisfies the
following, for all $x\in \Gamma(\cR_{\cS,\F})$, $x$ does not admits a
neighborhood $D$ in $X$ such that:
\begin{itemize}
\item $D$ is a disk;
\item $\cR_{\cS,i}(.,(\F,\nabla))$ are all constant on $D$;
\item $D\cap \Gamma_{\cS}=\emptyset$. 
\end{itemize}

By the work of
J. Poineau and A. Pulita \cite{and,np2}, the graph $\Gamma(\cR_{\cS,\F})$ is locally
finite. Therefore the spectrum admits a locally finite controlling
graph with respect to $\cS$. In this paper we study the spectrum only on
points of type (2) and (3), which is not a problem since $\Gamma(\cR_{\cS,\F})$ does not contain points
of type (1) or (4) (see \cite[Théorème
IV]{Lut37}, \cite[Theorem~4.5.15]{ked13}).

As we mentioned before the value of the spectrum depends on the  
choice of $b$, in particular on $b\in T_x\Gamma\cup (T_x\cC\setminus
T_xX)$. Therefore, decomposition~\eqref{eq:92} also depends on the
choice of $b$. Hence, by considering the decomposition with respect to
all these directions, we may obtain a finer decomposition. We illustrate
this by the following simple example.

\begin{exam}\label{sec:introduction-3}
 Suppose that $X=\disf{0}{R}\setminus (\diso{1}{1}\cup
 \diso{c}{r}\cup\diso{0}{1})$, with $R>1$, $r<1$, and $|c|=1$. We consider
 $(M,\nabla)$ the differential module over $(\h{x_{0,1}},\dT)$ defined
 as follows:
 \begin{equation}
  \nabla
  \begin{pmatrix}
   f_1\\f_2\\f_3
  \end{pmatrix}=
  \begin{pmatrix}
   \dT f_1\\ \dT f_2\\ \dT f_3
  \end{pmatrix}+
  \begin{pmatrix}
   \frac{a_1}{T}&1&0\\
   0&\frac{a_2}{T-1}& 1\\
   0&0&\frac{a_3}{T-c}\\
  \end{pmatrix}
 \end{equation}
 with $a_1$, $a_2$, $a_3\in F$, $|a_1|=|a_2|=|a_3|>1$, and $a_i- a_j\not\in \ZZ_p$ for $i\ne j$. We set $b_0=(0,x_{0,1})$,
 $b_1=(1,x_{0,1})$, and $b_c=(c,x_{0,1})$. Then we have
 $\Sigma_{\nabla_{d_{b_0}},F}(\Lk{M})=\{x_{0,|a_1|}\}\cup \{a_1\}+\ZZ_p$,
 $\Sigma_{\nabla_{d_{b_1}},F}(\Lk{M})=\{x_{0,|a_1|}\}\cup\{a_2\}+\ZZ_p$, and
 $\Sigma_{\nabla_{d_{b_c}},F}(\Lk{M})=\{x_{0,|a_1|}\}\cup\{a_3\}+\ZZ_p$. In this
 situation $(M,\nabla)$ is a pure differential module, i.e. all the
 radii of convergence are equal, which means that we cannot decompose
 the differential module with respect to the radii of
 convergence. If we decompose with respect to one direction then
 $(M,\nabla)$ decomposes into a direct sum of two differential
 module. However, form the value $\Sigma_{\nabla_{b_0},F}(\Lk{M})$,
 $\Sigma_{\nabla_{b_1},F}(\Lk{M})$ and $\Sigma_{\nabla_{b_c},F}(\Lk{M})$, we can conclude
 that $(M,\nabla)$ decomposes into a direct sum of differential
 modules of rank 1. For the computation of the spectrum we use \cite[Theorem~4.23]{Azz24}.
\end{exam}
We precise that the decomposition~\eqref{eq:92} depends also on the choice
of $\Psi_b$. Fortunately, this dependency serves to refine the
decomposition further. In a future work, we plan to compute the spectrum with respect to all
derivations of the form $d_{b,l}=N\cdot f_b^l\cdot\Psi_b^*(\dT)$ with
$l\in \ZZ$ and $b$ is a missing branch or lying in the controlling graph, and
proving a decomposition theorem as in
Theorem~\ref{sec:introduction-1}. We suspect that the resulting decomposition will not
depend on $\Psi_b$.

Recall that there exists another kind of decomposition called {\it Refined decomposition}, we refer to
Section~\ref{sec:relat-with-refin-1} or \cite[Theorem~10.6.7]{Ked} for
a precise definition, this decomposition provides a decomposition finer
than the one provided by the spectral radii of convergence. By comparing the
refined decomposition with the decomposition with respect to the spectrum,
we found out that the refined decomposition is neither finer nor coarser than the one provided
by~\eqref{eq:92} with respect to all branches $b\in T_x\Gamma\cup (T_x\cC\setminus T_xX)$. 

We point out that Dwork-Robba's decomposition, which provides a local
decomposition with respect to the spectral radii of convergence on a
point of type (2), (3) or (4) of a quasi-smooth curve (cf. {\cite[First thm. of
 Section~4]{DR77},\cite[Theorem~4.1.1]{np3}}), it generalizes as
follows. Let $\cS$ be a weak triangulation of a quasi-smooth curve
$X$. Let $x\in X_{[2,3]}$ and $\mathscr{C}$ be a quasi-smooth algebraic
curve such that an affinoid neighborhood of $x$ is isomorphic to an
affinoid domain of $\mathscr{C}$ and $x$ is in the interior of
$\mathscr{C}$. Let $b$, $b'\in T_x\mathscr{C}$, then
there exists $g\in \Hx$ such that $d_{b'}=gd_b$. For $(M,\nabla_b)$ a differential module over
$(\Hx,d_b)$, we set $\nabla_{b'}=g\cdot\nabla_{b}$. Let $(\F,\nabla)$ be a
differential equation over $X$, let $\Gamma(\cR_{\cS,\F})$
be the controlling graph of radii of convergence of $(\F,\nabla)$, with respect to
$\cS$. Then we have $T_x\Gamma(\cR_{\cS,\F})\cup (T_x\mathscr{C}\setminus T_x
X)=\{b_1,\cdots,b_{N_x}\}$, we fix $b\in T_xX\setminus T_x\Gamma(\cR_{\cS,\F})$.

\begin{Theo}[{Theorem~\ref{sec:dwork-robb-decomp-1}}]\label{sec:introduction-5}
Let $(M_x,\nabla_x)$ be a differential module over $(\cO_{X,x},d_b)$, with $(M_x,\nabla_x)=(\F_x,\nabla_{d_b})$. For each $b_i\in \{b_1,\cdots,b_{N_x}\}$, let $\Sigma_{b_i}$ be a subset of $\AF$ such that for all $z, z'\in \Sigma_{b_i}$ with $z\ne z'$, we have $\{z\}+\ZZ_p\cap \{z'\}+\ZZ_p=\emptyset$ and $\Sigma_{(\nabla_x){b_i},F}(\Lk{M_x\ct_{\cO_{X,x}}\Hx})=\Sigma_{b_i}+\ZZ_p$. We set $\Sigma:=\prod_{i=1}^{N_x}\Sigma_{b_i}$. For $v\in \Sigma$, we have $v=(v_1,\cdots, v_{N_x})$, where $v_i \in \Sigma_{b_i}$ for $i=1,\cdots,N_x$. Then
\begin{equation}
(M_x,\nabla_x)=\bigoplus_{v\in\Sigma}(M_{x,v},\nabla_{x,v}),
\end{equation}
such that if $M_{x,v}\ne 0$, then each sub-quotient of $(M_{x,v}\ct\Hx,(\nabla_{x,v})_{b_i})$ has spectrum equal to $\{v_i\}+\ZZ_p$.
\end{Theo}

In this paper, we also study another kind of variation of the
spectrum, which consists in studying the continuity of the map
\begin{equation}
 \label{eq:97}
 \Fonction{\Sigma_{.,F}(E):E}{\cK(\AF)}{f}{\Sigma_{f,F}(E),}
\end{equation}
where $E$ is a $F$-Banach algebra, and $\Sigma_{f,F}(E)$ is the
spectrum of $f$ as an element of $E$, in the sense of Berkovich \cite[Chapter~7]{Ber}. We observe by an example
(cf. Example~\ref{sec:cont-spectr-an-1}) that in general this map is not
continuous. However, we prove that if $\Sigma_{f,F}(E)$
is totally discontinuous, then $\Sigma_{.,F}(E)$ is continuous at $f$.

Since the above results prove that the spectrum of every differential equation is totally disconnected, we are able to derive from \eqref{eq:97} the continuity of the variation of the spectrum $\Sigma_{\nabla,F}(\Lk{M})$ with respect to the connection. In other words, when a connection on $M$ is approximated by another, then so do the respective spectra. More specifically, the main result can be succinctly stated as follows:

\begin{Theo}[Theorem~\ref{sec:link-betw-spectr-2}, see Remark~\ref{sec:link-betw-spectr-5}]\label{sec:introduction-2}
Let $x\in \Xh$ be a point of type (2) or (3), $(M,\nabla)$ be a
 differential module over $(\Hx,d_b)$, and $\{(M,\nabla_l)\}_{l\in
  \NN}$ be a family of differential modules over $(\Hx,d_b)$ such
 that $\nabla_l\to \nabla$ in $\Lk{M}$. Let
 $z_1,\cdots,z_\nu$ be points of type (2) or (3) and $a_1,\cdots,
 a_\mu\in F$ be such that
 \begin{equation}
  \Sigma_{\nabla,F}(\Lk{M})=\{z_1,\cdots,z_\nu,a_1,\cdots, a_\mu\}+\ZZ_p,
 \end{equation}
 with $\{a_i\}+\ZZ_p\cap \{a_j\}+\ZZ_p=\emptyset$ (resp. $\{z_i\}+\ZZ_p\cap
 \{z_j\}+\ZZ_p=\emptyset$) if $i\ne j$, and $(\nu,\mu)\ne (0,0)$. Then there exists $l_0\in \NN$
 such that for all $l\geq l_0$ we have
 \begin{equation}
  (M,\nabla_l)=\bigoplus_{i=1}^\nu(M_{l,z_i},\nabla_{l,z_i})\oplus \bigoplus_{k=1}^\mu(M_{l,a_k},\nabla_{l,a_k}),
 \end{equation}
 where
 \begin{enumerate}
 \item $\dim_{\Hx} M_{l,z_i}=\dim_{\Hx} M_{z_i}$ and $\dim_{\Hx} M_{l,a_k}=\dim_{\Hx}
  M_{a_k}$,
 \item $\Sigma_{\nabla_{l,z_i}}=\{z_i\}+\ZZ_p$,
 \item $\forall \epsilon>0$, $\exists l_\epsilon\geq l_0$, $\forall
  l\geq l_\epsilon\implies
  \Sigma_{\nabla_{l,a_k}}\subset\bigcup_{j\in \NN}\diso{a_k+j}{\epsilon}$. 
 \end{enumerate}
\end{Theo}

From Theorems~\ref{sec:introduction-1} and \ref{sec:introduction-2},
we deduce easily that for each $i$ we have $\RS_i(x,(M,\nabla_l))\to
\RS_i(x,(M,\nabla))$. Moreover, if $\RS_i(x,(M,\nabla))<1$, then there
exists $l_0\in\NN$ such that for all $l\geq l_0$ we have
$\RS_i(x,(M,\nabla_l))=\RS_i(x,(M,\nabla))$.
This is a quite remarkable fact, because it can be used to prove general results about the radii of convergence by approximating the connection. For instance, in the affine line, we may approximate a connection by a sequence of connections whose coefficients are rational fractions.

The paper is organized as follows. In Section~\ref{sec:preliminaries}, we set
notations and recall all
the necessary definitions and results. In
Section~\ref{sec:cont-spectr-an}, we study the continuity of the
spectrum of some ultrametric operators. Section
\ref{sec:variation-spectrum-f} is divided into four parts. In the
first part we prove Proposition~\ref{sec:introduction}. In the second
part, we prove Theorem~\ref{sec:introduction-2} in the
case of the analytic line. In the third part we use
Theorem~\ref{sec:introduction-2} and Proposition~\ref{sec:introduction} to prove
Theorem~\ref{sec:introduction-1}. In the last part, we explain how the controlling graph of the
radii of convergence, with respect to a weak triangulation, controls
the spectrum. In the last section, on the one hand, we prove Theorem~\ref{sec:introduction-5} the spectral version of Dwork-Robba decomposition, which
provides a finer decomposition than the original one. On the other hand, we explain
in more detail why the refined decomposition is neither finer nor coarser then the one provided by the spectrum.  

{\bf Acknowledgments.} We express our gratitude to Andrea Pulita and Jérôme Poineau for answering our
various questions, for their suggestions and helpful discussions. We
also thank Antoine Ducros, Francesco Baldassarri and Hansheng Diao for
their suggestion to improve the writing. We are also grateful to
Sébastien Palcoux and Andrew Best for correcting our writing.\\
The author was supported by NSFC Grant \#12250410239.
\tableofcontents

\section{Preliminaries}\label{sec:preliminaries}
Throughout this paper all the rings are supposed to have a unit element. We denote by
$\RR$ the field of real numbers, by $\ZZ$ the ring of integers and by
$\NN$ the set of natural numbers containing $0$. We set $\R+$ for $\{r\in \RR;\;
r\geq 0\}$ and $\R+^{\times}$ for $\R+\setminus \{0\}$.

We fix $(F,|.|)$ to be an ultrametric complete field of characteristic
0. Let $|F|$ be $\{|a|;\; a\in F \}$. Let $E(F)$ be the category whose
objects are complete ultrametric field $(\Omega,|.|_\Omega)$, with $\Omega$ is a field extension
of $F$ and $|.|_\Omega$ extends $|.|$, and whose arrows are isometric
ring morphisms. For $(\Omega,|.|_\Omega)\in E(F)$, let
$\Omega^{alg}$ be an algebraic closure of $\Omega$, the absolute value
extends uniquely to $\Omega^{alg}$. We denote by $\wac$ the completion
of $\Omega^{alg}$ with respect $|.|_\Omega$. 

\subsection{Linear differential equations over a quasi-smooth curve}
In this paper we consider $F$-analytic spaces in the sense of
Berkovich (see \cite{Ber}). We denote by $\AF$ (resp. $\Pf$) the
analytic affine (resp. projective)
line over the ground field, with coordinate $T$. Let $(X, \cO_X)$ be an analytic
space, for any $x\in X$, the residue field of the local ring $\cO_{X,x}$
is naturally valued and we denote by $\Hx$ its completion.

Let $\Omega\in E(F)$ and $c\in\Omega$. Let $I$ be a non-empty interval
of $\R+$ and $r\in \R+^{\times}$ we set

\begin{equation}
\DI{\Omega}{I}(c):=\{x\in \A{\Omega}; \; |T(x)-c|\in I\} 
\end{equation}
We set $\disco{\Omega}{c}{r}:=\DI{\Omega}{[0,r)}(c)$ (resp. $\discf{\Omega}{c}{r}:=\DI{\Omega}{[0,r]}(c)$) and call it
open (resp. closed) disc. Denote by $x_{c,r}$ the unique point in the Shilov boundary of
$\discf{\Omega}{c}{r}$.\\
If $0\not\in I$ we say that $\DI{\Omega}{I}(c)$ is an annulus of type $I$, if moreover
$I$ is open
(resp. closed, resp. semi-closed) we say that $\DI{\Omega}{I}(c)$ is open
(resp. closed, resp. semi-closed). We set
$\Couo{\Omega}{c}{r_1}{r_2}:=\DI{\Omega}{(r_1,r_2)}(c)$ and $\Couf{\Omega}{c}{r_1}{r_2}:=\DI{\Omega}{[r_1,r_2]}(c)$. We drop the subscript $\Omega$ when no confusion is possible.
 \begin{Defi}\label{sec:type-points-a_k}
  Let $x\in \AF$ and $y\in\pik{\widehat{F^{alg}}}\-1(x)$. We define
  the radius of $x$ to be the value:
  \begin{equation}
   \label{eq:60}
  r_F(x)=\inf_{c\in F^{alg}}|T(y)-c|. 
 \end{equation}

  It does not depend on the choice of $y$. We drop $F$ when no
  confusion is possible. 
 \end{Defi}

 \begin{rem}
  Assume that $F$ is algebraically closed. For a point
  $x_{c,r}$ of type (1), (2) or (3), we have $r_F(x_{c,r})=r$.
 \end{rem}

Assume that $F$ is algebraically closed. Let $c\in F$. The following
map
\begin{equation}\Fonction{[0,+\infty)}{\AF}{r}{x_{c,r}}\end{equation}
induces a homeomorphism between $[0,+\infty)$ and its image.

\begin{nota}\label{sec:type-points-a_k-1}
 We denote by $[x_{c,r},\infty)$ (resp. $(x_{c,r},\infty)$) the image of
 $[r,+\infty)$ (resp. $(r,+\infty)$), by $[x_{c,r},x_{c,r'}]$
 (resp. $(x_{c,r},x_{c,r'}]$, $[x_{c,r},x_{c,r'})$, $(x_{c,r},x_{c,r'})$) the image
 of $[r,r']$ (resp. $(r,r']$, $[r,r')$, $(r,r')$).
\end{nota}

\begin{Defi}
 Let $\dI(c)$ be an annulus over $F$. We set
 \begin{equation}
  \label{eq:3}
  \text{Mod}(\dI(c))=\frac{\inf I}{\sup I}\quad \text{and} \quad
  \ell(I)=\log \text{Mod}(\dI(c))
 \end{equation}
\end{Defi}

\begin{Pro}\label{sec:line-diff-equat}
 Let $I$ and $J$ be bounded intervals in $\R+^{\times}$ and $c$, $c'\in
 F$. 
 \begin{enumerate}
 \item If $0\in I\cap J$, then $\di{I}(c)\simeq \di{J}(c')$ if and
  only if $I\simeq J$ and $\frac{\sup I}{\sup J}\in |F^{\times}|$.
  \item If $0<\inf I$, $\inf J$, then $\di{I}(c)\simeq \di{J}(c')$ if and
  only if $I\simeq J$, $\frac{\sup I}{\sup J}\in |F^{\times}|$
  and $\rm{Mod}(\di{I}(c))= \rm{Mod}(\di{J}(c'))$. 
 \end{enumerate}
\end{Pro}

\begin{rem}
 This means for some $\Omega\in E(F)$, the annulus $\DI{\Omega}{I}(c)$
 is isomorphic to $\DI{\Omega}{J}(c)$, with $I\simeq J$, $\inf
 J=\text{Mod}(\dI(c))$ and $\sup J=1$.
\end{rem}

\begin{Defi}
 We say that a non-empty $F$-analytic space $X$ is a {\it virtual
  disk} (resp. {\it virtual annulus}) if $X\ct_F\kac$ is isomorphic
 to a union of disks (resp. annuli whose orientation is preserved by
 $\rm{Gal}(F^{alg}/F)$ over $\kac$ (see \cite[3.6.32 and
 3.6.35]{Duc})).
\end{Defi}
\subsubsection{The shape of a quasi-smooth curve}\label{sec:shape-quasi-smooth-1}
\begin{Defi}
 We say that an $F$-analytic space $(X,\cO_X)$ is a quasi-smooth curve if
 its respective sheaf of Kähler differentials forms $\Omega_{X/F}$ is locally
 free of rank 1. 
\end{Defi}

\begin{conv}
 In the whole section, $X$ is a connected quasi-smooth $F$-analytic curve. We set
 $X^{\rm{hyp}}$ (resp. $X_{[i]}$, resp. $X_{[2,3]}$) to be the set of
 points of $X$ not of type (1) (resp. of type (i), resp. of type (2)
 or (3)).
 
\end{conv}
\begin{Defi}
A subset $\cS$ of $X$ is a {\it weak
  triangulation} (resp. {\it triangulation}) of $X$ if
 \begin{enumerate}
 \item $\cS$ is locally finite and only contains points of type (2) and
  (3);
 \item any component of $X\setminus \cS$ is (resp. relatively
  compact) isomorphic to virtual open disk or annulus.
 \end{enumerate}
\end{Defi}

\begin{Defi}
 Let $X$ be a quasi-smooth curve and $\cS$ a weak triangulation. The
 skeleton $\Gamma_S$ of $X$, with respect to $\cS$, is the subset of $X$ containing $\cS$ and
 every point of $X\setminus \cS$ that does not have a virtual open disk as a neighborhood.
\end{Defi}

\begin{exams}~\\
 \begin{itemize}
 \item for $\diso{c}{r}$ (resp. $\disf{c}{r}$) the set $\cS=\emptyset$
  (resp. $\cS=\{x_{c,r}\}$) is a minimal weak triangulation with
  $\Gamma_{\cS}=\emptyset$ (resp. $\Gamma_{\cS}=\{x_{c,r}\}$).
 \item for $\couo{c}{r_1}{r_2}$ (resp. $\couf{c}{r_1}{r_2}$) the set $\cS=\emptyset$
  (resp. $\cS=\{x_{c,r_1},x_{c,r_2}\}$) is a minimal weak
  triangulation with
  $\Gamma_\cS=(x_{c,r_1},x_{c,r_2})$ (resp. $\Gamma_\cS=[x_{c,r_1},x_{c,r_2}]$) 
 \end{itemize}
\end{exams}
The skeleton of a disk (resp. annulus) will be the skeleton with
respect to the minimal triangulation.
For a closed annulus $C$ of type $I_C=[r_1,r_2]$, we define the
homeomorphism map:
\begin{equation}
 \label{eq:9}
 \Fonction{\imath_{I_C}:\Gamma_C}{[\text{Mod}(C),1]}{x}{\fra{r_2}.r_F(x).}
\end{equation}
From Proposition~\ref{sec:line-diff-equat}, we can deduce that if
$\phi: C\to C'$ is an isomorphism, then
$\imath_{I_{C'}}\circ\phi_{|_{\Gamma_C}}=\imath_{I_C}$. Hence, we can defined a
canonical map $\imath_C:\Gamma_C\to [\text{Mod}(C),1]$.
\begin{rem}
 If $F$ is algebraically closed, the skeleton $\Gamma_\cS$ can be seen as the union of $\cS$ and the
 skeleta of the connected
 components $X\setminus \cS$ (which are open disks and annuli).
\end{rem}

\begin{rem}
 The skeleton $\Gamma_\cS$ does not contain a point of type (1) and (4),
 because such a point admits open disks as neighborhood basis. 
\end{rem}

\begin{Theo}[{\cite[Theorem 5.1.14]{Duc}}]
 Any quasi-smooth $F$-analytic smooth curve admits a triangulation,
 and hence a weak triangulation.
\end{Theo}

The existence of a weak triangulation on $X$ induces the locally
uniquely arcwise connectedness of $X$, in other words $X$ is locally
simply connected graph. Hence, for any point $x\in X $ we can choose a
neighborhood $U$ such that for any point $y\in U$ the segment $[x,y]$
is well defined. We set
\begin{equation}
 \label{eq:2}
 T_xX:=\{ (y,U);\; y\in U\, \text{and } U\text{ is simply a connected
  neighborhood of }x\}/\sim,
\end{equation}
where $\sim$ is defined as follows
\begin{equation}
 \label{eq:8}
 (y,U)\sim (z,V) \; \text{if and only if } [x,y]\subset [x,z] \text{ or } [x,z]\subset [x,y].
\end{equation}
The set $T_xX$ represents the branches out of $x$. A representative of a
branch $b\in T_xX$ is called a {\it section} of $b$. For all
simply connected neighborhood $U$ of $x$, and each $b$ we can
find $y\in U$ such that $(y,U)$ is a section. Moreover, there exists a
bijection between $T_xX$ and the set of connected component of
$U\setminus \{x\}$. Hence, by saying that we can choose an annulus (
or its skeleton) as section of
$b$, we mean just find $U$ such that the corresponding connected component
of $b$ is an annulus. For a subset $\Gamma$ of $X$, and $x\in\Gamma$ we set
$T_x\Gamma$ to be the set of branches out $x$ admitting a section in
$\Gamma$ (i.e. $(y,U)$ such that $U\cap\Gamma$ is connected and $y\in U\cap \Gamma$).

\begin{Defi}
 We say that a subset $\Gamma$ is a {\it finite} (resp. {\it locally
  finite}) subgraph of $X$, if there exists a finite (resp. locally
  finite) family $\mscr{V}$ of simply connected affinoid
  domain of $X$ that covers $\Gamma$, and such that, for every $V\in \mscr{V}$, 
  the set $\Gamma\cap V$ is the convex hull of a finite number of points.
 \end{Defi}
 
 \begin{rem}
  If $\Gamma$ is a locally finite graph, then for all $x\in \Gamma$
  the set $T_x\Gamma$ is finite.
 \end{rem}
 
 \begin{Defi}
  Let $C$ be a closed annulus. We say that $f:\Gamma_C\to
  \R+^{\times}$ is {\it log-affine} (resp. {\it log-convex},
  resp. {\it log-concave}), if the map ${\log\circ f\circ
  \imath_C\-1\circ \exp:[\ell(\Gamma_C),0]\to \R+}$ is affine
 (resp. convex, resp. concave). 
\end{Defi}

\begin{Defi}
 Let $\Gamma$ be a locally finite subgraph of $X$. A map $f:\Gamma\to
 \R+^{\times}$ is said to be {\it piecewise log-affine} (resp. {\it log-convex},
  resp. {\it log-concave}) if $\Gamma$ may be
 covered by a locally finite family $\mscr{J}$ of closed intervals
 such that, for every $J\in\mscr{J}$ and each skeleton $I$ of a closed annulus lying in the
 topological interior of $J$, the restriction of $f$ to $I$ is
 log-affine (resp. log-convex, resp. log-concave).  
\end{Defi}
 
\begin{nota}
 Let $\Omega\in E(F)$. We set
 $X_\Omega=X\ct_F\Omega$. For every $\Omega'\in E(\Omega)$, we denote by
 $\piro{\Omega'}{\Omega}:X_{\Omega'}\to X_{\Omega}$ the
 canonical projection morphism. 
\end{nota}
\begin{Defi}\label{sec:shape-quasi-smooth-3}
 A point $x\in X$ is said to be {\it universal} if for every
 $\Omega\in E(F)$, the tensor norm on the algebra $\Hx\ct_F\Omega$ is
 multiplicative. In this case this norm defines a point in
 $\pik{\Omega}\-1(x)$ in $X_\Omega$ denoted by $x_\Omega$. 
\end{Defi}
\begin{Pro}[{\cite[Corollaire~3.14]{poi}}]
 If $F$ is algebraically closed, every point $x\in X$ is universal.
\end{Pro}

Let $\cS$ be a weak triangulation of $X$, we set $\cS_{\kac}:=\pik{\kac}\-1(\cS)$. It is easy to see that
$\cS_{\kac}$ is a weak triangulation and $\Gamma_{\cS_{\kac}}=\pik{\kac}\-1(\Gamma_\cS)$. Now
 let $\Omega\in E(F)$. We set
 \begin{equation*}
  \begin{array}{cc}
   \cS_{\wac}=\{x_{\wac}\in X_{\wac};\; x\in \cS_{\kac}\};&
   \Gamma_{\cS_{\wac}}=\{x_{\wac}\in X_{\wac};\; x\in
   \Gamma_{\cS_{\kac}}\}\\
   \cS_\Omega=\piro{\wac}{\Omega}( \cS_{\wac});&\Gamma_{ \cS_\Omega}=\piro{\wac}{\Omega}( \Gamma_{\cS_{\wac}}).\\
  \end{array}
 \end{equation*}
 
 \begin{Theo}[{\cite[Theorem~2.15]{np2}}]
  Let $x$ be a point of $X_{\kac}$ and let $\Omega\in E(\kac)$.
  \begin{enumerate}
  \item If $x$ is of type (i), where $i\in\{1,2\}$, then so is $x_{\Omega}$. If
   $x$ is of type (j), where $j\in\{3,4\}$, then $x_\Omega$ is of
   type (j) or (2).
  \item The fiber $\piro{\Omega}{\kac}\-1(x)$ is connected and if
   $\piro{\Omega}{\kac}\-1(x)\ne \{x_\Omega\}$ the connected
   components of $\piro{\Omega}{\kac}\-1(x)\setminus \{x_\Omega\}$ are open
   disks with boundary $\{x_\Omega\}$. Moreover they are open in $X_{\Omega}$.
  \end{enumerate}
 \end{Theo}
 \begin{cor}[{\cite[Corollary~2.16]{np2}}]
  Let $\Omega\in E(F)$ and $S$ be a weak triangulation of $X$. Then
  $S_\Omega$ is a weak triangulation of $X_\Omega$ whose skeleton is $\Gamma_{S_\Omega}$.
 \end{cor}

 \begin{Defi}\label{sec:shape-quasi-smooth}
  Let $x\in X$. Let $\Omega\in E(\Hx)$ be algebraically closed and
  $t_x\in \pik{\Omega}\-1(x)\cap X_{\Omega}(\Omega)$. We set $D_\Omega(x,S)$ to be the
  maximal open disk containing $t_x$ in
  $X\setminus\Gamma_{S_\Omega}$, and $D(x)$ to be the connected
  component of $\piro{\Omega}{\kac}\-1(x)\setminus \{x_\Omega\}$
  containing $t_x$.
\end{Defi}
\begin{Defi}\label{sec:shape-quasi-smooth-2}
 Let $x\in X$. Let $\Omega\in E(\Hx)$ be algebraically closed and
  $t_x\in \pik{\Omega}\-1(x)\cap X_\Omega(\Omega)$. We consider $T_x: D_\Omega(x,S)\overset{\sim}{\to}
\disco{\Omega}{0}{R}$, we set $r_{{T_x},F}(x)=0$ if $x$ is of type (1), and
equal to the radius of the disk $T_x(D(x))$ otherwise. We set $r_{\cS,F}(x):=\frac{r_{{T_x},F}(x)}{R}$. 
\end{Defi}

\begin{rem}
 The value $r_{\cS,F}(x)$ depends only on the skeleton $\Gamma_\cS$, and
 if $x\in\Gamma_\cS$ then $r_{\cS,F}(x)=1$.
\end{rem}

\subsubsection{The radii of convergence of a linear differential equation
over a quasi-smooth curve}\label{sec:radii-conv-line-3} By a linear differential equation over
$(X,\cO_X)$, we mean a locally free $\cO_X$-module $\F$ of finite type, of rank $n$, with an
integrable connection $\nabla$, i.e. a $F$-linear map
$\nabla:\F(U)\to \Omega_{X/F}(U)\ot\F(U)$ satisfying the Leibniz
rule $\nabla(f. e)=df\ot e+f.\nabla(e)$. Let $x\in X$ and $\Omega\in E(\Hx)$
algebraically closed, we
set $(\F_\Omega,\nabla_\Omega):=\pik{\Omega}^*(\F,\nabla)$ and
$\F_\Omega^{\nabla_\Omega}:=\Ker \nabla_\Omega$. We consider          
$T_x: D_\Omega(x,\cS)\overset{\sim}{\to}
\disco{\Omega}{0}{R}$, we defined
\begin{equation}
 \label{eq:4}
 \cR_{\cS,i}(x,(\F,\nabla)):=\frac{\sup\{r\in[0,R]; \dim_F(
\F_\Omega^{\nabla_\Omega}(\disco{\Omega}{t_x}{r}))\geq n-i+1\}}{R} 
\end{equation}
to be the $i$th radius of convergence at $x$ of $(\F,\nabla)$, 
\begin{equation}
 \label{eq:5}
 \cR_\cS(x,(\F,\nabla))=(\cR_{\cS,1}(x,(\F,\nabla)),\cdots,
 \cR_{\cS,n}(x,(\F,\nabla)))\in (0,1]^r
\end{equation}
to be the multiradius of convergence at $x$ of $(\F,\nabla)$. We may drop $\nabla$ from the notation if it is obvious from the
context.

\begin{rem}
 The values of $\cR_{\cS,i}(x,(\F,\nabla))$ do not depend on $\Omega$ and $T_x$. Hence for the simplicity, from now on we choose $\Omega$
 such that we can choose the coordinate $T_x: D_\Omega(x,\cS)\overset{\sim}{\to}
 \disco{\Omega}{0}{1}$.
\end{rem}

We set
\begin{equation}
 \label{eq:19}
 \omega_{\cS,i}(x,(\F,\nabla))=\F_\Omega^{\nabla_\Omega}(\disco{\Omega}{t_x}{\cR_{\cS,i}(x,(\F,\nabla))},
\end{equation}
then we have
\begin{equation}
 \label{eq:20}
 0\ne\omega_{\cS,n}(x,(\F,\nabla))\subseteq\cdots\subseteq \omega_{\cS,1}(x,(\F,\nabla)).
\end{equation}
We may drop $\nabla$ from the notation if it is obvious from the
context.

\begin{Defi}\label{sec:radii-conv-line-2}
 We say that $i$ separates the radii of $(\F,\nabla)$ at $x$ if
 either $i=1$ or if one of the following equivalent conditions holds:
 \begin{enumerate}
 \item $\cR_{\cS,i-1}(x,\F)<\cR_{\cS,i}(x,\F)$;
 \item $\omega_{\cS,i}(x,\F)\subsetneq \omega_{\cS,i-1}(x,\F)$.
 \end{enumerate}
 We say that $i$ separates the radii of $\F$ if it separates the
 radii of $\F$ for all $x\in X$.
\end{Defi}

Suppose now that $x$ is not of type (1), we define
the $i$th spectral radius of convergence $x$ of $(\F,\nabla)$ to be
\begin{equation}
 \label{eq:6}
 \RS_{i}(x,(\F,\nabla))=\frac{\min(\cR_{\cS,i}(x,(\F,\nabla)),r_{\cS,F}(x))}{r_{\cS,F}(x)}
\end{equation}
and
the spectral multiradius to be
\begin{equation}
 \label{eq:7}
\RS(x,(\F,\nabla))=(\RS_{1}(x,(\F,\nabla)),\cdots,
 \RS_{n}(x,(\F,\nabla)))\in (0,1]^r. 
\end{equation}
\begin{rem}
 For all $i$, the value $\RS_{i}(x,(\F,\nabla))$ is nothing but
 $\cR_{\emptyset,i}(t_x,({\F_\Omega}_{|_{D(x)}},\nabla_\Omega))$. That
 is why the spectral radii are intrinsic to the point $x$.
\end{rem}

\begin{Theo}[{\cite[Theorem~3.6]{np2}}]\label{sec:radii-conv-line}
 The multiradius $\cR_\cS$ satisfies the following properties:
 \begin{enumerate}
 \item it is continuous;
 \item its restriction to any locally finite graph $\Gamma$ is
  piecewise log-affine;
 \item on any interval $J$, the log-slopes of its restriction are
  rational numbers of the from $\frac{m}{j}$ with $m\in \ZZ$ and
  $j\in \{1,\cdots, n\}$.
 \item there exists a locally finite subgraph $\Gamma$ of $X$ and a
  continuous retraction $r: X\to \Gamma$ such that the map $\cR_\cS(.,(\F,\nabla))$
  factorizes by $r$. 
 \end{enumerate}
\end{Theo}

\begin{rem}\label{sec:radii-conv-line-1}
 Note that the map $r_{\cS,F}$ on any finite graph of $X^{\rm{hyp}}$ is continuous, piecewise log-affine with
 slopes equal to $1$ on any interval in $X\setminus\Gamma_\cS$ and $0$
 on any interval in $\Gamma_\cS$. Therefore, on every finite graph of
 $X^{\rm{hyp}}$, $\RS(x,(\F,\nabla))$
 satisfies Properties 1), 2) and 3) of Theorem~\ref{sec:radii-conv-line}.
\end{rem}

\begin{Defi}\label{sec:radii-conv-line-4}
The {\it controlling graph} $\Gamma(\cR_{\cS},\F)$ of the radii of convergence $\cR_{\cS}(.,(\F,\nabla))$ 
with respect to $\cS$, is a graph that satisfies the
following, for all $x\in \Gamma(\cR_{\cS,\F})$, $x$ does not admit a
neighborhood in $X$ such that:
\begin{itemize}
\item $D$ is a disk;
\item $\cR_{\cS,i}(.,(\F,\nabla))$ are all constant on $D$;
\item $D\cap \Gamma_{\cS}=\emptyset$. 
\end{itemize}
\end{Defi}

\begin{cor}[{\cite{and,np2}}]
 The graph $\Gamma(\cR_{\cS},\F)$ is locally finite.
\end{cor}

\subsubsection{Spectral radii of a differential module over $\Hx$} 

Let $X$ be a quasi-smooth analytic curve. Let us give now the definition of the spectral radii of a differential
module $(M,\nabla)$ over $(\Hx,d)$, with $x\in\Xh$ and $d$ a nonzero bounded $F$-linear
derivation defined on a neighborhood of $x$. We consider first a point $x\in\AF$ not of type (1) and
a differential module $(M,\nabla)$ over $(\Hx,g\dT)$ of rank $n$, with
$g\in\Hx^\times$.           
 Any differential module
 $(M,\nabla)$ over $(\Hx,g\dT)$, with $g\in \Hx^\times$, can be seen as a
 differential equation $(\F_\Omega,\nabla_\Omega)$ over $D(x)$, free as $\cO(D(x))$-module. Indeed we have the
 embedding
 \begin{equation}
  \label{eq:22}
  \Fonction{\xi_x:\Hx}{\cO(D(x))}{f(T)}{\sum_i\dT^i(f(T))\frac{(T-t_x)^i}{i!}},
 \end{equation}
 and by construction $\dT$ extends naturally to $\dT$ on $\cO(D(x))$. We set
 \begin{equation}
  \label{eq:11}
  \RS_{i}(x,(M,\nabla)):=\cR_{\emptyset,i}(t_x, (\F_\Omega,\nabla_\Omega)).
 \end{equation} 
 Consider now $x\in \Xh$. We need to construct an embedding
 as \eqref{eq:22}, for that we have to choose a convenient
 bounded derivation on $\Hx$. The following theorem provides a suitable étale morphism between an affinoid neighborhood
 of $x$ and an affinoid domain of $\AF$, which allows for pull-back
 the derivation $\dT$.
\begin{Theo}[{\cite[Theorem~3.12]{np2}, \cite{Duc}}]\label{sec:spectr-radii-diff}
 Let $x\in X$ be a point of type (2). Let $b_1,\cdots, b_\nu$ and $b$ be
 distinct branches out of $x$. Let $N$ be a positive integer. There
 exists an affinoid neighborhood $Z$ of $x\in X$, a quasi-smooth
 affinoid curve $Y$, an affinoid domain $W$ of $\Pf$ and a finite
 étale map $\psi: Y\to W$ such that the following hold:
 \begin{enumerate}
 \item $Z$ is isomorphic to an affinoid domain of $Y$ and $x$ lies in the
  interior of $Y$; 
 \item the degree of $\psi$ is prime to $N$;
 \item $\psi\-1(\psi(x))=\{x\}$;
 \item almost every connected component of $Y\setminus \{x\}$ is an
  open unit disk with boundary $\{x\}$;
 \item almost every connected component of $W\setminus \{\psi(x)\}$
  is an open disk with boundary $\{\psi(x)\}$;
 \item for almost every connected component $V$ of $Y\setminus
  \{x\}$, the induced morphism $V\to \psi(V)$ is an isomorphism;
 \item for every $i\in\{1,\cdots,\nu\}$, the morphism $\psi$ induces
  an isomorphism between a section of $b_i$ and a section of
  $\psi(b_i)$ and we have $\psi\-1(\psi(b_i))\subset Z$;
 \item {$\psi ^{-1}(\psi (b))=\{b\}$}. 
 \end{enumerate}
\end{Theo}

We choose such an étale morphism $\psi:Z\to W$ defined around $x$ with
$N=\max(\crk,1)$. For the following, for more details we refer to \cite[Section~3.1]{np3}. Note that if $T$ is a coordinate function on
an affinoid domain $W$ of $\AF$, then it extends naturally to a
coordinate function on
$W_\Omega$, for all $\Omega\in E(F)$. Let $T$ be a coordinate function
on $W$. By \cite[Lemma~3.1.2]{np3}, we have the isomorphism
$\psi_\Omega: D(x)\overset{\sim}{\to} D(\psi(x))$, which means that
$\psi^\#(T)$ is a
coordinate function on $D(x)$, for the simplicity we still denote it
by $T$. In this case, we have $\psi^*(\dT)=\dT$ on $D(x)$, and we set
$d:\Hx\to\Hx$ to be $d=\psi^*(\dT)$. There exists a unique morphism of
differential ring
$\xi_x:(\Hx,d)\to (\cO(D(x)),\dT)$ such that the following diagram
commute
\begin{equation}
 \label{eq:23}
 \xymatrix{\h{\psi(x)}\ar[d]_{\psi^\#}\ar[r]^{\xi_{\psi(x)}}&\cO(D(\psi(x))
  \ar[d]_{\psi^\#}\\
  \Hx\ar[r]^{\xi_x}& \cO(D(x)),\\} 
\end{equation}
and it is defined by
\begin{equation}
 \label{eq:25}
 \Fonction{\xi_x:\Hx}{\cO(D(x))}{f}{\sum_id^i(f)\frac{(T-t_x)^i}{i!}}.
\end{equation}
Therefore we can pull-back any differential module $(M,\nabla)$ over
$(\Hx,gd)$, with $g\in \Hx^\times$, to a differential equation $(\F_\Omega,\nabla_\Omega)$
over $D(x)$. We set
\begin{equation}
 \label{eq:27}
 \RS_{i}(x,(M,\nabla)):=\cR_{\emptyset,i}(t_x, (\F_\Omega,\nabla_\Omega)).
\end{equation}
Given a differential equation $(\F,\nabla)$ over $X$, and considering
$(\F_x\ot_{\cO_{X,x}}\Hx,\nabla)$ as a differential module over
$(\Hx,d)$, where $d$ is a non zero $F$- linear derivation defined on a
neighborhood $V$ of $x$ and bounded on affinoid domain, it is important to
observe that
\begin{multline}
 \label{eq:29}
 (\F_x\ct_{\cO_{X,x}}\Hx)\ct_{\Hx}\cO(D(x))=(\F(V)\ct_{\cO_X(V)}\Hx) \ct_{\Hx}\cO(D(x))\\=(\F(V)\ct_k\Omega)\ct_{\cO_{X_\Omega(V_\Omega)}}
 \cO(D(x))=\F_\Omega(D(x)).
\end{multline}
Hence, we have the equality
\begin{equation}
 \label{eq:28}
 \RS_i(x,(\F,\nabla))=\RS_i(x, (\F_x\ot_{\cO_{X,x}}\Hx,\nabla)).
\end{equation}

\begin{Defi}
 Let $(M,\nabla)$ be a differential module over $(\Hx,d)$.
 \begin{itemize}
 \item We say that $(M,\nabla)$ is {\it pure} if and only if all its
  spectral radii are equal.
 \item We say that $(M,\nabla)$ is {\it solvable} if and only if all
  its spectral radii are equal to 1 (maximal).
 \end{itemize}
\end{Defi}
 
\subsubsection{Robba's decomposition by the spectral radii and augmented
 Dwork-Robba decomposition} There exist two important
results about the decomposition of a differential equation by the
radii of convergence. The first
one is a decomposition by the spectral radii of differential module
over $\Hx$, initially provided by
Robba \cite{Rob75a} in the framework of affine line. The second
one is a local decomposition by the spectral radii of convergence,
firstly proved by Dwork and Robba \cite{DR77} around a point of type
(2) of $\AF$. These two results where
generalized by Poineau and Pulita in \cite{np3} for the case of a
curve, and in some conditions they provide a global
decomposition by the radii of convergence instead of the spectral
radii. We recall these results in their general form.
\begin{nota}
 Let $(M,\nabla)$ (resp. $(\F,\nabla)$) be a differential module over $(A,d)$
 (resp. differential equation over $X$), we denote by $(M^*,\nabla^*)$
 (resp. $(\F^*,\nabla^*)$) its dual. 
\end{nota}
\begin{Theo}[Robba, {\cite{Rob75a}, \cite[Corollary~3.6.]{np3}}]\label{sec:robb-decomp-spectr}
 Let $x\in \Xh$. Let $(M,\nabla)$ be a differential
 module over $(\Hx,d)$, where $d:\Hx\to\Hx$ is a nonzero bounded $F$-linear derivation. There exists a unique decomposition
 \begin{equation}
  \label{eq:21}
  (M,\nabla)=\bigoplus_{0<\rho\leq 1}(M^\rho,\nabla^\rho)
 \end{equation}
 with the following properties:
 \begin{enumerate}
 \item we have
  $M_\rho=0$ if and only if $\rho\notin
  \{\RS_1(x,(M,\nabla)),\cdots,\RS_n(x,(M,\nabla))\}$;
 \item If $\rho=\RS_i(x,(M,\nabla))$, then $(M^\rho,\nabla^\rho)$ is
  pure with $\RS_1(x,(M^\rho,\nabla^\rho))=\RS_i(x,(M,\nabla))$.
 \item For all $\rho\in(0,1]$, we have $(M^\rho)^*\simeq (M^*)^\rho$.
  
 \end{enumerate}
\end{Theo}
\begin{rem}
 A decomposition satisfying only point (1) and of (2) of
 Theorem~\ref{sec:robb-decomp-spectr}, satisfies automatically (3) of
 the theorem. 
\end{rem}

\begin{Theo}[{\cite[First thm. of Section~4]{DR77},\cite[Theorem~4.1.1]{np3}}]\label{sec:robb-decomp-spectr-2}
 Let $x\in \Xh$. Let $(M_x,\nabla_x)$ be a differential module over
 $(\cO_{X,x},d)$, where $d$ is a nonzero bounded $F$-linear derivation defined in
 affinoid neighborhood of $x$. There exists a unique decomposition
 \begin{equation}
  \label{eq:30}
  (M_x,\nabla_x)=\bigoplus_{0<\rho\leq 1}(M^\rho_x,\nabla_x^\rho)
 \end{equation}
 such that for all $0<\rho\leq$ we have
 $M_x^\rho\ct_{\cO_{X,x}}\Hx=(M_x\ct_{\cO_{X,x}}\Hx)^\rho$. Moreover
 if $M_x^{\geq \rho}:=\bigoplus_{\rho\leq \rho'}M_x^{\rho'}$ then 
 \begin{enumerate}
 \item The canonical composite map $(M^*_x)^{\geq \rho}\to M^*_x\to
  (M_x^{\geq\rho})^*$ is an isomorphism, in particular
  ${(M_x^\rho)^*\simeq (M_x^*)^\rho}$.
 \item If $M_x$ and $N_x$ are differential modules over
  $(\cO_{X,x},d)$, and if $\rho\ne \rho'$, then $\hom_{\cO_{X,x}}(M_x^\rho,N_x^{\rho'})=0$.
 \end{enumerate}
\end{Theo}

\begin{Theo}[{\cite[Theorem~5.3.1]{np3}}]\label{sec:robb-decomp-spectr-1}
 Assume that the index $i$ separates the radii of $\F$ over $X$. Then
 $\F$ admits a sub-object $(\F_{\geq i},\nabla_{\geq i})\subset
 (\F,\nabla)$ such that for all $x\in X$ one has
 \begin{enumerate}
 \item $\rank\F_{\geq
   i}=\dim_\Omega\omega_{\cS,i}(x,\F)=n-i+1$.
 \item For all $j=1,\cdots,n-i+1$ the canonical inclusion
  $\omega(x,\F_{\geq i})\subset \omega(x,\F)$ identifies
  $\omega_{\cS,j}(x,\F_{\geq i})= \omega_{\cS,j+i-1}(x,\F)$.
 \end{enumerate}
 Define $\F_{<i}$ by
 \begin{equation}
  \label{eq:31}
  0\to\F_{\geq i}\to \F \to \F_{<i}\to 0.
 \end{equation}
 Then, for $x\in X$, one has
 \begin{equation}
  \label{eq:32}
  \cR_{\cS,j}(x,\F)=\begin{cases}\cR_{\cS,j}(x,\F_{<i})& \text{if }
                              j=1,\cdots,
                              i-1\\
            \cR_{\cS,j-i+1}(x,\F_{\geq i})& \text{if } j=i,\cdots,n.\\
  \end{cases}
\end{equation}
\end{Theo}

\begin{rem}
 We will see in Section~\ref{sec:refin-decomp-with} that, using the spectrum, we can improve slightly the decomposition.
\end{rem}
\subsection{Differential modules and spectral theory in the sense of Berkovich}\label{sec:spectr-theory-sense-2} Let us recall
now the definition of the spectrum of Berkovich. To avoid any
confusion we fix another coordinate function $S$ on $\AF$. Let $E$ be
a $F$-Banach algebra and $f\in E$. The spectrum of $f$ as an element
of $E$ is the set $\Sigma_{f,F}(E)$ of point $x\in \AF$ such that
$f\ot 1-1\ot S(x)$ is not invertible $E\ct_F\Hx$. If there is no
ambiguity, we may just use the notation $\Sigma_f$.

As in the complex case, the spectrum exhibits the same beneficial properties, such as it is a
nonempty compact set, the smallest closed disk centered at zero and
containing $\Sigma_f$ has radius equal to $\nsp{f}:=\lim\limits_{n\to
 \infty}\nor{f^n}^{\fra{n}}$ (cf. \cite[Theorem~7.1.2]{Ber}). Moreover, the resolvent function of $f$
defined as follows:
\begin{equation*}
 \Fonction{R_f:\;\AF\setminus \Sigma_f}
 {\coprod\limits_{x\in\AF\setminus\Sigma_f}E\ct_F\Hx}{x}{(f\ot 1-1\ot T(x))\-1}
\end{equation*}
is an {\it analytic function}, with value in $E$, equal to 0 at $\infty$. In
other words,
$R_f\in \cO_{\Pf}(E)(\Pf\setminus\Sigma_f)$, where $\cO_{\Pf}(E)$ is the
{\it sheaf of analytic function} over $\Pf$ with value in $E$, defined as
follows:
\begin{equation*}
 \cO_{\Pf}(E)(U)=\varprojlim\limits_{V\subset U}E\ct_F\cA_V
\end{equation*}
where $U$ is an open set of $\Pf$, $V$ an affinoid domain of $\Pf$ and
$A_V$ the associated $F$-affinoid algebra. 

\begin{Lem}[{\cite[Lemma~2.20]{Azz20}}]\label{sec:spectr-theory-sense}
If $a\in \AF\setminus\Sigma_f$ then the largest open disk centered at
$a$ and contained in $\AF\setminus\Sigma_f$ has radius $R=\nsp{(f-a)\-1}\-1$. 
\end{Lem}

Let us explain now in which sense we associate a spectrum in the
sense of Berkovich to a differential equation over a quasi-smooth analytic curve. Recall that given $(\Omega,d)$ with
$\Omega\in E(F)$ and $d:\Omega\to \Omega$ be a nonzero bounded $F$-linear
derivation, the spectrum of a differential module $(M,\nabla)$ over
$(\Omega,d)$ is the spectrum of $\nabla$ as a bounded $F$-linear
operator. This means that the spectrum depends highly on the
choice of $d$. In our paper \cite{Azz24}, assuming that $F$ is
algebraically closed, we find out that there exists a family
of derivations for which we were able to determine the spectrum of
any differential module $(M,\nabla)$. Moreover, in that case the
spectrum provides the data of spectral radii of convergence and the
formal exponent of the equation. We recall this result after
providing some additional notations.
 
Let $(\F,\nabla)$ be a differential equation over a quasi-smooth
analytic curve $X$. Let $x\in X$, then there exists an open
neighborhood $U$ of $x$ such that there exist an isomorphism of
$\cO_{X|_U}$-module between $\Omega_{X/F|_U}$ and $\cO_{X|_U}$, which
is bounded on affinoid domains. By choosing an isomorphism we define a
nonzero
$F$-linear 
derivation $d:\cO_{X|_U}\to \cO_{X|_U}$, which is bounded on
affinoid domains. For $x\in \Xh$ the derivation
$d$ extends to a nonzero bounded $F$-linear derivation on $\Hx$. By
composing $\nabla$ with the isomorphism, for all $V\subset U$ open
subset the pair $(\F(V),\nabla)$ is a differential module over
$(\cO_X(V),d)$, moreover by scalar extension
$(\F(x),\nabla_d):=(\F_x\ot_{\cO_{X,x}}\Hx,\nabla)$ is a differential
module over $(\Hx,d)$.   

\begin{rem}
 Any nonzero $F$-linear derivation $D$ on $\cO_{X|_U}$, bounded on the affinoid domains, is
 equal to $g.d$, where $g\in \cO_{X}(U)\setminus\{0\}$. Moreover, we can find a
 smaller open subset $V$ of $U$ such that the derivation $D$ is
 induced by an isomorphism $\Omega_{X/F|_V}\simeq\cO_{X|_V}$ that is
 bounded on the affinoid domains. Therefore, if $D$ is such derivation,
 well defined on a neighborhood of $x\in \Xh$, then we can consider
 $(\F(x),\nabla_D)$ the differential module over $(\Hx,D)$, in other
 words the choice of this derivation will not change the differential
 equation around $x$.  
\end{rem}

Suppose now that $X$ is an analytic domain of $\AF$, for 
simplicity we may assume that it is connected. Let $x=x_{c,r}\in X_{[2,3]}$. Let $f\in
\Hx$, such that $f^{\sharp}:\Hx\to \Hx$, $T-c\mapsto f$ is a well defined isomorphism in the category $E(F)$. We set $d_f:=f\D{f}$. 
\begin{rem}
 For $a\in F$, the spectral radius of convergence of $(\Hx,f\D{f}-a)$
 is equal to $\cR(\delta(a))$.
\end{rem}
 \begin{Theo}[{\cite[Theorem~1.1]{Azz24}}]\label{sec:spectr-theory-sense-1}
 Assume that $F$ is algebraically closed and $\crk=p>0$. Let
 $(M,\nabla)$ be a differential module $(\Hx, f\D{f})$. Let
 $\phi:\AF\to \AF$, $T-c\mapsto (T-c)^p$. We denote $\phi^l(x)$ by
 $x^{p^l}$. We use notations of \eqref{eq:93} and \eqref{eq:94}.
\begin{enumerate}

\item There exist $z_1,\cdots,z_{\nu}\in \AF\setminus F$ and
 $a_1,\cdots, a_\mu\in F$, such that
 \[\Sigma_{\nabla,F}(\Lk{M})=\{z_1,\cdots,z_\nu, a_1,\cdots,
  a_\mu\}+\ZZ_p.\]
 where $z_i$ has the same type as $x$, and $(\nu,\mu)$ is not equal $(0,0)$.
\item We can choose $z_i$ and $a_i$ such that the set $\{z_1,\cdots,z_\nu, a_1,\cdots,
  a_\mu\}$ has a minimal cardinality. Indeed it is enough to keep only
  $z_i$ and $a_i$, for which we have $\{z_i\}+\ZZ_p\cap
  \{z_j\}+\ZZ_p=\emptyset$ and $\{a_i\}+\ZZ_p\cap \{a_j\}+\ZZ_p=\emptyset$
  for $i\ne j$.
 
 \item We choose $\{z_1,\cdots,z_\nu, a_1,\cdots,
  a_\mu\}$ to be minimal. Then we have a unique (up to an isomorphism) decomposition
  \begin{equation}
  \label{eq:86}
   (M,\nabla)=\bigoplus_{i=1}^{\nu}(M_{z_i},\nabla_{z_i})\oplus \bigoplus_{j=1}^{\mu}(M_{a_j},\nabla_{a_j}),
  \end{equation}
  where $(M_{z_i},\nabla_{z_i})$ and $(M_{a_j},\nabla_{a_j})$ are
  differential modules over $(\Hx,f\D{f})$, such that $\Sigma_{\nabla_{z_i},F}(\Lk{M_{z_i}})=\{z_i\}+\ZZ_p$
  and $\Sigma_{\nabla_{a_j},F}(\Lk{M_{a_j}})=\{a_j\}+\ZZ_p$.
 \item For all $a\in F$, the differential module
  $(M_{z_i},\nabla_{z_i})$ (resp. $(M_{a_j},\nabla_{a_j})$) is pure.
  \item Let $c_i\in F$ and $r_i>0$ such that $z_i=x_{c_i,r_i}$. If
  $|p|^{l}\leq r_i< |p|^{l-1}$, with $l\in \NN\setminus\{0\}$, then $\car(\{z_i\}+\ZZ_p)=p^l$ and $\{z_i\}+\ZZ_p=\{x_{c_i,r_i},
  x_{c_i+1,r_i},\cdots, x_{c_i+p^l-1,r_i}\}$. If $r_i\geq1$, $\car(\{z_i\}+\ZZ_p)=1$ and
  $\{z_i\}+\ZZ_p=\{x_{c_i,r_i}\}$.
 \item If $r_i>1$, let $P_{z_i}(f\D{f})$ be a differential polynomial
  associated to $(M_{z_i},\nabla_{z_i})$. Then the image by
  $\pi_{\widehat{\Hx^{alg}}/F}$ of all roots of $P_{z_i}(S)$ (the
  commutative polynomial associated to $P_{z_i}(f\D{f})$) is equal to $z_i$. 
 \item If $|p|^l<r_i\leq|p|^{l-1}$, let $P_{z_i}(p^l(T-c)\dT)$ be a differential
  polynomial associated to $(\phi^l\circ f)_*(M,\nabla)$ (as a differential
  module over $(\h{x^{p^l}},p^l(T-c)\dT)$). Then the image by
  $\pi_{\widehat{\h{x^{p^l}}^{alg}}/F}$ of all roots of $P_{z_i}(S)$
  (the commutative polynomial associated to $P_{z_i}(p^l(T-c)\dT)$) is equal to $\{x_{c_i,r_i},
  x_{c_i+1,r_i},\cdots, x_{c_i+p^l-1,r_i}\}$. In the special case
  where $r_i=|p|^{l-1}$ we have $\{x_{c_i,r_i},
  x_{c_i+1,r_i},\cdots, x_{c_i+p^l-1,r_i}\}=\{x_{c_i,r_i},
  x_{c_i+1,r_i},\cdots, x_{c_i+p^{l-1}-1,r_i}\}$.

 \item If $r_i\geq 1$. For $a\in \disf{c_i}{r_i}\cap F$ we have
  \begin{equation}
  \label{eq:87}
   \RS_1(x,(M_{z_i},\nabla_{z_i}-a))=\frac{\omega}{r_i},
  \end{equation}
  and for all $a\in F\setminus \disf{c_i}{r_i}$
  \begin{equation}
 \label{eq:88}
   \RS_1(x,(M_{z_i},\nabla_{z_i}-a))=\cR(\delta({a-c_i}))=\frac{\omega}{|a-c_i|}.
  \end{equation}
 \item If $|p|^l\leq r_i<|p|^{l-1}$. For all $a\in \bigcup_{j=0}^{p^l-1} \disf{c_i+j}{r_i}\cap
  F$ we have
  \begin{equation}
  \label{eq:89}
   \RS_1(x,(M_{z_i},\nabla_{z_i}-a))=\left(\frac{|p|^l\omega}{r_i}\right)^{\fra{p^l}},
  \end{equation}
  and for all $a\in F\setminus \bigcup_{j=0}^{p^l-1} \disf{c_i+j}{r_i}$
  \begin{equation}
  \label{eq:90}
   \RS_1(x,(M_{z_i},\nabla_{z_i}-a))=\cR(\delta({a-c_i})).
  \end{equation}
 \item For all $a\in
  \{a_j\}+\ZZ_p$, $(M_{a_j},\nabla_{a_j})$ is solvable, and for all $a\in
  F\setminus \{a_j\}+\ZZ_p$, we have $\RS_1(x,(M_{a_j},\nabla_{a_j}-a))=\cR(\delta({a-a_j}))$. 
\end{enumerate}
 \end{Theo}

 \begin{rem}
  From this theorem, we observe that the spectrum depends on the
  choice of $f$. 
 \end{rem}

One of the aim of this paper is to generalize
 Theorem~\ref{sec:spectr-theory-sense-1} for differential equations
 over a quasi-smooth analytic curve $X$, and study some of its
 applications. If $\cS$ is a weak triangulation of $X$, then each point $x$
 of $X\setminus \cS$ has a neighborhood isomorphic to an analytic
 domain of $\AF$, therefore the theorem extends easily for theses
 points. The problem now is for points $x$ of type (2) of $\cS$, how can we choose a
 derivation with respect to the branches $b\in T_xX$. Furthermore, as we observed, 
 even missing branches can provide information about the equation. In the
 case of curves, how can we find these missing branches. Our strategy to deal with these issues is inspired from the works of J. Poineau and
 A. Pulita \cite{and,np2}. That is why our first step is to study the
 continuity and the variation of
 the spectrum over a segment of the form $(c,x]\subset \AF$, with $c\in
 F$. For that, the next section explain the meaning of the
 continuity of the spectrum.

\section{Continuity of the spectrum of an ultrametric operator}\label{sec:cont-spectr-an}
 Let $\cK(\AF)$ be the set of nonempty compact subsets of
   $\AF$. We endow $\cK(\AF)$ with the exponential topology,
   generated by the basis composed by all sets of the following form:
   \begin{equation}
    \label{eq:1}
    (U,\{U_i\}_{i\in I}):=\{\Sigma\in\cK(\AF);\; \Sigma\subset U,\;
    \Sigma\cap U_i\ne\emptyset\; \forall i\},
   \end{equation}
  where $U$ is an open of $\AF$ and $\{U_i\}_{i\in I}$ is a finite
  open cover of $U$. In this case, since $\AF$ is a Hausdorff space,
  then so is for $\cK(\AF)$. For this topology, the union of two compact
  sets is a continuous map. For further properties about this
  topology we can refer to \cite[Section~5]{Azz20}.
  \begin{nota}
     For a non-zero $F$-Banach algebra $E$, we set
     \begin{equation}
      \label{eq:10}
      \Fonction{\Sigma_{.,F}(E):\; E}{\cK(\AF)}{f}{\Sigma_{f,F}(E),}
     \end{equation}
     where $\Sigma_{f,F}(E)$ is the spectrum of $f$ as an element
     of $E$.
     
    \end{nota}
The aim of this section is the study the continuity of
$\Sigma_{.,F}(E)$. Since the spectral semi-norm $\nsp{.}: E\to \R+$ may not be continuous, the map $\Sigma_{,F}(E)$ may
    fail to be continuous. For that we provide the following example:

    \begin{exam}\label{sec:cont-spectr-an-1}
     This example is inspired by the example given in
     \cite[p. 34]{aup}. Assume that $\crk=p>0$. Let
     $E:=\Lk{F\{T\}}$. Let $\mathrm{Sh}\in E$ be the operator defined by 
     \[ \forall n\in \NN,\qq \mathrm{Sh}(T^n)=\alpha_n T^{n+1},\]
     where
     \[\alpha_n=
      \begin{cases}
       p^\ell& \text{ if } n=2^\ell(2m+1)-1 \text{ with } \ell\in
       \NN\setminus\{0\},\; m\in \NN\\
       1& \text{ otherwise}\\
       
      \end{cases}.\]
     For $\ell\in\NN\setminus\{0\}$, let $\mathrm{S}_\ell\in E$ be the operator
     defined by
     
     \[\forall n\in \NN,\qq \mathrm{S}_\ell(T^n)=
      \begin{cases}
       0 & \text{ if } n=2^\ell(2m+1)-1 \text{ with } m\in \NN\\
       \mathrm{Sh}(T^n) & \text{ otherwise}\\
       
      \end{cases}.\]
Then we have
    $(\mathrm{Sh-S}_\ell)(\sum_{i\in\NN}a_iT^i)=\sum_{m\in\NN}p^\ell
    a_{2^\ell(2m+1)-1}T^{2^\ell(2m+1)}$. Therefore,
    $\nor{\mathrm{Sh-S_\ell}}\leq |p|^\ell$, and we obtain $\mathrm{S}_\ell\overset{\ell\to \infty}{\to}
    \mathrm{Sh}$. On one hand we have $\mathrm{S}_\ell^{2^\ell}=0$. Hence,
    $\nsp{\mathrm{S}_\ell}=0$ and $\Sigma_{\mathrm{S}_\ell,F}(E)=\{0\}$. On the other hand, since
    $\mathrm{Sh}^{n'}(T^n)=\alpha_n\alpha_{n+1}\cdots
    \alpha_{n+n'-1}T^{n+n'}$, we have \[\nor{\mathrm{Sh}^{n'}}=\sup_{n\in\NN}|\alpha_n\alpha_{n+1}\cdots
    \alpha_{n+n'-1}|.\] By construction we have:
    \[\alpha_0\alpha_2\cdots
     \alpha_{2^\ell-2}=\prod_{j=1}^{\ell-1}p^{j2^{\ell-j-1}}=(p^{\sum_{j=1}^{\ell-1}j2^{-j-1}})^{2^\ell},\]
    thus, as $2>\frac{2^\ell}{2^\ell-1}$ we have
   \[|\alpha_0\alpha_2\cdots
     \alpha_{2^\ell-2}|^{\fra{2^\ell-1}}>
     (p^{-\sum_{j=1}^{\ell-1}j2^{-j-1}})^2.\]
    We set $\sigma=\sum_{j=1}^\infty j2^{-j-1}=4$. Then we obtain
    \[ 0<p^{-2\sigma}\leq \lim\limits_{\ell\to
      \infty}\nor{\mathrm{Sh}^{2^\ell-1}}^{\fra{2^\ell-1}}=\nsp{\mathrm{Sh}}.\]
    Consequently, the sequence $(\Sigma_{\mathrm{S}_\ell,F}(E))_{\ell\in \NN}$
    does not converge to $\Sigma_{\mathrm{Sh},F}(E)$ and $\Sigma_{.,F}(E)$
    is not continuous at $\mathrm{Sh}$.
     
   \end{exam}

\subsection{Some continuity results on the spectrum}

\subsubsection{The Cauchy integral} To make the reader comfortable, we explain notations involved in the formula of Cauchy integral (cf. \eqref{eq:16}), already
presented in \cite[Chapter
 8]{Ber}.
 \begin{conv}
  We assume that $F$ is algebraically closed.
 \end{conv}

 \begin{Defi}
  Let $X$ be an analytic domain of $\Pf$. Let $E$ be an $F$-Banach algebra. Let $\Sigma$ be a compact subset of $X$. The $F$-algebra of
  analytic functions defined in a neighborhood of $\Sigma$ with
  value in $E$ is defined as follows:
  \[\cO_X(E)(\Sigma):=\varinjlim_{\Sigma\subset U}\cO_X(E)(U)\]
  where $U$ is an open subset of $X$. We endow $\cO_X(E)(\Sigma)$ with
  the semi-norm \[\nor{f}_\Sigma:=\max_{x\in\Sigma}|f(x)|.\]
 We set $\cO_X(\Sigma):=\cO_X(F)(\Sigma)$.
 \end{Defi}

 \begin{nota}
  Let $V$ be an affinoid domain of $\Pf$. If $\infty\in V$, we set
  \[(\cA_V)_0:=\{f\in \cA_V|\; f(\infty)=0\}.\]
  Otherwise, we set $(\cA_V)_0:=\cA_V$.
 \end{nota}

 \begin{nota}
  Let $\Sigma$ be a compact (resp. open) subset of $\Pf$. If $\infty\in \Sigma$, we set
  \[(\cO_{\Pf}(E)(\Sigma))_0:=\{f\in \cO_{\Pf}(E)(\Sigma)|\; f(\infty)=0\}.\]
  Otherwise, we set $(\cO_{\Pf}(E)(\Sigma))_0:=\cO_{\Pf}(E)(\Sigma)$. 
 \end{nota}

 \begin{Theo}[Holomorphic functional calculus,{\cite[Corollary~7.3.4]{Ber}}]\label{sec:notat-defin}
  Let $E$ be a nonzero $F$-Banach algebra and let $f\in E$. The
  morphism of $F$-algebras
  $F[T]\to E$, which assigns $f$ to $T$, extends to a bounded morphism
  \[\cO_{\AF}(\Sigma_{f,F}(E))\to E.\]
  The image of an element $g\in\cO_{\AF}(\Sigma_{f,F}(E))$ will be
  denoted by $g(f)$.
 \end{Theo}

 \begin{Theo}[Shilov Idempotent Theorem, {\cite[Theorem~7.4.1]{Ber}}]\label{sec:notat-defin-2}
  Let $E$ be a nonzero commutative $F$-Banach algebra, and let $\cU$ be an
  open and closed subset of $\cM(E)$. Then there exists a unique
  idempotent $e\in E$ which is equal to $1$ on $\cU$ and $0$ outside $\cU$. 
  
 \end{Theo}

 A closed disk $D$ of $\Pf$ is a set of the form $\disf{a}{r}$
 or $\Pf\setminus \diso{a}{r}$ with $a\in F$. Its radius $r(D)$ is
 equal to $r$. Any affinoid domain of $\Pf$ is a finite
 disjoint
 union of connected affinoid domains. Note that any connected affinoid domain $V\subset\Pf$ is a
 finite intersection of closed disks of $\Pf$. Moreover, we can
 represent $V$ uniquely as $V=\bigcap_{i=0}^\mu D_i$, where $D_i$ are closed disks in $\Pf$ with
 $D_i\not\subset D_j$ for any $i\ne j$.

 \begin{nota}
  Let $V$ be a connected affinoid domain of $\Pf$. Let $V=\bigcap_{i=0}^\mu
  D_i$ be the unique representation introduced above. We set
  \[\EE(V):=\{D_0,\cdots, D_\mu\}.\]
  Let $V$ be an affinoid domain of $\Pf$ and let $V_1,\cdots, V_\nu$
  be its connected components. We
  set \[\EE(V):=\bigcup_{i=1}^\nu \EE(V_i).\] 
 \end{nota}

 \begin{rem}
  In the case where $V$ is not connected, we can find $D$ and $D'$ in
  $\EE(V)$ such that $D\subset D'$. For example $(\disf{0}{2}\cap(\Pf\setminus\diso{0}{1}))\cup \disf{0}{\fra{2}}$.
 \end{rem}

 \begin{nota}
  We denote by $\Af$ the set of affinoid domain of $\Pf$ and
  by $\Af_c$ the set of connected affinoid domain of $\Pf$. The set
  of finite unions of closed disks which are not equal to $\Pf$ will be
  denoted by $\Af_d$.
 \end{nota}

 \begin{exam}
 The affinoid domain $(\Pf\setminus\diso{0}{1})\cup \disf{0}{\fra{2}}$ is
  an element of $\Af_d$. 
 \end{exam}           
 
 Let $V:=\coprod_{i=1}^m V_i\subset \Pf$ be an affinoid domain. Using the Mittag-Leffler
 decomposition for connected affinoid domains
 \cite[Proposition~2.2.6]{Van}, we obtain an isometric isomorphism of
 Banach spaces: 
 \begin{equation}
  \label{eq:15}
  \Fonction{\bigoplus_{D\in
     \EE(V)}(\cA_{D})_0}{(\cA_V)_0}{(f_{D})_{D\in \EE(V)}}{f:=\sum\limits_{i=1}^m(\sum\limits_{D\in\EE(V_i)}f_{D}|_{V_i}).1_{V_i}}
 \end{equation}
where $1_{V_i}$ is the characteristic function of $V_i$. 
 
  Let $D$ and $D'$ be two disks. We say that $D$ and $D'$ are of the same
  type if $\{\infty\}\cap D=\{\infty\}\cap D'$. 
  
 Let $V$ and $V'$ be two connected affinoid domains such that
 $V\subset V'$. We assume that $\infty\in V$ or $\infty\not\in V'$. Then each disk $D'\in \EE(V')$ contains exactly one
 disk $D\in \EE(V)$. Thus, we get the map:
 \[\Fonction{\EE(V')}{\EE(V)}{D'}{D}.\]

 \begin{nota}
  Let $V$ and $V'$ be two connected affinoid domains. We write $V\prec V'$, if the map $\EE(V')\to\EE(V)$:
  $D'\mapsto D$ is a
  bijection and $D\ne D'$ for all $D'\in \EE(V')$. More generally,
  for $V$, $V'\in \Af$ where $V=\bigcup_{i=1}^\nu V_i$ (resp.
  $V'=\bigcup_{i=1}^{\nu '} V'_i$) is the decomposition into
  connected affinoids, we write $V\prec V'$, if for each $i\in
  \{1,\cdots, \nu\}$ there
  exists $j\in \{ 1,\cdots,\nu\}$ such that $V_i\prec V'_j$. 
 \end{nota}

 \begin{Defi}
  Let $D$ and $D'$ be two disks. We say that $D$ and $D'$ are {\em
   complementary} if they are of different types, $D\cup D'=\Pf$
  and $r(D)\ne r(D')$. Let $V$, $V'\in \Af$. We say that $V$ and $V'$
  are {\em complementary} if $V\cup V'=\Pf$ and there exists a bijection $\EE(V)\to
  \EE(V')$ sending each disk $D\in\EE(V)$ to a complementary disk
  $D'\in \EE(V')$. 
 \end{Defi}

 \begin{rem}
  If such bijection exists then it is unique.
 \end{rem}

 \begin{exam}
  Let $V= \couf{0}{\fra{2}}{\frac{3}{2}}\cup \disf{0}{\fra{4}}$. The
  affinoid $V'=(\Pf\setminus \diso{0}{1})\cup
  \couf{0}{\fra{8}}{\frac{3}{4}}$ is a complementary set of $V$.
 \end{exam}
 
 \begin{rem}
  Note that  there exist affinoid domains that do not admit
  complementary sets. Indeed, for example the closed annulus $\couf{a}{r}{r}$ does not admit complementary
  sets. Indeed, we have $\EE(\couf{a}{r}{r})=\{\disf{a}{r},\;
  \Pf\setminus\diso{a}{r}\}$. If $V'$ is a complementary affinoid set then we
  must have $\EE(V')=\{\disf{a}{r_1},\; \Pf\setminus\diso{a}{r_2}\}$
  with $r_1>r$ and $r_2<r$. Then the unique possibility is
  $V'=\couf{a}{r_2}{r_1}$ which does not verify $V'\cup \couf{a}{r}{r}=\Pf$. 
 \end{rem}

 \begin{nota}
 We denote by $\Af '$ the set of affinoid domains which admits
 complementary sets. We set $\Af_c':=\Af_c\cap \Af '$. 
\end{nota}

\begin{Lem}
 Let $V\in \Af$. Then $V\in \Af'$ if and only if there exists $V'\in
 \Af$ with $V'\prec V$. 
\end{Lem}

\begin{Lem}
 Let $V:=\coprod_{i=1}^n V_i\in \Af'$, where the $V_i$ are connected affinoid
 domains. Let $U$ be a complementary set of $V$. Then we have:

 \begin{equation}
  \label{eq:40}
  U=\bigcap_{i=1}^n\bigcup_{D\in\EE (V_i)} D'
 \end{equation}

 where $D'\in \EE(U)$ is the complementary disk of $D$. In particular,
 if $V_1, V_2\in \Af'$, then $V_1\cup V_2\in \Af'$.
\end{Lem}

\begin{proof}
 We set $X:=\bigcap_{i=1}^n\bigcup_{D\in\EE (V_i)}
 D'=\bigcap_{i=1}^n\bigcup_{j=1}^{n_i}D'_{ij}$. Since the bijection
 $D\mapsto D'$ is unique, it is enough to show that $\EE(X)=\EE(U)$ and
 $X\cup V=\Pf$. We set $I:=\prod_{i=1}^{n}\{1,\cdots, n_i\}$ and for
 $\ell\in I$ we have $\ell=(\ell_1,\cdots,\ell_n)$. Then we
have $X:=\bigcup_{\ell\in I}\bigcap_{i=1}^nD'_{i \ell_i}$. To prove that
for each
$i$ and $j$ there exits $\ell\in I$ such that $D'_{ij}\in
\EE(\bigcap_{k=1}^nD_{k \ell_k}')$, we need to prove that there exists
$\ell$ such that $D'_{ij}\not\supset \bigcap_{k=1, k\ne
 i}^nD_{k\ell_k}'$. Since $D_{ij}'\supset\bigcap_{k=1,k\ne
 i}^nD_{k\ell_k}'$, implies $\exists k\ne i$ such $D_{ij}'\supset
D_{k\ell_k}'$, we only need to prove that there exists $\ell$ such that for
all $k\ne i$ we have $D'_{ij}\not\supset D'_{k\ell_k}$. Since for all
$k\ne i$ we have $V_k\cap V_i=\emptyset$, we can easily find such an $\ell$. Hence, $\EE(X)=\EE(U)$. By construction we have $\Pf\setminus
V\subset X$, hence we obtain $V\cup X=\Pf$.

\end{proof}
 
 \begin{Lem}
  Let $(U,V)$ be a pair of complementary subsets. Then
  \[U\cap V=\bigcup_{D\in\EE(V)}(D\cap D'),\]
  where $D'$ is the disk in $\EE(U)$ complementary to $D$.
 \end{Lem}
 \begin{proof}
  It follows from the decomposition in \eqref{eq:40}.
 \end{proof}
\begin{Lem}\label{sec:notat-defin-1}
 Let $\Sigma\subset\Pf$ be a closed subset with $\Sigma\ne
 \emptyset$, $\Pf$. Then the intersection of affinoid neighborhoods
 of $\Sigma$ that belong to $\Af'$ coincides with $\Sigma$.
\end{Lem}
\begin{proof}
 Since $\Af'$ is stable under union and $\Sigma$ is compact, the
 $V\in\Af'$ containing $\Sigma$ form a basis of neighborhoods of
 $\Sigma$. Hence, the result holds.
\end{proof}

Let $V\in\Af$. We can represent $V$ as intersection $\bigcap_{i=1}^n
V_i$ of elements in $\Af_d$ with $V_i\cup V_j=\Pf$ for $i\ne j$ as
follows. Let $V'\in \Af'$ such that $V\prec V'$ and let $U$ be a
complementary set of $V'$. Then $\EE(V)\overset{\sim}{\to} \EE(U)$,
$D\mapsto D'$. Let $U_1,\cdots, U_n$ be the connected components of $U$. We set $V_i=\bigcup_{D'\in
 \EE(U_i)}D$. Then we have%
\begin{equation}
 \label{eq:39}
 V=\bigcap_{i=1}^n V_i\qqq \text{ and }\qqq \; \forall i\ne j \qqq
 V_j\cup V_i=\Pf.
\end{equation}

This representation does not depend on the choice of $U$.

\begin{exam}
 Let $V=\couf{0}{\fra{2}}{\frac{3}{2}}\cup \disf{0}{\fra{4}}$. Then
 $V=V_1\cap V_2$ with
 \[V_1=\disf{0}{\frac{3}{2}}\qqq \text{ and } \qqq V_2=(\Pf\setminus
  \diso{0}{\fra{2}})\cup \disf{0}{\fra{4}}.\]
\end{exam}

Using Mayer-Vietoris sequence we obtain an isomorphism
$(\cA_V)_0\overset{\sim}{\to}\bigoplus_{i=1}^n
(\cA_{V_i})_0$ of Banach space. By combining this isomorphism with (\ref{eq:15}), we
obtain the isomorphism:

\begin{equation}
 \label{eq:24}
  \Fonction{\bigoplus_{D\in
     \EE(V)}(\cA_{D})_0}{(\cA_V)_0}{(f^{D})_{D\in \EE(V)}}{f:=\sum\limits_{i=1}^n(\sum\limits_{D\in\EE(V_i)}(f^{D}.1_{D}))|_V}
  \end{equation}
where $1_D$ is the characteristic function of $D$.
  \begin{rem}
   Note that if $V$ belong to $\Af_d$ or $\Af_c$, then the
   isomorphisms~(\ref{eq:15}) and~(\ref{eq:24}) coincide. Indeed, in
   the case where $V\in\Af_d$, any complementary set $U$ is a
   connected affinoid domain, hence
   $V=\bigcup_{D'\in\EE(U)}D$. In the case where $V\in \Af_c$, then
   a complementary set $U$ is an element of $\Af_d$, therefore
   the disks $D\in\EE(V)$ coincide with 
   the $V_i$'s of \eqref{eq:39}.
  \end{rem}

  \begin{rem}
   The isomorphisms~(\ref{eq:15})
   and~(\ref{eq:24}) are connected as
   follows: $f_D=\sum f^{D'}$ where the sum is taken over all
   $D'\in\EE(V)$ with $D\subset D'$.
  \end{rem}
  
  \begin{exam}
  Let $V=\couf{0}{\fra{2}}{\frac{3}{2}}\sqcup \disf{0}{\fra{4}}$. We
  set $X_1:=\couf{0}{\fra{2}}{\frac{3}{2}}$,
  $X_2:=\disf{0}{\fra{4}}$, $D_1:=\disf{0}{\frac{3}{2}}$,
  $D_2:=\Pf\setminus\disf{0}{\fra{2}}$ and
  $D_3:=\disf{0}{\fra{4}}$. Then the $V_i's$ of \eqref{eq:39} are:
  $V_1=D_1$ and $V_2=D_2\cup D_3$. The image of
  $(f^{D_1},f^{D_2},f^{D_3})$ by the isomorphism \eqref{eq:24} is
  $f=f^{D_1}|_V+(f^{D_2}.1_{D_2}+f^{D_3}.1_{D_3})|_V$. We have
  $f^{D_1}|_V=f^{D_1}.1_{X_1}+f^{D_1}.1_{X_2}$,
  $(f^{D_2}.1_{D_2})|_V=f^{D_2}.1_{X_1}$ and
  $(f^{D_3}.1_{D_3})|_V=f^{D_3}.1_{X_2}$. Then,
  \[ f= (f^{D_1}+f^{D_2}).1_{X_1}+(f^{D_1}+f^{D_3}).1_{X_2}.\]
  Consequently, we have $f_{D_1}=f^{D_1}$, $f_{D_2}=f^{D_2}$ and $f_{D_3}=f^{D_1}+f^{D_3}$.
 \end{exam}

 \begin{exam}
  Let $V=\couf{0}{\fra{4}}{\fra{2}}\sqcup
  \couf{1}{\fra{4}}{\fra{2}}$. We set
  $X_1:=\couf{0}{\fra{4}}{\fra{2}}$ ,
  $X_2:=\couf{1}{\fra{4}}{\fra{2}}$, $D_1:=\disf{0}{\fra{2}}$,
  $D_2:=\disf{1}{\fra{2}}$, $D_3:=\Pf\setminus\diso{0}{\fra{4}}$ and
  $D_4:=\Pf\setminus\diso{1}{\fra{4}}$. Then the $V_i's$ of
  \eqref{eq:39} are: $V_1=D_1\cup D_2$, $V_2=D_3$ and $V_3=D_4$. The
  image of $(f^{D_1},f^{D_2},f^{D_3}, f^{D_4})$ by the isomorphism
  \eqref{eq:24} is
  $f=(f^{D_1}.1_{D_1}+f^{D_2}.1_{D_2})|_V+f^{D_3}|_V+f^{D_4}|_V$. We
  have $(f^{D_1}.1_{D_1})|_V=f^{D_1}.1_{X_1}$,
  $(f^{D_2}.1_{D_2})|_V=f^{D_2}.1_{X_2}$,
  $(f^{D_3})|_V=f^{D_3}.1_{X_1}+f^{D_3}.1_{X_2}$ and
  $(f^{D_4})|_V=f^{D_4}.1_{X_1}+f^{D_4}.1_{X_2}$. Here,
  $f^{D_3}.1_{X_2}\in \cA_{D_2}$ and
  $f^{D_4}.1_{X_1}\in\cA_{D_1}$. Then,
  \[f=(f^{D_1}+f^{D_3}+f^{D_4}).1_{X_1}+(f^{D_2}+f^{D_3}+f^{D_4}).1_{X_2}.\]
Consequently, we have $f_{D_1}=f^{D_1}+f^{D_4}$,
$f_{D_2}=f^{D_2}+f^{D_3}$, $f_{D_3}=f^{D_3}$ and $f_{D_4}=f^{D_4}$.
 \end{exam}

  \begin{Defi}
   Let $D$ be a closed disk. The {\em residue operator} $\res_D:
   \cA_D\to F$ is defined as follows. If $D\subset \AF$, then
   $\res_D\equiv 0$. If $D=\Pf\setminus\diso{a}{r}$ with $a\in F$
  then
\[\Fonction{\res_D: \cA_D}{F}{\sum_{i=0}^\infty f_i(T-a)^{-i}}{f_1}.\]
It does not depend on the choice of $a$. 
 
 \end{Defi}

 \begin{rem}
  It is easy to see that, if $D=\Pf\setminus\diso{a}{r}$ then $\res_D$
  is bounded and $\nor{\res_D}=r$. Otherwise, $\nor{\res_D}=0$.
 \end{rem}

 Let $D$ and $D'$ be complementary closed disks. Since we have the
 isomorphism (\ref{eq:24}):
 \[\Fonction{(\cA_D)_0\oplus (\cA_{D'})_0}{(\cA_{D\cap
    D'})_0}{(f,f')}{f|_{D\cap D'}+f'|_{D\cap D'}},\]
we can define:
 \begin{equation}
  \label{eq:26}
  \Fonction{\res_{D\cap D'}:(\cA_{D\cap D'})_0}{F}{f}{\res_D (f_D)+\res_{D'}(f_{D'}),}
 \end{equation}

 Moreover, if $D\subset\AF$ then
 $\res_{D\cap D'}(f)=\res_{D'}(f_{D'})$.

 \begin{rem}
  Let $B$ be a nonzero $F$-Banach space. The operators $\res_D$ and
  $\res_{D\cap D'}$ extend naturally
  to bounded operators $\res_D: B\ct_F \cA_D\to B$ and
  $\res_{D\cap D'}:B\ct_F(\cA_{D\cap D'})_0\to B$.
 \end{rem}

 \begin{Defi}
  Let $\Sigma$ be a closed subset of $\Pf$ different from
  $\emptyset$ and $\Pf$. A {\em contour around} $\Sigma$ is a pair
  $(U,V)$ of complementary subsets such that $U$ is a neighborhood
  of $\Sigma$ and $V\subset\Pf\setminus\Sigma$.
 \end{Defi}

 \begin{exam}
  Let $\Sigma=\{a,b\}$ be a subset of $F$. Let
  $U:=\disf{a}{r_a}\sqcup\disf{b}{r_b}$ and
  $V:=\Pf\setminus(\diso{a}{r'_a}\sqcup\diso{b}{r'_b})$
 with $r'_a<r_a$ and
  $r'_b<r_b$. Then $(U,V)$ is a contour of $\Sigma$ and $U\cap V=\couf{a}{r'_a}{r_a}\cup\couf{b}{r'_b}{r_b}$. 
 \end{exam}

 We introduce the Cauchy integral as in \cite[Chapter~8]{Ber}. Let
 $E$ be a nonzero $F$-Banach algebra and let $\Sigma$ be a compact
  subset of $\Pf$. Let $\gamma=(U,V)$ be a contour of $\Sigma$. Then
  the {\em Cauchy integral} around $\gamma$ is the
bounded $F$-linear map $\cE: E\ct_F(\cA_V)_0\to
  \Lk{(\cA_U)_0,E})$, defined as follows: 

   \begin{equation}
    \label{eq:16}
  \Fonction{\cE: E\ct_F(\cA_V)_0}{\Lk{(\cA_U)_0,E}}{\phi}{
   \cE\phi:{f}\to{\sum_{D\in \EE(V)}(Res_{D\cap D'})(f^{D'}|_{D\cap
     D'}\cdot\phi_D |_{D\cap D'})}}, 
 \end{equation}
 where $\Lk{(\cA_U)_0,E}$ is the $F$-Banach
 algebra of bounded $F$-linear map $(\cA_U)_0\to E$, and $D'$ is the disk in $\EE(U)$ complementary to $D$. 

  \begin{Pro}\label{sec:cauchy-integral}
   Let $E$ be a nonzero $F$-Banach algebra and let $g\in E$. Let $(U,V)$ be
   a contour of $\Sigma_{g,F}(E)$, and $R_g$ is the resolvent
   function associated to $g$. Then we have
   \[\cE R_g(f)=f(g),\]
where $f(g)$ is the image of $g$ by the map of Theorem~\ref{sec:notat-defin}.
  \end{Pro}
  \begin{proof}
   See \cite[Theorem 8.1.1, Remark 8.1.2]{Ber}
  \end{proof}

  \subsubsection{Continuity results}

  In the following, we prove that under some assumptions the
  map $\Sigma_{.,F}(E):E\to \cK(\AF)$ is continuous.

  \begin{Lem}\label{sec:continuity-results}
   Let $E$ be a nonzero $F$-Banach algebra and let $f\in E$. For any
   open neighborhood $U\subset\AF$ of $\Sigma_{f,F}(E)$, there exists a
   positive real number $\delta$ such that
   \[\nor{f-g}<\delta\Rightarrow \Sigma_{g,F}(E)\subset U.\]
  \end{Lem}

  \begin{proof}
   We set $\Sigma_f:=\Sigma_{f,F}(E)$ and
   $\Sigma_g:=\Sigma_{g,F}(E)$. Let $U$ be an open
   neighborhood of $\Sigma_f$. Let $\epsilon>0$, we have:
   \[\AF\setminus U=[(\AF\setminus
   U)\cap\disf{0}{\nsp{f}+\epsilon}]\bigcup [\AF\setminus(U\cup\disf{0}{\nsp{f}+\epsilon})].\]
Since the spectral semi-norm is upper semi-continuous, there exists
$\eta_0>0$ such that
\[\nor{f-g}<\eta_0\Rightarrow \nsp{g}<\nsp{f}+\epsilon.\]
Therefore, we obtain
\[\Sigma_g\subset\disf{0}{\nsp{g}}\subset\disf{0}{\nsp{f}+\epsilon}\]
(cf. \cite[Theorem~7.1.2]{Ber}). Hence, we have
\[\nor{f-g}<\eta_0\Rightarrow\AF\setminus(U\cup\disf{0}{\nsp{f}+\epsilon})\subset\AF\setminus\Sigma_g.\]
We prove now that there exists $\eta_1>0$ such that
\[\nor{f-g}<\eta_1\Rightarrow (\AF\setminus
 U)\cap\disf{0}{\nsp{f}+\epsilon}\subset \AF\setminus \Sigma_g.\]

Let $x\in \AF\setminus U$, and let $V_x\subset\AF\setminus \Sigma_f$
be an affinoid neighborhood of $x$. Then $f\ot 1- 1\ot T$ is
invertible in $E\ct_F\cA_{V_x}$. We set $\eta_x=\nor{(f\ot 1- 1\ot
 T)\-1}\-1$. Since for all $g\in E$ we have $\nor{f-g}=\nor{f\ot 1-g\ot 1}$, if $\nor{f-g}<\eta_x$ then $g\ot 1-1\ot T$
is invertible in $B\ct_F \cA_{V_x}$ (cf. \cite[\S 2.4
Proposition~3]{bousp}). Therefore, we obtain

\[\nor{f-g}<\eta_x\Rightarrow V_x\subset \AF\setminus \Sigma_g.\]
As $(\AF\setminus U)\cap\disf{0}{\nsp{f}+\epsilon}$ is compact, there
exists a finite subset $\{x_1,\cdots, x_m\}\subset \AF\setminus U$
such that $(\AF\setminus U)\cap\disf{0}{\nsp{f}+\epsilon}\subset
\bigcup_{i=1}^mV_{x_i}$. We set $\eta_1:=\min_{1\leq i\leq
 m}\eta_{x_i}$.
Hence, we obtain
\[\nor{f-g}<\eta_1\Rightarrow (\AF\setminus U)\cap\disf{0}{\nsp{f}+\epsilon}\subset
 \bigcup_{i=1}^mV_{x_i}\subset \AF\setminus \Sigma_g.\]
We set $\delta:=\min(\eta_0,\eta_1)$. Consequently, we obtain:

\[\nor{f-g}<\delta\Rightarrow \AF\setminus U\subset\AF\setminus \Sigma_g.\]\end{proof}

\begin{Theo}\label{sec:continuity-results-1}
   Let $E$ be a nonzero \emph{commutative} $F$-Banach algebra. The spectrum
   map \[\Sigma_{.,F}:E\to \cK(\AF)\] is continuous.
  \end{Theo}

  \begin{proof}
  We prove the statement by contradiction. Suppose that
  there exists an element $f\in E$ and a sequence $(f_n)_{n\in
   \NN}$ that converges to $f$, such that
  $(\Sigma_{f_n,F}(E))_{n\in \NN}$ does not converges to
  $\Sigma_{f,F}(E)$. We set $\Sigma_f:=\Sigma_{f,F}(E)$ and
  $\Sigma_{f_n}:=\Sigma_{f_n,F}(E)$. Note that, by
  Lemma~\ref{sec:continuity-results}, for any open neighborhood
  $U\subset\AF$ of $\Sigma_f$, there exists $N\in \NN$ such that
  we have $\Sigma_{f_n}\subset U$ for all $n>N$. Therefore, there
  exists $x\in \Sigma_f$ and an open neighborhood $U_x$ of $x$,
  such that $\Sigma_{f_n}\cap U_x=\emptyset$. We know
  that $\pik{\Hx}\-1(\Sigma_{f_n})=\Sigma_{f_n\ot 1, \Hx}(E\ct_F
  \Hx)$(cf. \cite[Theorem~7.1.6]{Ber}). Then we have
  \[\Sigma_{f_n\ot 1,\Hx}(E\ct_F \Hx)\cap
   \pik{\Hx}\-1(U_x)=\emptyset\]
  Since $\pik{\Hx}\-1(U_x)$ is a neighborhood of $T(x)$, there exists $\epsilon>0$ such that
  $\disco{\Hx}{T(x)}{\epsilon}\subset \pik{\Hx}\-1(U_x)$. Hence, we
  have $
  \disco{\Hx}{T(x)}{\epsilon}\subset\A{\Hx}\setminus \Sigma_{f_n\ot 1,\Hx}(E\ct_F \Hx)$. Therefore, we obtain\[\nsp{(f_n\ot
   1- 1\ot T(x))\-1}<\fra{\epsilon}\] (cf. Lemma
  \ref{sec:spectr-theory-sense}). Since $E\ct_F \Hx$ is
  commutative, the spectral semi-norm is sub-multiplicative. We set $u:=f\ot 1-1\ot
  T(x)$ and $u_n:=f_n\ot 1- 1\ot T(x)$. Therefore, we obtain
  \[\nsp{1-uu_n\-1}=\nsp{u_n\-1(u_n-u)}\leq
   \nsp{u_n\-1}\nsp{u_n-u}\leq\fra{\epsilon}\nor{u_n-u}.\]
  Since $f_n$ converges to $f$, $u_n$ converges to $u$. This
  implies that $uu_n\-1$ is invertible and $u$ is right
  invertible. By analogous arguments, $u_n\-1u$ is invertible and
  $u$ is left invertible. Hence, $u$ is invertible which
  contradicts the hypothesis.
  \end{proof}

  \begin{Lem}\label{sec:continuity-results-2}
   Let $E$ be a nonzero $F$-Banach algebra and $f\in E$. Let $\cS$ be a
   closed and open subset of $\Sigma_{f,F}(E)$. For any open
   subset $\cU\subset\AF$
   containing $\cS$, there exists a positive real number
   $\delta$ such that
   \[\nor{f-g}<\delta\Rightarrow \Sigma_{g,F}(E)\cap \cU\ne
   \emptyset.\]
  
  \end{Lem}

   \begin{proof}
     We set $\Sigma_f:=\Sigma_{f,F}(E)$. We suppose that there
     exists an open neighborhood $\cU$ of $\cS$ and a sequence
     $(f_n)_{n\in\NN}$ that converges to $f$ with
     $\Sigma_{f_n,F}(E)\cap \cU=\emptyset$ for all $n\in\NN$. Let $(U,V)$ be a
     contour of $\Sigma_f$ (which exists by Lemma~\ref{sec:notat-defin-1}) such that $U=U_1\cup U_2$ where $U_1$
     (resp. $U_2$) is an affinoid neighborhood of $\cS$
     (resp. $\Sigma_f\setminus\cS$) and $U_1\cap
     U_2=\emptyset$. We can assume that $\Sigma_{f_n,F}(E)\subset
     U$ (cf. Lemma~\ref{sec:continuity-results}) and $V\subset
     \AF\setminus \Sigma_{f_n,F}(E)$ for all $n\in \NN$, hence
     $(U,V)$ is a contour of $\Sigma_{f_n,F}(E)$ for all $n\in \NN$. We can assume that $U_1\subset \cU$. Let
     $1_{U_1}\in \cA_U$ be the characteristic function of
     $U_1$. By Proposition~\ref{sec:cauchy-integral}, we have
     $\cE R_f(1_{U_1})=1_{U_1}(f)$, which is an idempotent element
     of $E$ (cf.\cite[Propositions~7.1.4, 7.2.4]{Ber} and
     Theorem~\ref{sec:notat-defin-2}), since $U_1$ is non-empty $1_{U_1}(f)$ is not equal to
     $0$. Since $\Sigma_{f_n}\cap U_1=\emptyset$, by
     Theorem~\ref{sec:notat-defin} we have
     $1_{U_1}(f_n)=0$. Hence, by
     Proposition~\ref{sec:cauchy-integral} we obtain $\cE R_{f_n}(1_{U_1})=1_{U_1}(f_n)=0$. 
     On the one hand we have $(R_{f_n})_{n\in\NN}$ converges to
     $R_f$ in $E\ct_F \cA_V$. On other hand, the map $\cE: E\ct_F (\cA_V)_0\to
     \Lk{(\cA_U)_0,E}$ is continuous. Hence we obtain a
     contradiction since $1_{U_1}(f)\ne 0$. 
    \end{proof}

    \begin{Theo}\label{sec:continuity-results-3}
    Let $E$ be a nonzero $F$-Banach algebra and $f\in E$. If 
    $\Sigma_{f,F}(E)$ is totally disconnected, then the map
    $\Sigma_{.,F}: E \to \cK(\AF)$ is continuous at $f$. 
   \end{Theo}

   \begin{proof}
    Let $(f_n)_{n\in \NN}\subset E$ be a sequence that converges
    to $f$. Let $(U,\{U_i\}_{i\in I})\subset K(\AF)$ be an open neighborhood of
    $\Sigma_{f,F}(E)$. Let $x\in \Sigma_{f,F}(E)$, since $\Sigma_{f,F}(E)$ is totally
    disconnected, the closed-open subsets of $\Sigma_{f,F}(E)$
    form a fundamental system of neighborhoods of $x$. Hence, for each
    $i$ there exists a closed-open subset
    $\cS\subset\Sigma_{f,F}(E)$ such that $\cS\subset U_i$. By Lemma~\ref{sec:continuity-results},
    there exists $N_0\in \NN$ such that for all $n>N_0$ we have
    $\Sigma_{f_n,F}(E)\subset U$. By
    Lemma~\ref{sec:continuity-results-2}, for all $i\in
    I$ there exists $N_i\in\NN$ such that for all $n>N_i$, $\Sigma_{f_n,F}(E)\cap U_i\ne\emptyset$. Therefore, for
    all $n>\max(N_0,\max_{i\in I}(N_i))$ we have
    $\Sigma_{f_n,F}(E)\in (U,\{U_i\}_{i\in I})$.
   \end{proof}

   \begin{cor}\label{sec:continuity-results-4}
     Let $\Omega\in E(F)$. The map $\Sigma_{.,F}(\cM_n(\Omega)):
     \cM_n(\Omega)\to \cK(\AF)$ is continuous.
    \end{cor}

\section{Main results of the paper}\label{sec:variation-spectrum-f}
\begin{conv}
 In this section we assume that $F$ is algebraically closed and $\crk=p>0$.
\end{conv}

\subsection{Variation of the spectrum the case of an analytic domain
 of $\AF$}\label{sec:vari-spectr-case-6}
Let $X$ be a connected affinoid domain of $\AF$ and let $(\F,\nabla)$
be a differential equation over $X$ of rank $n$. We consider on $X$
the triangulation $\cS=\partial (X)$ (Shilov boundary of $X$). Let
$x\in X_{[2,3]}$, $b\in T_x\Pf$, let $D_b$ be the open disk containing
$b$ and $c\in D_b\cap F$. Let $f$ be a coordinate function on $X$, in the
sense where the induced map $X\to \AF$ by $F[T]\to \cO(X)$,
$T-c\mapsto f$ is a closed immersion and $f(X)=X$. We assume that for all $y\in
(c,x]\cap X$ we have $f(y)=y$. We set $d_f=f\D{f}$, to simplify
notation we set $\nabla_f:=\nabla_{d_f}$. In this section we prove that the following
map is continuous
\begin{equation}
 \label{eq:12}
 \Fonction{\Sigma_{\F,c,f}: (c,x]\cap X}{\cK(\AF)}{y}{\Sigma_{\nabla_{f},F}(\Lk{\F(y)})}.
\end{equation}

In order to do that, we need to prove the following intermediate
results.
\begin{rem}\label{sec:vari-spectr-case}
Let $a\in F$, we set $\cR_a(x)$
 to be the spectral radius of convergence of $(\Hx,d_f-a)$. Then
 $\cR_a$ is constant over $(c,x]\cap X$ (cf. \cite[Lemma~5.35]{Azz24}). Note
 that, for all $a\in F$ we have $\cR_a(x)=\cR_{-a}(x)=\cR(\delta(a))$.  
\end{rem}

 \begin{Lem}\label{sec:vari-spectr-case-1}
 Suppose that for $y_0\in (c,x]\cap X$, we have
 $\Spe(y_0)=\{z\}+\ZZ_p$. Let $c_1,\cdots, c_m$ satisfying the
 following conditions:
 \begin{itemize}
 \item for all $i$, we have $|z-c_i|=r_F(z)$, i.e $z=x_{c_i,r_F(z)}$;
 \item for all $i\ne j$, we have $\delta(c_i-c_j)=|c_i-c_j|=r_F(z)$; 
 \item there exists an interval neighborhood $I$ of $y_0$ in $(c,y_0]\cap X$
  (resp. $[y_0,x]$ if $y_0\ne x$), for which we have
  \begin{itemize}
  \item for all $i$ we have
   $\RS_n(.,(\F(.),\nabla_f-c_i))>\RS_n(y_0,(\F(y_0),\nabla_f-c_i))$
   on $I\setminus\{y_0\}$;
  \item for all $i$ and $k$, $\RS_k(.,(\F(.),\nabla_f-c_i))$ is
   log-affine on $I$.
  \end{itemize}
 \end{itemize}
 We set $k_{c_i}=\min_{1\leq j\leq n}\{k;\;
 \RS_k(.,(\F(.),\nabla_f-c_i))>\RS_n(y_0,(\F(y_0),\nabla_f-c_i))
 \text{ on } I\}$.

 Then for each $y\in I$ we have the following decomposition
 \begin{equation}
  \label{eq:80}
  (\F(y),\nabla_f)=\left(\bigoplus_{i=1}^m(M_{c_i}(y),\nabla_{c_i}(y))\right)\oplus (M_m(y),\nabla_m(y)),
 \end{equation}
 with $\RS_k(y,
 (M_{c_i}(y),\nabla_{c_i}(y)-c_i))=\RS_{k_{c_i}+k-1}(y,(\F(y),\nabla_f-c_i))$
 and $\dim_{\Hy}M_{c_i}=n-k_{c_i}+1$.
\end{Lem}

\begin{proof}
 We point out that since $\Spe(y_0)=\{z\}+\ZZ_p$, then
 $(\F(y_0),\nabla_f-a)$ is pure for all $a\in F$. Hence, for all $i$
 we have
 $\RS_1(y_0,(\F(y_0),\nabla_f-c_i))=\RS_n(y_0,(\F(y_0),\nabla_f-c_i))=\cR(r_F(z))$
  (notation in \eqref{eq:94})
  (cf. Theorem~\ref{sec:spectr-theory-sense-1}).\\~

 We can obtain decomposition \eqref{eq:80} by iterating the following
 process for each $0<l\leq m$. We set
 $(M_0(y),\nabla_0(y))=(\F(y),\nabla_f)$, by decomposing with respect
 to the spectral radii of convergence of $(M_0(y),\nabla_0(y)-c_1)$
 (cf. Theorem~\ref{sec:robb-decomp-spectr})
 we obtain
 \begin{equation}
  \label{eq:67}
   (\F(y),\nabla_f)=(M_{c_1}(y),\nabla_{c_1}(y))\oplus (M_1(y),\nabla_1(y))
  \end{equation}
  with 
 $\RS_k(y,(M_{c_1}(y),\nabla_{c_1}(y)-c_1))=\RS_{k_{c_1}+k-1}(y,(\F(y),\nabla_f-c_1))$. Suppose
 now that for $2\leq l\leq m$ we have
 \begin{equation}
  \label{eq:72}
  (\F(y),\nabla_f)=\left(\bigoplus_{i=1}^{l-1}(M_{c_i}(y),\nabla_{c_i}(y))\right)\oplus (M_{l-1}(y),\nabla_{l-1}(y)),
 \end{equation}
 such that $\RS_k(y,
  (M_{c_i}(y),\nabla_{c_i}(y)-c_i))=\RS_{k_{c_i}+k-1}(y_0,(\F(y_0),\nabla_f-c_i))$
  and $\dim_{\Hy}M_{c_i}(y)=n-k_{c_i}+1$. We set
  $n_{l-1}=\dim_{\Hx}M_{l-1}$. Then, since for $i\ne j$
  \begin{equation}
  \label{eq:73}
 \RS_k(y,
  (M_{c_i}(y),\nabla_{c_i}(y)-c_j))=\min(\RS_k(y,
  (M_{c_i}(y),\nabla_{c_i}(y)-c_i)),\cR_{c_i-c_j}(y))=\cR_{c_i-c_j}(y)=\cR(r_F(z)),
  \end{equation}
  we have 
  \begin{equation}
 \label{eq:77}
   \RS_{n_{l-1}}(y,(M_{l-1}(y),\nabla_{l-1}(y)-c_l))=\RS_n(y,(\F(y),\nabla_f-c_l))
\end{equation}
 and
 \begin{equation}
 \label{eq:78}
  \RS_{n_{l-1}-(n-k_{c_l})}(y,(M_{l-1}(y),\nabla_{l-1}(y)-c_l))=\RS_{k_{c_l}}(y,(\F(y),\nabla_f-c_l)).
\end{equation}
We set
  $(M_{c_l}(y),\nabla_{c_l}(y))$ to be the differential module satisfying 
  \begin{equation}
   \label{eq:79}
   (M_{l-1}(y),\nabla_{l-1}(y))=(M_{c_l}(y),\nabla_{c_l}(y))\oplus (M_l(y),\nabla_l(y))
  \end{equation}
  with
  $\RS_k(y,(M_{c_l}(y),\nabla_{c_l}(y)-c_l))=\RS_{k_{c_l}+k-1}(y,(\F(y),\nabla_f-c_l))$,
  where \eqref{eq:79} is the decomposition with
  respect to the spectral radii of $(M_{l-1}(y),\nabla_{l-1}(y)-c_l)$
  (cf. Theorem~\ref{sec:robb-decomp-spectr}). 
 \end{proof}
 
\begin{Lem}\label{sec:vari-spectr-case-2}
Suppose that for $y_0\in (c,x]\cap X$, we have $\Spe
  (y_0)=\{z\}+\ZZ_p$. Then there exists an interval neighborhood $I$
  of $y_0$ in $(c,y_0]\cap X$ (resp. $[y_0,x]$ if $y_0\ne x$), such that
  \begin{itemize}
   \item if $z=a\in F$, then for all $y\in I$ we have
    \[\Spe(y)=\{x_{a,\psi_1(y)},\cdots,
     x_{a,\psi_\mu(y)}\}+\ZZ_p\]
    where $\{x_{a,\psi_1(y)},\cdots, x_{a,\psi_\mu(y)}\}$ has minimal
    cardinal, with
  $\psi_1,\cdots, \psi_\mu:I\to \R+$ are continuous maps
    on $I$, piecewise log-affine on $I\setminus\{y_0\}$, and
    $\psi_i(y_0)=0$ (i.e. $x_{a,\psi_i(y_0)}=a$).
   
  \item if $z\in \AF\setminus F$, then for all $y\in I$ we have
    \[\Spe(y)=\{x_{c_1,\phi_1(y)},\cdots,x_{c_\mu,\phi_\nu(y)}\}+\ZZ_p,\]
    and $\{x_{c_1,\phi_1(y)},\cdots,x_{c_\nu,\phi_\nu(y)}\}$ has
    minimal cardinality, with $c_1,\cdots,c_\nu\in F$ ($c_i$ are
    not necessarily different),
    $\phi_1,\cdots,\phi_{\nu}:I\to \R+$ piecewise log-affine maps on $I$, and $z=x_{c_i,\phi_i(y_0)}$.
   
  \end{itemize}
 \end{Lem}

 \begin{proof}
  We start by dealing with the easy case where $z=a\in F$. We choose
  $I$ to be
  a neighborhood interval of $y_0$ in $(c,y_0]\cap X$ (resp. $[y_0,x]$ if $y_0\ne x$), such that $\RS_i(.,(\F(.),\nabla_f-a))$ is log-affine for
  all $i$ (cf. Remark~\ref{sec:radii-conv-line-1}). Then $\RS_i(.,(\F(.),\nabla_f-a))$ may
  be constant and identically equal to $1$ or strictly decreasing on
  $I$. Hence, for all $y\in I$ and $e\in F$, we have
  \[\RS_i(y,(\F(y),\nabla_f-e)=\min(\RS_i(y,(\F(y),\nabla_f-a)),\cR_{e-a}(y)).\]
\\~\\
  We set
  \begin{equation}
 \label{eq:81}
   \psi_i(y)=
   \begin{cases}
    0& \text{ if } \RS_i(y,(\F(y),\nabla_f-a))=1\\
    \frac{|p|^k\omega}{\RS_i(y,(\F(y),\nabla_f-a))^{p^k}}& \text{
                               if }
                               \omega^{\fra{p^{k-1}}}\leq
                               \RS_i(y,(\F(y),\nabla_f-a))\leq
                               \omega^{\fra{p^k}},
                               \text{
                               with }
                               k\in\NN\setminus\{0\},\\
    \frac{\omega}{\RS_i(y,(\F(y),\nabla_f-a))} & \text{ if }
                          \RS_i(y,\F(y),\nabla_f-a)\leq \omega.
   \end{cases}
  \end{equation}
The maps $\psi_1,\cdots,\psi_n$ are clearly continuous on $I$ and
piecewise log-affine on $I\setminus\{y_0\}$. By Theorem~\ref{sec:spectr-theory-sense-1}, we should have $\Spe(y)=\bigcup_{i=1}^n\{
x_{a,\psi_i(y)}\}+\ZZ_p$. By choosing $\psi_{i_1},\cdots,\psi_{i_\nu}$
to be distinct and $\{\psi_{i_1},\cdots,\psi_{i_\nu}\}=\cup_{i=1}^n\{ \psi_i\}$, $\{x_{a,\psi_{i_1}(y)},\cdots,
x_{a,\psi_{i_\nu}(y)}\}$ has minimal cardinality.

Suppose now that $z\in \AF\setminus F$. Let $c_1,\cdots,c_m$ such that
\[\RS_n(.,(\F(.),\nabla_f-c_i))>\RS_n(y_0,(\F(y_0),\nabla_f-c_i)),\]
$z=x_{c_i,r_F(z)}$ and $\delta(c_i-c_j)=|c_i-c_j|=r_F(z)$ for $i\ne j$. Then
by Theorem~\ref{sec:radii-conv-line}, we can choose an interval $I$ neighborhood of $y$ in
$(c,x]\cap X$ (resp. $[y_0,x]$ if $y_0=x$), such that for all $i$ and
$k$, $\RS_k(.,(\F(.),\nabla_f-c_i))$ is log-affine on $I$. Hence, by
Lemma~\ref{sec:vari-spectr-case-1}, $m$ cannot
exceed $n$. We choose $m$ to be maximal. We set $k_{c_i}=\min_{1\leq j\leq n}\{j;\;
 \RS_j(.,(\F(.),\nabla_f-c_i))>\RS_n(y_0,(\F(y_0),\nabla_f-c_i))
 \text{ on } I\}$. Then by Lemma~\ref{sec:vari-spectr-case-1} we have the decomposition 
\begin{equation}\label{eq:82}
  (\F(y),\nabla_f)=\left(\bigoplus_{i=1}^m(M_{c_i}(y),\nabla_{c_i}(y))\right)\oplus (M_m(y),\nabla_m(y)),
 \end{equation}
with
$\RS_k(y,(M_{c_i}(y),\nabla_{c_i}(y)-c_i))=\RS_{k_{c_i}+k-1}(y,(\F(y),\nabla_f-c_i))$
and $\dim_{\Hy} M_{c_i}=n-k_{c_i}+1$. We set
$n_m:=\dim_{\Hy}M_m(y)$. Note that for $e\in F$, with
$\delta(e-c_{i_0})>r_F(z)$ for some $i_0$. Since
$\RS_k(.,(\F(.),\nabla_f-c_{i_0}))$ are log-affine on $I$ and
$\RS_k(.,(\F,\nabla_f-e))$ is continue, for all $y\in I$ we have
\begin{equation}
 \label{eq:91}
\RS_k (y,(\F(y),\nabla_f-e))=\min(\RS_k
(y,(\F(y),\nabla_f-c_{i_0})),\cR_{c_{i_0}-e}(y))\leq \cR(r_F(z)). 
\end{equation}
Therefore,
since $m$ is maximal for all $e\in F$ we have
$\RS_{n_m}(y,(M_m(y),\nabla_m-e))\leq \cR(r_F(z))$.\\
Hence, for each $c_i$, for all $e\in
F$ such that $|e-c_i|<r_F(z)$ (which implies in particular $\delta(e-c_i)<r_F(z)$), and all $y\in I$ we have
\[\RS_k(y,(M_{c_i}(y),\nabla_{c_i}(y)-e))=\RS_{k_{c_i}+k-1}(y,(\F(y),\nabla_{f}-e)).\]
This implies that $\RS_k(.,(M_{c_i}(.),\nabla_{c_i}(.)-e))$ is
continuous on $I$, and log-affine if $e=c_i$. Therefore, for all $y\in I$ and for all $e\in F$ we have
\begin{equation}
 \label{eq:83}
\RS_k(y,(M_{c_i}(y),\nabla_{c_i}(y)-e))=\min(\RS_k(y,(M_{c_i}(y),\nabla_{c_i}(y)-c_i)),\cR_{c_i-e}(y)). 
\end{equation}

Then by setting 
\begin{equation}
\label{eq:84}
 \phi_{i,k}(y)= \begin{cases}

    \frac{|p|^l\omega}{\RS_k(y,(M_{c_i}(y),\nabla_{c_i}(y)-c_i))^{p^l}}& \text{
                               if }
                               \omega^{\fra{p^{l-1}}}\leq
                               \RS_k(y,(M_{c_i}(y),\nabla_{c_i}(y)-c_i))\leq
                               \omega^{\fra{p^l}},
                               \text{
                               with }
                               l\in\NN\setminus\{0\},\\
    \frac{\omega}{\RS_k(y,(M_{c_i}(y),\nabla_{c_i}(y)-c_i))} & \text{ if }
                          \RS_k(y,(M_{c_i}(y),\nabla_{c_i}(y)-c_i))\leq \omega,
   \end{cases}  
\end{equation}
 and Theorem~\ref{sec:spectr-theory-sense-1}, we obtain
 $\Sigma_{\nabla_{c_i}(y),F}(\Lk{M_{c_i}(y)})=\bigcup_{k=1}^{n-k_{c_i}+1}\{
 x_{c_i,\phi_{i,k}(y)}\}+\ZZ_p$, with $\phi_{i,k}$ piecewise log-affine on
 $I$, and $\bigcup_{k=1}^{n-k_{c_i}+1}\{
 x_{c_i,\phi_{i,k}(y)}\}$ has minimal cardinality. Then there exists a subset $S\subset\{1,\cdots, n-k_{c_i}+1\}$,
 such that for all $y\in I$ we have $\bigcup_{k=1}^{n-k_{c_i}+1}\{
 x_{c_i,\phi_{i,k}(y)}\}=\bigcup_{k\in S}\{
 x_{c_i,\phi_{i,k}(y)}\}$, and for all $y\in I\setminus\{y_0\}$ $\car (S)=\car(\bigcup_{k=1}^{n-k_{c_i}+1}\{
 x_{c_i,\phi_{i,k}(y)}\})$.\\

We set $c_0=c_{i_0}$ for some $i_0$. Then for all $e\in F$ and $y\in I$, we have
 \[\RS_k(y,(M_m(y),\nabla_m(y)-e))=\RS_k(y,(\F(y),\nabla_f-e))\]
 for $k\in\{1,\cdots, n_m\}$. Hence, for all $e\in F$ and all $k$, the
 map $\RS_k(.,(M_m(.),\nabla_m(.)-e))$ is continuous on $I$, and
 log-affine on $I$ if $|e-c_0|\leq r_F(z)$. Therefore, for all $y\in
 I$ and for all $e\in F$ we have
 \[\RS_k(M_m(y),\nabla_m(y)-e)=\min(\RS_k(M_m(y),\nabla_m(y)-c_0),\cR_{e-c_0}(y)). \]
Then by setting 
\begin{equation}
\label{eq:85}
 \phi_{0,k}(y)= \begin{cases}

    \frac{|p|^l\omega}{\RS_k(y,(M_m(y),\nabla_{m}(y)-c_0))^{p^l}}& \text{
                               if }
                               \omega^{\fra{p^{l-1}}}\leq
                               \RS_k(y,(M_{m}(y),\nabla_{m}(y)-c_0))\leq
                               \omega^{\fra{p^l}},
                               \text{
                               with }
                               l\in\NN\setminus\{0\},\\
    \frac{\omega}{\RS_k(y,(M_{m}(y),\nabla_{m}(y)-c_0))} & \text{ if }
                          \RS_k(y,(M_{m}(y),\nabla_{m}(y)-c_0))\leq \omega,
   \end{cases}  
\end{equation}
 and Theorem~\ref{sec:spectr-theory-sense-1}, we obtain
 $\Sigma_{\nabla_{m}(y),F}(\Lk{M_{m}(y)})=\bigcup_{k=1}^{n_m}\{
 x_{c_0,\phi_{0,k}(y)}\}+\ZZ_p$, with $\phi_{0,k}$ log-affine on
 $I$, and $\bigcup_{k=1}^{n_m}\{
 x_{c_0,\phi_{0,k}(y)}\}$ has minimal cardinality. Then there exists a subset $S\subset\{1,\cdots, n_m\}$,
 such that for all $y\in I$ we have $\bigcup_{k=1}^{n_m}\{
 x_{c_0,\phi_{0,k}(y)}\}=\bigcup_{k\in S}\{
 x_{c_0,\phi_{0,k}(y)}\}$, and for all $y\in I\setminus \{y_0\}$ we have $\car (S)=\car(\bigcup_{k=1}^{n_m}\{
 x_{c_0,\phi_{0,k}(y)}\})$. By \cite[Corollary~5.47]{Azz24} we obtain
 the result.
 \end{proof}

\begin{Pro}\label{sec:vari-spectr-case-4}
Let $y_0\in (c,x]\cap X$. Let $z_1,\cdots, z_\nu\in \AF\setminus F$,
$a_1,\cdots,a_\mu\in F$, with $(\nu,\mu) \ne (0,0)$, such that
$\Sigma_{\F,c,f}(y_0)=\{z_1,\cdots,z_\nu,a_1,\cdots,a_\mu\}+\ZZ_p$,
and $\{z_1,\cdots,z_\nu,a_1,\cdots, a_\mu\}$ has minimal
cardinality. Then in
$(c,y_0]\cap X$
(resp. $[y_0,x]$ if $y_0\ne x$) there exists an interval neighborhood $I$ of $y_0$ such that
\begin{enumerate}
 
\item for each $z_i$ there exist $c_{i,1},\cdots, c_{i,{n_i}}\in F$
 ($c_{i,j}$ are not necessarily distinct)
 and $\phi_{i,1},\cdots, \phi_{i,{n_i}}:I\to \R+$ piecewise log-affine on $I$
 such that $z_i=x_{c_{i,j}, \phi_{i,j}(y_0)}$,
\item for each $a_i$ there exist
 $\psi_{i,1},\cdots,\psi_{i,m_i}:I\to\R+$ continuous on $I$ and piecewise log-affine on $I\setminus\{y_0\}$ such that
 $\psi_{i,j}(y_0)=0$ (i.e. $a_i=x_{a_i,\psi_{i,j}(y_0)}$),
\item for all $y\in I$ we have
 \[\Sigma_{\F,c,f}(y)=\{x_{c_{1,1},\phi_{1,1}(y)},\cdots,
  x_{c_{\nu,n_\nu},\phi_{\nu,n_\nu}(y)},
  x_{a_1,\psi_{1,1}(y)},\cdots,x_{a_\mu,\psi_{\mu,m_\mu}(y)}\}+\ZZ_p,\]
 \item the set $\{x_{c_{1,1},\phi_{1,1}(y)},\cdots,
  x_{c_{\nu,n_\nu},\phi_{\nu,n_\nu}(y)},
  x_{a_1,\psi_{1,1}(y)},\cdots,x_{a_\mu,\psi_{\mu,m_\mu}(y)}\}$
  has minimal cardinality.
 \end{enumerate}
\end{Pro}

\begin{proof}
By Theorem~\ref{sec:spectr-theory-sense-1}, we have
$\Spe(y_0)=\{z_1(y_0),\cdots, z_{\mu(y_0)}(y_0)\}+\ZZ_p$ with
$\{z_i(y_0)\}+\ZZ_p\cap \{z_j(y_0)\}+\ZZ_p=\emptyset$ if $i\ne j$. We may assume that
$r_F(z_1(y_0))\leq\cdots\leq r_F(z_{\mu(y_0)} (y_0))$. Since $z_i(y_0)$ are
points not of type (4), we can find $b_i\in F$ such that
$|z_i(y_0)-b_i|=r_F(z_i(y_0))$. Then we have
$|z_i(y_0)-b_i|<|z_j(y_0)-b_i|$ for $i<j$. Let $(M_{z_i(y_0)}, \nabla_{z_i(y_0)})$ as in
decomposition in Theorem~\ref{sec:spectr-theory-sense-1}, we set $m(z_i(y_0)):=\dim_{\h{y}}
M_{z_i(y_0)}$.

We prove the statement by induction on
$\mu(y_0)$. For $\mu(y_0)=1$, by Lemma~\ref{sec:vari-spectr-case-2}
the statement holds on $y_0$. Suppose now that the statement holds for
$\mu(y_0)<\mu$. We set $i_1=n-m(z_1(y_0))+1$. Then the index $i_1$ separates the
radii of $(\F,\nabla_f-b_1)$ on $y_0$. Let $U$ be an open neighborhood
of $y_0$ such that, the index $i_1$ separates the
radii of $(\F,\nabla_f-b_1)$ on $U$. By
Theorem~\ref{sec:robb-decomp-spectr-1}, we have
\begin{equation}
 \label{eq:48}
 0\to (\F_{|_U\; \geq i_1},\nabla_f-b_1)\to (\F_{|_U},\nabla_f-b_1)\to
 (\F_{|_U\; < i_1},\nabla_f-b_1)\to 0.
\end{equation}

Let $(\F_1,\nabla_1)$ and $(\F_2,\nabla_2)$ such that,
$(\F_1,\nabla_1-b_1)= (\F_{|_U\; \geq i_1},\nabla_f-b_1)$ and
$(\F_2,\nabla_2-b_1)=(\F_{|_U\; < i_1},\nabla-b_1)$. By construction
we have $\spe{\F_1}(y_0)=\{z_1(y_0)\}+\ZZ_p$ and
$\spe{\F_2}(y_0)=\{z_2(y_0),\cdots, z_\mu(y_0)\}+\ZZ_p$. Since, for
all $y\in (c,x]\cap U$ we have
$\Spe(y)=\spe{\F_1}(y)\cup\spe{\F_2}(y)$ (cf. \cite[Corollary~5.47]{Azz24}). By the induction hypothesis
we obtain the result. 
\end{proof}

\begin{rem}
We may ask whether we can choose the same $c_{i,j}$ in both sides out of
$y_0$ in $(c,x]\cap X$. In general it is not possible; consider, for example, the equation of rank one
 $\dT+\frac{a}{T(c-T)}$ defined over $\disf{0}{1}\setminus\{0,c\}$,
with $y_0=x_{0,|c|}$, $|c|<1$ and $|a|>1$. Indeed, in the branch $(0,x_{0,1}]$,
letting $\epsilon>0$ be as small as necessary, 
for $|c|-\epsilon< r\leq|c|$ we have
$\Sigma_{\F,0,T}(x_{0,r})=\{x_{{a}{c\-1},{|ac^{-2}|r}}\}+\ZZ_p=\{x_{{a}{c\-1},{|ac^{-2}|r}}\}$, and for
$|c|\leq r<|c|+\epsilon$, we have
$\Sigma_{\F,0,T}(x_{0,r})=\{x_{0,|a|r\-1}\}+\ZZ_p=\{x_{0,|a|r\-1}\}$. Note that in this
example even $\RS_1$ is not log-concave.
\end{rem}

\begin{Lem}\label{sec:vari-spectr-case-3}
 Let $I$ be an interval, $a\in F$ and $\phi:I\to \R+$ be a continuous function. Then the
 following map is continuous
 \begin{equation}
  \label{eq:14}
  \Fonction{\Psi :I}{\cK(\AF)}{r}{\{x_{a,\phi(r)}\}+\ZZ_p.}
 \end{equation}
\end{Lem}

\begin{proof}
Let $r_0\in I$. We distinguish two cases, one when $\phi(r_0)=0$ and 
the other one when $\phi(r_0)>0$.
\begin{itemize}
\item {\bf Case} $\phi(r_0)=0$: In this case the sets of the form
 $(U_n,\{\diso{a+i}{|p|^n}\}_{0\leq i \leq p^{n+1}-1})$, with $n\in \NN$
 and $U_n=\bigcup_{i=0}^{p^{n+1}-1}\diso{a+i}{|p|^n}$, form a
 neighborhood basis of $\Psi(r_0)=a+\ZZ_p$. Since $\phi$ is
 continuous, for all $n\in \NN$, there exists $I_n$ a
 neighborhood of $r_0$ such that, for all $r\in I_n$ we have
 $\phi(r)<|p|^n$. Hence, for all $r\in I_n$ and $0\leq i\leq p^{n+1}-1$ we have
 $x_{a+i,\phi(r)}\in \Psi(r)\cap \diso{a+i}{|p|^n}$. Since
 $a+\ZZ_p\subset U_\epsilon$, we have $\Psi(r)\subset
 U_n$. Therefore, we have $\Psi(r)\in (U_n,
 \{\diso{a+i}{|p|^n}\}_{0\leq i\leq p^{n+1}-1})$.
\item {\bf Case} {$\phi(r_0)>0$}: In this case, if $\phi(r_0)>1$ then $\Psi(r_0)=\{x_{a,\phi(r_0)}\}$. The sets of the form ${\couo{c}{\epsilon_1}{\epsilon_2},\{\couo{c}{\epsilon_1}{\epsilon_2}\})}$, 
 with $1<\epsilon_1<\phi(r_0)<\epsilon_2$ and $|c-a|\leq \phi(r_0)$,
 form a neighborhood basis of $\Psi(r_0)$. Since $\phi$ is
 continuous, for all $1<\epsilon_1<\phi(r_0)<\epsilon_2$, there exits
 $I_{\epsilon_1,\epsilon_2}$ a neighborhood of $r_0$ such that, for
 all $r\in I_{\epsilon_1,\epsilon_2}$ we have 
 $\epsilon_1<\phi(r)<\epsilon_2$. Hence, for all $r\in
 I_{\epsilon_1,\epsilon_2}$ we have $\Psi(r)=\{x_{a,\phi(r)}\}$ and 
 $\Psi(r)\subset\couo{c}{\epsilon_1}{\epsilon_2}$. Suppose now that
 $|p|^{n+1}<\phi(r_0)\leq |p|^{n}$. We have 
 $\Psi(r_0)=\{x_{a,\phi(r_0)},\cdots, x_{a+p^{n+1}-1,\phi(r_0)}\}$. The sets of
 the form
 $(U_{\epsilon_1,\epsilon_2},\{\couo{c_i+i}{\epsilon_1}{\epsilon_2}\}_{0\leq
  i\leq p^{n+1}-1}) $, with
 $U_{\epsilon_1,\epsilon_2}=\bigcup_{i=0}^{p^{n+1}-1}\couo{c_i+i}{\epsilon_1}{\epsilon_2}$,
 $|p|^{n+1}< \epsilon_1<\phi(r_0) <\epsilon_2$
  and $|c_i-a|\leq \phi(r_0)$, form a
 neighborhood basis of $\Psi(r_0)$. Since $\phi$ is
 continuous, for all $|p|^{n+1}<\epsilon_1<\phi(r_0)<\epsilon_2$, there exits
 $I_{\epsilon_1,\epsilon_2}$ a neighborhood of $r_0$ such that, for
 all $r\in I_{\epsilon_1,\epsilon_2}$ we have 
 $\max(\max_i|c_i-a|,\epsilon_1)<\phi(r)<\epsilon_2$. This implies that for all $r\in
 I_{\epsilon_1,\epsilon_2}$, $\Psi(r)=\{x_{a,\phi(r)},\cdots,
 x_{a+p^{n+1}-1,\phi(r)}\}$. Hence, for all $1\leq i\leq p^{n+1}-1$
 we have $\Psi(r)\cap
 \couo{c_i+i}{\epsilon_1}{\epsilon_2}=\{x_{a+i,\phi(r)}\}$ and
 $\Psi(r)\in (U_{\epsilon_1,\epsilon_2},\{\couo{c_i+i}{\epsilon_1}{\epsilon_2}\}_{0\leq
  i\leq p^{n+1}-1}) $. 

\end{itemize}
\end{proof}

\begin{Theo}
 The following map is continuous
 \begin{equation*}
  \Fonction{\Sigma_{\F,c,f}: (c,x]\cap X}{\cK(\AF)}{y}{\Sigma_{\nabla_{d_f},F}(\Lk{\F(y)})}. 
 \end{equation*}
\end{Theo}

\begin{proof}
 By Proposition~\ref{sec:vari-spectr-case-4} and
 Lemma~\ref{sec:vari-spectr-case-3} we obtain the result.
\end{proof}

\begin{rem}
 Since in each branch we do not choose the same derivation, we proved
 the continuity only on each skeleton of an annulus. It is worth
 specifying that in general the continuity of the spectrum does not hold
 a neighborhood of point, even if we fix the same derivation in the
 whole neighborhood. For example this is not true even for the trivial
 differential equation, see \cite[Theorem~3.20]{Azz24}.
\end{rem}

\begin{cor}\label{sec:vari-spectr-case-5}
 Let $(\F,\nabla)$ be a differential equation over an open disk
 $D$. Suppose that all the radii of convergence are constant
 (respectively identically equal to $R_1,\cdots, R_n$). Then for all
 $x\in D_{[2,3]}$, for all $c\in F$ such that $x=x_{c,r}$, and for all $f$ a coordinate function defined in
an affinoid neighborhood $V$ of $x$ in $D$ (in the
meaning where the induced map $V\to \AF$ by $F[T]\to \cO(V)$,
$T-c\mapsto f$ is a closed immersion with $f(V)=V$), such that for all $y\in
(c,x]\cap V$ we have $f(y)=y$, we have
 \begin{equation}
  \label{eq:51}
   \Fonction{\Sigma_{\F_{|_V},c,f}: (c,x]\cap V}{\cK(\AF)}{y}{\bigcup_{i=1}^n\{x_{0,\phi_i(y)}\}+\ZZ_p,}
  \end{equation}
  with
  \begin{equation}
   \label{eq:52}
   \phi_i(y)=
   \begin{cases}
    0& \text{ if } r_F(y) \leq R_i,\\
    \frac{|p|^l\omega\cdot r_F(y)^{p^l}}{R_i^{p^l}} & \text{ if }
                         \omega^{\fra{p^{l-1}}}\cdot
                         r_F(y)\leq
                         R_i \leq
                         \omega^{\fra{p^{l}}}\cdot
                         r_F(y),
                         \text{ with }
                         l\in\NN\setminus\{0\},\\
    \frac{\omega\cdot r_F(y)}{R_i} & R_i\leq \omega\cdot r_F(y).\\
   \end{cases}
  \end{equation}
\end{cor}

\subsection{Approximating the spectrum, by
 approximating the connection}
In this part we prove that we can approximate the spectrum of a given
connection, by approximating this connection. 
\begin{Lem}\label{sec:link-betw-spectr-3}
  Let $\Omega\in E(F)$, let $d:\Omega\to \Omega$ be a bounded $F$-linear derivation, such that $d\ne 0$. Let $\{(M,\nabla_l)\}_{l\in\NN}$ (resp. $(M,\nabla)$) be
  a family of differential modules (resp. differential module) over
  $(\Omega,d)$, such that $\nabla_l\to\nabla$ in $\Lk{M}$. Then there exists a cyclic vector $m$ of $(M,\nabla)$
  for which we can construct a sequence of vectors $(m_l)_{l\in \NN}$
  such that, $m_l$ is a cyclic vector of $(M,\nabla_l)$ and
  $m_l\overset{l\to \infty}{\to} m$. 
 \end{Lem}

 \begin{proof}
  Let $e_1,\dots, e_n$ be a basis of $M$ as an $\Omega$-vector space.
  Recall that for $f\in \Omega$, $m$ is a cyclic vector of $(M,\nabla)$ as a
  differential module over $(\Omega,d)$, if and only if $m$ is a cyclic vector of $(M,f\nabla)$ as a differential module
  over $(\Omega,f.d)$. Let $f\in \Omega$ such that $df\ne 0$. We
  construct $m$ (resp. $m_l$) as follows. For $i,j\geq 0$, we set
  \begin{equation}
   \label{eq:98}
   c(i,j)=\sum_{k=0}^j(-1)^k\binom{j}{k}(\fra{df}\nabla)^k(e_{i+j+1-k}),
   \quad c_l(i,j)=\sum_{k=0}^j(-1)^k\binom{j}{k}(\fra{df}\nabla_l)^k(e_{i+j+1-k}),
  \end{equation}
  and
  \begin{equation}
   \label{eq:99}
   c_i(t)=\sum_{j=0}^{n-1}\frac{t^j}{j!}c(i,j), \quad c_{i,l}(t)=\sum_{j=0}^{n-1}\frac{t^j}{j!}c_l(i,j).
  \end{equation}
  
  Let $P(t)=\det(c_0(t),\cdots,c_{n-1}(t))$ and
  $P_l(t)=\det(c_{0,l}(t),\cdots,c_{n-1,l}(t))$ in the basis
  $e_1,\cdots,e_n$ of $M\ot_{\Omega}\Omega [t]$. Let
  $\{z_1,\cdots,z_{n(n-1)}\}$
  (resp. $\{z_{1,l},\cdots,z_{n(n-1),l}\}$) be the multiset of roots
  of $P(t)$ (resp. $P_l(t)$).

  Since $\nabla_l\to
  \nabla$, for $P(t)=\sum_{s=0}^{n(n-1)}\alpha_s t^s$ and
  $P(t)=\sum_{s=0}^{n(n-1)}\alpha_{s,l} t^s$, we have $\alpha_{s,l}\to
  \alpha_s$ for all
  $s\in\{0,\cdots,n(n-1)\}$. Hence, we have (cf. Corollary~\ref{sec:continuity-results-4}) 
  \begin{equation}
   \label{eq:100}
  \forall \epsilon>0;\, \exists \ell\in\NN;\, \forall l\geq \ell\text{
   we have } \{z_{1,l},\cdots,z_{n(n-1),l}\}\subset\bigcup_{s=0}^{n(n-1)}\disco{\wac}{z_s}{\epsilon}. 
 \end{equation}

 Since $F$ is of characteristic zero, we can choose $a\in F$ such that
 $P(f-a)\ne 0$. Let $\epsilon_0=\min_s|z_s-(f-a)|$. By \eqref{eq:100},
 for $\epsilon_0$
 there exists $l_0\in \NN$ such that $\forall l\geq l_0$ we have
 $\{z_{1,l},\cdots,z_{n(n-1),l}\}\subset\bigcup_{s=0}^{n(n-1)}\disco{\wac}{z_s}{\epsilon_0}$. Therefore, 
 $f-a\not\in\{z_{1,l},\cdots,z_{n(n-1),l}\}$ for all $\forall l\geq l_0$. Hence we have
 $P_l(f-a)\ne 0$ for all $l\geq l_0$.

 Since
 \begin{equation}
  \label{eq:18}
  \nabla^i(c_0(f-a))=c_i(f-a), \quad \nabla_l^i(c_{0,l}(f-a))=c_{i,l}(f-a)
 \end{equation}
(see proofs of \cite[Theorems~5.7.2, 5.7.3]{Ked}), $c_0(f-a)$ is a
cyclic vector of $(M,\fra{df}\nabla)$, and $c_{0,l}(f-a)$ is a cyclic
vector of $(M,\fra{df}\nabla_l)$ for all $l\geq l_0$. We set
$m=c_0(f-a)$, $m_l=c_{0,l}(f-a)$ for $l\geq l_0$, and for $l< l_0$ we
set $m_l$ to be a cyclic vector of $(M,\fra{df}\nabla_l)$. 
  The vector $m$ (resp. $m_l$) is a cyclic vector of $(M,\fra{df}\nabla)$
  (resp. $(M,\fra{df}\nabla_l)$), hence a cyclic
  vector of $(M,\nabla)$ (resp. $(M,\nabla_l)$). Since $\nabla_l\to
  \nabla$, by construction we have $m_l\to m$.
 \end{proof}

Let $X$ be an affinoid domain of $\AF$ and $x=x_{c,r}\in X_{[2,3]}$. For $f\in
\Hx$, such that $f^{\sharp}:\Hx\to \Hx$, $T-c\mapsto f$ is a well defined isomorphism in the category $E(F)$, we set $d_f:=f\D{f}$.

 \begin{Theo}\label{sec:link-betw-spectr-2}
Let $(M,\nabla)$ be a
 differential module over $(\Hx,d_f)$, and $\{(M,\nabla_l)\}_{l\in
  \NN}$ be a family of differential modules over $(\Hx,d_f)$, such that $\nabla_l\to
 \nabla$ in $\Lk{M}$. Let
 $z_1,\cdots,z_\nu$ be points of type (2) or (3) and $a_1,\cdots,
 a_\mu\in F$ such that,
 \begin{equation}
  \label{eq:34}
  \Sigma_{\nabla,F}(\Lk{M})=\{z_1,\cdots,z_\nu,a_1,\cdots, a_\mu\}+\ZZ_p,
 \end{equation}
 with $\{z_1,\cdots,z_\nu,a_1,\cdots, a_\mu\}$ has minimal
 cardinal, and $(\nu,\mu)$ is not equal to $(0,0)$. Then there exists $l_0\in \NN$
 such that for all $l\geq l_0$ we have
 \begin{equation}
  \label{eq:42}
  (M,\nabla_l)=\bigoplus_{i=1}^\nu(M_{l,z_i},\nabla_{l,z_i})\oplus \bigoplus_{k=1}^\mu(M_{l,a_k},\nabla_{l,a_k}),
 \end{equation}
 where
 \begin{enumerate}
 \item $\dim_{\Hx} M_{l,z_i}=\dim_{\Hx} M_{z_i}$ and $\dim_{\Hx} M_{l,a_k}=\dim_{\Hx}
  M_{a_k}$;
 \item $\Sigma_{\nabla_{l,z_i},F}(\Lk{M_{l,z_i}})=\{z_i\}+\ZZ_p$;
 \item $\forall \epsilon>0$, $\exists l_\epsilon\geq l_0$; $\forall
  l\geq l_\epsilon\implies
  \Sigma_{\nabla_{l,a_k},F}(\Lk{M_{l,a_k}})\subset\bigcup_{j\in \NN}\diso{a_k+j}{\epsilon}$. 
 \end{enumerate}
 \end{Theo}
 
 \begin{proof}
 Let $\phi: \AF\to \AF$, $T-c\mapsto (T-c)^p$. We set
 $(M_k,\nabla_k)=(\phi^k\circ f)_*(M,\nabla)$ and
 $(M_k,\nabla_{k,l})=(\phi^k\circ f)
 _*(M,\nabla_l)$. Let $\mathscr{B}_k=\{e^k_1,\cdots, e^k_{p^kn}\}$ be
 a basis as $\Hx$-vector space of $M_k$. Let $m_k$ (resp. $m_{k,l}$)
 be a cyclic vector as in Lemma~\ref{sec:link-betw-spectr-3} of
 $(M_k,\nabla_k)$ (resp. $(M_k,\nabla_{k,l})$). Let $U_k$
 (resp. $U_{k,l}$) be the matrix change from $\mathscr{B}_k$ to
 $\{m_k,\nabla_k(m_k),\cdots, \nabla_k^{p^kn-1}(m_k)\}$
 (resp. $\{m_{k,l},\nabla_{k,l}(m_{k,l}),\cdots,
 \nabla_{k,l}^{p^kn-1}(m_{k,l})\}$). Hence, by Lemma~\ref{sec:link-betw-spectr-3} we have
 $U_{k,l}\overset{l\to\infty}{\to} U_k$
 (resp. $U_{k,l}^{-1}\overset{l\to \infty}{\to} U_k^{-1}$). Since
 $\nabla_l\to \nabla$, then $\nabla_{k,l}\to \nabla_k$ (they are the
 same as bounded $F$-linear operators). Therefore, we have
 $U^{-1}_{k,l}\nabla_{k,l}U_{k,l}\to U_k^{-1}\nabla_k U_k$. This means
 that if $P_k (p^k(T-c)^p\dT)$ (resp. $P_{k,l}(p^k(T-c)\dT)$) is the differential polynomial
 associated to $m_k$ (resp. $m_{k,l}$), then the coefficients of
 $P_{k,l}(S)$ converge to the coefficients of $P_k(S)$. Let $n_i=\dim_{\Hx}
 M_{a_i}$. Let $k_0\in \NN$ such that,
 \begin{itemize}
 \item $\RS_{n-n_i}(x,(M,\nabla-a_i))<\omega^{\frac{1}{p^{k_0}}}$ for
  all $i\in \{1,\cdots,\mu\}$;
 \item $r_F(z_i)> |p|^{k_0}$ for all $i\in\{1,\cdots,\nu\}$;
 \item $\delta(a_i)>|p|^{k_0}$ for all $i\in\{1,\cdots,\mu\}$ with $a_i\ne 0$;
 \item $\delta(a_i-a_j)>|p|^{k_0}$ for all $i,j\in\{1,\cdots,\mu\}$ with
  $i\ne j$. 
 \end{itemize}
For $z\in \AF$, we set
$\cD(z)=\pi^{-1}_{\widehat{\Hx^{alg}}/F}(z)\setminus
\{z_{\widehat{\Hx^{alg}}}\}$
(cf. Definition~\ref{sec:shape-quasi-smooth-3}). The polynomial $P_{k_0}(S)$ admits $(\dim_{\Hx}
 M_{z_i})p^{k_0}$ roots counted with multiplicity (as a commutative
 polynomial) in $\bigcup_{j=0}^{p^{k_0}-1}\cD(z_i+j)$
 (cf. \cite[Theorem~4.21]{Azz24}). Since $P_{k_0,l}(S)\to P_{k_0}(S)$, there
 exists $l_1\in\NN$ such that, for all $l\geq l_1$ the polynomial $P_{k_0,l}(S)$ admits $(\dim_{\Hx}
 M_{z_i})p^{k_0}$ roots counted with multiplicity (as a commutative
 polynomial) in $\bigcup_{j=0}^{p^{k_0}-1}\cD(z_i+j)$. Hence, for all
 $l\geq l_1$ we have $\{z_1,\cdots,z_\nu\}+\ZZ_p\subset
 \Sigma_{\nabla_l,F}(\Lk{M})$. Moreover, if $(M_{l,z_i},\nabla_{l,z_i})$ is the
 sub-differential module corresponding to the element of the spectrum
 $z_i$, then $\dim_{\Hx}M_{l,z_{i}}=\dim_{\Hx}M_{z_i}$. On the other hand,
 $P_k(S-a_i)$ admits $n_ip^{k_0}$ roots with absolute value equal or less then
 $|p|^{k_0}$. Hence, there exists $l_{k_0}\geq l_1$ such that for all
 $l\geq l_{k_0}$ the polynomial $P_{k,l}(S-a_i)$ admits $n_ip^{k_0}$ roots
 with absolute value equal or less then
 $|p|^{k_0}$, for all $i\in \{1,\cdots,\mu\}$. Hence, for all $l\geq l_{k_0}$ we have
 $\RS_{n-n_i+1}(x,(M,\nabla_l-a_i))\geq\omega^{\fra{p^{k_0}}}$, and $\RS_{n-n_i}(x,(M,\nabla_l-a_i))<\omega^{\fra{p^{k_0}}}$. Therefore,
 for all $l\geq l_{k_0}$ and for all $i\in\{1,\cdots,\mu\}$ we have the decomposition
 \begin{equation}
  \label{eq:47}
  (M,\nabla_l)=(M_{l,i},\nabla_{l,i})\oplus (M_{l,a_i},\nabla_{l,a_i}),
 \end{equation}
 with $\dim_{\Hx}M_{l,a_i}=n_i$,
 $\RS_1(x,(M_{l,a_i},\nabla_{l,a_i}-a_i))\geq\omega^{\fra{p^{k_0}}}$
 and $\RS_{n-n_i}(x,(M_{l,i},\nabla_{l,i}-a_i))<\omega^{\fra{p^{k_0}}}$. We
 need to prove now that there exists $l_0\geq l_{k_0}$ such that, for
 all $i\ne j$ and $l\geq l_0$ we have $(M_{l,a_i},\nabla_{l,a_i})\cap
 (M_{l,a_j},\nabla_{l,a_j})=\{0\}$ . Let
 $(U,\{U_{z_1},\cdots,U_{z_\nu},U_{a_1},\cdots,U_{a_\mu}\})$ be an open
 neighborhood of $\Sigma_{\nabla,F}(\Lk{M})$ (cf. \eqref{eq:1}) such that,
 \begin{itemize}
 \item $U_{z_i}$ (resp. $U_{a_i}$) is an open neighborhood of
  $\{z_i\}+\ZZ_p$ (resp.$\{a_i\}+\ZZ_p$);
 \item $U_{z_i}\cap U_{z_j}=\emptyset$ for all
  $i,j\in\{1,\cdots,\nu\}$ with $i\ne j$;
  \item $U_{z_i}\cap U_{a_j}=\emptyset$ for all
  $i\in\{1,\cdots,\nu\}$ and $j\in\{1,\cdots,\mu\}$;
 \item $U_{a_i}=\bigcup_{j\in\NN}\diso{a_i+j}{|p|^{k_0}}$ for all $i\in\{1,\cdots,\mu\}$.
 \end{itemize}
 By the previous hypothesis we have $U_{a_i}\cap U_{a_j}=\emptyset$
 for all $i,j\in\{1,\cdots,\mu\}$ with $i\ne j$. Since by
 Theorem~\ref{sec:continuity-results-3} we have $\Sigma_{\nabla_l,F}(\Lk{M})\to\Sigma_{\nabla,F}(\Lk{M})$, there exists $l_0\geq
 l_{k_0}$ such that, for all $l\geq l_0$ we have
 $\Sigma_{\nabla_l,F}(\Lk{M})\in
 (U,\{U_{z_1},\cdots,U_{z_\nu},U_{a_1},\cdots,U_{a_\mu}\})$. Hence, by
 the previous hypothesis we
 should have $\Sigma_{\nabla_{l,a_i},F}(\Lk{M_{l,a_i}})\subset U_{a_i}$, which means
 $\Sigma_{\nabla_{l,a_i},F}(\Lk{M_{l,a_i}})\cap\Sigma_{\nabla_{l,a_j},F}(\Lk{M_{l,a_j}})=\emptyset$
 for $i\ne j$. Therefore, for all $l\geq l_0$ we have $(M_{l,a_i},\nabla_{l,a_i})\cap
 (M_{l,a_j},\nabla_{l,a_j})=\{0\}$. It is obvious that $(M_{l,z_i},\nabla_{l,z_i})\cap
 (M_{l,a_j},\nabla_{l,a_j})=\{0\}$ for all $i\in\{1,\cdots,\nu\}$ and $j\in\{1,\cdots,\mu\}$, and $(M_{l,z_i},\nabla_{l,z_i})\cap
 (M_{l,z_j},\nabla_{l,z_j})=\{0\}$ for all
  $i,j\in\{1,\cdots,\nu\}$ with $i\ne j$. Since $\sum_{i=1}^{\nu}
 \dim_{\Hx}M_{l,z_i}+\sum_{i=1}^{\mu}\dim_{\Hx}M_{l,a_i}=n$, we have
 \begin{equation}
  \label{eq:49}
  (M,\nabla_l)=\bigoplus_{i=1}^\nu(M_{l,z_i},\nabla_{l,z_i})\oplus \bigoplus_{k=1}^\mu(M_{l,a_k},\nabla_{l,a_k}).
 \end{equation}
 The fact that $\forall \epsilon>0$, $\exists l_\epsilon\geq l_0$
 such that $\forall
  l\geq l_\epsilon$ implies $\Sigma_{\nabla_{l,a_k},F}(\Lk{M_{l,a_k}})\subset\bigcup_{j\in
   \NN}\diso{a_i+j}{\epsilon}$, can be deduced directly from the
  fact that $\Sigma_{\nabla_l,F}(\Lk{M})\to \Sigma_{\nabla,F}(\Lk{M})$.
\end{proof}

\subsection{The link between the spectrum and the spectral radii of convergence, the case of a
 quasi-smooth curve} We consider now a general quasi-smooth connected
curve $X$ and a weak triangulation $\cS$ on $X$. The goal of this section is to provide a link between the
spectrum and the radii of convergence, as for the case of an affinoid
domain of $\AF$.

We know that if $x$ is a point of type (2) and lying
 in the interior of $X$, then there exists a one to one correspondence between rational points of the residual curve $\mathscr{C}_x$ and the
 elements $b\in T_xX$ (cf. \cite[Théorème~4.2.10]{Duc}). Since the
 spectrum is intrinsic to the differential equation restricted
 to the generic disk of the point $x$, we should be able to choose a
 derivation for each element of $\mathscr{C}_x$, for which the spectrum
 satisfy the same property as for the case of affinoid domains of
 $\AF$. Indeed, let $x\in X_{[2,3]}$, if $x$ is a point of type (2),
 we can fined a quasi-smooth algebraic curve $\mathscr{C}$, such that there exists an
 affinoid neighborhood of $x$ in $X$ isomorphic to an affinoid
 domain of $\mathscr{C}$, and $x$ lies in the interior of
 $\mathscr{C}$ (see the proof of \cite[Theorem~3.12]{np2}). Let $b\in
 T_x\mathscr{C}$. We can
 choose an affinoid neighborhood $Y_b$ of $x$, in $\mathscr{C}$, such that there exists an étale morphism $\Psi_b:Y_b\to W_b$ as in
 Theorem~\ref{sec:spectr-radii-diff} with $W_b\subset\AF$,
 $\gcd({\rm deg}(\Psi_b),p)=1$ and $\Psi_b\-1(\Psi_b(b))=\{b\}$. If $x$ is a point
 of type (3), since $x$ admits an
 affinoid neighborhood isomorphic to an affinoid domain of $\AF$, we
 just choose $Y\subset\AF$ such that $x$ is in the interior of
 $Y$. In this case, for each $b\in T_xY$, we set $\Psi_b:Y\to Y$ to be the identity.

We define $d_b$ as
 follows: 
Let $C_b$ (resp. $C_{\Psi(b)}$) be an open annulus and a section of
$b$ (resp. $\Psi_b(b)$). Let $c_b\in F$ such that
$C_{\Psi(b)}=\couo{c_b}{r_1}{r_2}$. Let $f_b=\Psi^{\#}_b(T-c_b)$, then
we consider
\begin{equation}
 \label{eq:17}
 d_b={\rm deg}(\Psi_b)\cdot f_b\cdot\Psi_b^*(\dT).
\end{equation}
By construction $d_b$ is a bounded derivation well defined on
$\cO_X(Y_b)$, which extends without any problem to $\h{y}$ for all $y\in Y_b$. 
In the following theorem, we establish the link
 between the spectrum and the spectral radii when we choose $d_b$ as derivation.   
\begin{Theo}\label{sec:link-betw-spectr}
Let $(M,\nabla)$ be a differential module $(\Hx, d_b)$. Let
$\phi_b:\AF\to \AF$, $T-c_b\mapsto (T-c_b)^p$, and we set $y^{p^l}$ to
be $\phi^l_b(y)$. We set $N=\deg(\Psi_b)$. We use notations of
\eqref{eq:93} and \eqref{eq:94}.
\begin{enumerate}

\item There exist $z_1,\cdots,z_{\nu}\in \AF\setminus F$ and
 $a_1,\cdots, a_\mu\in F$, such that
 \begin{equation}
  \Sigma_{\nabla,F}(\Lk{M})=\{z_1,\cdots,z_\nu, a_1,\cdots,
  a_\mu\}+\ZZ_p,
 \end{equation}
 where $z_i$ has the same type as $x$, and $(\nu,\mu)$ is not
 equal to $(0,0)$. 
\item We can choose $z_i$ and $a_j$ so that the set $\{z_1,\cdots,z_\nu, a_1,\cdots,
  a_\mu\}$ has minimal cardinality. Indeed it is enough to keep only
  $z_i$ and $a_j$ for which we have $\{z_i\}+\ZZ_p\cap
  \{z_{i'}\}+\ZZ_p=\emptyset$ and $\{a_j\}+\ZZ_p\cap \{a_{j'}\}+\ZZ_p=\emptyset$
  for $i\ne i'$ and $j\ne j'$.
 
 \item We choose $\{z_1,\cdots,z_\nu, a_1,\cdots,
  a_\mu\}$ to be minimal. Then we have a unique (up to an isomorphism) decomposition:
  \begin{equation}\label{eq:58}
   (M,\nabla)=\bigoplus_{i=1}^{\nu}(M_{z_i},\nabla_{z_i})\oplus \bigoplus_{j=1}^{\mu}(M_{a_j},\nabla_{a_j}),
  \end{equation}
  where $(M_{z_i},\nabla_{z_i})$ and $(M_{a_j},\nabla_{a_j})$ are
  differential modules over $(\Hx,d_b)$, such that $\Sigma_{\nabla_{z_i},F}(\Lk{M_{z_i}})=\{z_i\}+\ZZ_p$
  and $\Sigma_{\nabla_{a_j},F}(\Lk{M_{a_j}})=\{a_j\}+\ZZ_p$.
 \item For all $a\in F$, the differential module
  $(M_{z_i},\nabla_{z_i}-a)$ (resp. $(M_{a_j},\nabla_{a_j}-a)$) is
  pure, i.e all its spectral radii coincide with each other.
  \item Let $c_i\in F$ and $r_i>0$ be such that $z_i=x_{c_i,r_i}$. If
  $|p|^{l}\leq r_i< |p|^{l-1}$ with $l\in \NN\setminus\{0\}$, then $\car(\{z_i\}+\ZZ_p)=p^l$ and $\{z_i\}+\ZZ_p=\{x_{c_i,r_i},
  x_{c_i+1,r_i},\cdots, x_{c_i+p^l-1,r_i}\}$. If
  $r_i\geq1$, we have $\car(\{z_i\}+\ZZ_p)=1$ and
  $\{z_i\}+\ZZ_p=\{x_{c_i,r_i}\}$.
 \item If $r_i>1$, let $P_{z_i}(N(T-c_b)\dT)$ be a differential polynomial
  associated to $(\Psi_b)_*(M_{z_i},\nabla_{z_i})$ (as a differential
  module over $(\h{\Psi_b(x)},N(T-c_b)\dT)$). Then the image by
  $\pi_{\widehat{\h{\Psi_b(x)}^{alg}}/F}$ of all roots of $P_{z_i}(S)$ (as a
  commutative polynomial) is equal to $z_i$. 
 \item If $|p|^l<r_i\leq|p|^{l-1}$, let $P_{z_i}(Np^l(T-c_b)\dT)$ be a differential
  polynomial associated to $(\phi_b^l\circ\Psi_b)_*(M_{z_i},\nabla_{z_i})$ (as a differential
  module over $(\h{\Psi_b(x)^{p^l}},Np^l(T-c_b)\dT)$). Then the image by
  $\pi_{\widehat{\h{\Psi_b(x)^{p^l}}^{alg}}/F}$ of all roots of $P_{z_i}(S)$ (as a
  commutative polynomial) is equal to \[\{x_{c_i,r_i},
   x_{c_i+1,r_i},\cdots, x_{c_i+p^l-1,r_i}\}.\]
  In the special case
  where $r_i=|p|^{l-1}$ we have \[\{x_{c_i,r_i},
  x_{c_i+1,r_i},\cdots, x_{c_i+p^l-1,r_i}\}=\{x_{c_i,r_i},
  x_{c_i+1,r_i},\cdots, x_{c_i+p^{l-1}-1,r_i}\}.\]

 \item If $r_i\geq 1$, for all $a\in \disf{c_i}{r_i}\cap F$ we have
  \begin{equation}
   \RS_1(x,(M_{z_i},\nabla_{z_i}-a))=\frac{\omega}{r_i}
  \end{equation}
  and for all $a\in F\setminus \disf{c_i}{r_i}$
  \begin{equation}
   \RS_1(x,(M_{z_i},\nabla_{z_i}-a))=\cR(\delta({a-c_i}))=\frac{\omega}{|a-c_i|}.
  \end{equation}
 \item If $|p|^l\leq r_i<|p|^{l-1}$, for all $a\in
  \bigcup_{j=0}^{p^l-1} \disf{c_i+j}{r_i}\cap F$ we have
  \begin{equation}
   \RS_1(x,(M_{z_i},\nabla_{z_i}-a))=\left(\frac{|p|^l\omega}{r_i}\right)^{\fra{p^l}}
  \end{equation}
  and for all $a\in F\setminus \bigcup_{j=0}^{p^l-1} \disf{c_i+j}{r_i}$
  \begin{equation}
   \RS_1(x,(M_{z_i},\nabla_{z_i}-a))=\cR(\delta({a-c_i})).
  \end{equation}
 \item For all $a\in
  \{a_j\}+\ZZ_p$, $(M_{a_j},\nabla_{a_j}-a)$ is solvable (all its
  spectral radii are maximal), and for all $a\in
  F\setminus \{a_j\}+\ZZ_p$, we have $\RS_1(x,(M_{a_j},\nabla_{a_j}-a))=\cR(\delta({a-a_j}))$. 
\end{enumerate}
\end{Theo}

\begin{cor}\label{sec:link-betw-spectr-6}
 Let $(M,\nabla_M)$ and $(N,\nabla_N)$ be differential modules over
 $(\Hx,d_b)$. Let
 $(M,\nabla_M)=\bigoplus_{i=1}^\nu(M_{z_i},\nabla_{M,z_i})$
 (resp. $(N,\nabla_N)=\bigoplus_{j=1}^\mu(N_\omega,\nabla_{N,\omega_j})$)
 be a decomposition with respect to the spectrum as in
 \eqref{eq:58}. Then for $\{z_i\}+\ZZ_p\cap
 \{\omega_j\}+\ZZ_p=\emptyset$, any morphism $\phi: M_{z_i}\to N_{\omega_j}$ of 
 differential modules over $(\Hx,d_b)$ is equal to $0$.
\end{cor}

We start by proving the link between the spectrum and the spectral radii
of convergence, but for the special case where $b\in T_xX$
and for differential module of the form $(\F(x),\nabla_{d_b})$, where
$(\F,\nabla)$ is a differential equation over $X$.

 \begin{Pro}\label{sec:link-betw-spectr-1}
 Let $x\in X_{[2,3]}$. Let $(\F,\nabla)$ be a differential equation over
 $X$. Then there exist $z_1,\cdots,
 z_\mu\in\AF$, with $(\{z_i\}+\ZZ_p)\cap
 (\{z_j\}+\ZZ_p)=\emptyset$ for $i\ne j$, such that:
 
 \begin{equation}\label{eq:56}
   \Sigma_{\nabla_{d_b},F}(\Lk{\F(x)})= \{z_1,\cdots,
   z_\mu\}+\ZZ_p.
  \end{equation}
 For
  each $i\in \{1,\cdots, n\}$ there exists
  $z_j\in\{z_1,\cdots,z_\mu\}$ such that
  \begin{equation}\label{eq:57}
 \delta(T(z_j)-a)=
   \begin{cases}
    0& \text{ if } \RS_i(x,(\F(x),\nabla_{d_b}-a))=1,\\
    \frac{p^l\omega}{\RS_i(x,(\F(x),\nabla_{d_b}-a))^{p^l}}& \text{ if }
                             \omega^{\fra{p^{l-1}}}\leq
                             \RS_i(x,(\F(x),\nabla_{d_b}-a))
                             \leq
                                \omega^{\fra{p^{l}}},\\
    \frac{\omega}{\RS_i(x,(\F(x),\nabla_{d_b}-a))}&
                            \RS_i(x,(\F(x),\nabla_{d_b}-a))\leq \omega.
   \end{cases}
  \end{equation}
 \end{Pro}

 \begin{proof}
  We set $N:={\rm deg}(\Psi_b)$. Let $c$ such that $C_b=\couo{c}{r_1'}{r_2'}$. Then since $f_b$ is
  invertible in $\cO_X(C_b)$ and $\Psi_{b|_{C_b}}:C_b\to
  C_{\Psi_b(b)}$ has degree $N$, in all sub-closed annulus $A\subset C_b$ we have $f_b=\gamma(T-c)^N(1+h)$ with
  $h\in \cO_X(A)$ and $|h|<1$ (cf. \cite[Lemma~9.7.1/1]{Bosc},
  \cite[Lemma~1.6]{Lut93}). Since $\gcd(N,p)=1$, $(1+h)^{\fra{N}}$ is
  well defined. We set $f=(T-c)
  (1+h)^{\fra{N}}$. Then $f$ is a coordinate function on $A$
  (cf. \cite[Proposition~9.7.1/2]{Bosc}) and $d_b=f\D{f}$. Hence, the
  spectrum on any point of $b$ in $C_b$ with $d_b$ as derivation
  satisfies \eqref{eq:56} and \eqref{eq:57}
  (cf. Theorem~\ref{sec:spectr-theory-sense-1}). Note that for $y\in
  (x_{c,r'_1},x_{c,r'_2})\cup\{x\}$, the
  spectrum of $(\F(y),\nabla_{d_b})$ is the same as the spectrum of
  $(\Psi_b)_*(\F(y),\nabla_{d_b})$, as a differential module over
  $(\h{\Psi_b(y)},N(T-c_b)\dT)$. By
  Proposition~\ref{sec:vari-spectr-case-4} and the continuity of the spectral
  radii of convergence (cf. Remark~\ref{sec:radii-conv-line-1}), we deduce the
  result on $x$.
 \end{proof}
To prove Theorem~\ref{sec:link-betw-spectr} for any
 differential module over $(\Hx,d_b)$ and any $b\in T_x\mathscr{C}$, we need the following lemma. 

\begin{Lem}\label{sec:link-betw-spectr-4}
  Let $x\in X_{[2,3]}$ and $d_b$ be a bounded derivation as in \eqref{eq:17}. Let
  $((M,\nabla_l))_{l\in \NN}$ be a family of differential modules over
  $(\Hx,d_b)$, and $(M,\nabla)$ be a differential module over
  $(\Hx,d_b)$. Suppose that $\nabla_l \overset{l\to \infty}{\to}
  \nabla$ in $\Lk{M}$. Then for all $i\in \{1,\cdots,n\}$ we have $\RS_i(x,(M, \nabla_l)) \overset{l\to
   \infty}{\to}\RS_i(x,(M,\nabla))$. Moreover, if
  $\RS_i(x,(M,\nabla))\ne 1$, then there exists $l_0\in \NN$ such
  that, for all $l\geq l_0$ we have $\RS_i(x,(M, \nabla_l)) =\RS_i(x,(M,\nabla))$.
 \end{Lem}

 \begin{proof}
  If $X\subset \AF$, it is an obvious fact of
  Theorem~\ref{sec:link-betw-spectr-2}. From
  Theorem~\ref{sec:spectr-radii-diff}, we have $\Psi_{b\,
   \Omega}^{-1}(D(\Psi_b(x)))$ is a finite union of open disks.
  Moreover, for all open disk $D\subset \Psi_{b\,
   \Omega}^{-1}(D(\Psi_b(x)))$, $\Psi_{b\, \Omega|_{D}}:D\to
  D(\Psi_b(x))$ is an isomorphism. This means that after normalizing
  (return them centered at zero and with radius equal to 1)
  $D$ and $D(\Psi_b(x))$, we have
  $\Psi_{b\,\Omega|_D}(\disco{\Omega}{0}{r})=\disco{\Omega}{0}{r}$
  (see \cite[Lemma~1.11]{BP20}). This means that if $(M,\nabla)$ is a
  differential module over $(\Hx,d_b)$, then the spectral radii of
  $(\Psi_b)_*(M,\nabla)$ are exactly
  \begin{equation}
   \label{eq:55}
   \underbrace{\RS_1(x,(M,\nabla)),\cdots,
    \RS_1(x,(M,\nabla))}_{N\text{ times}},\cdots\cdots, \underbrace{\RS_n(x,(M,\nabla)),\cdots,
    \RS_n(x,(M,\nabla))}_{N\text{ times} }.
  \end{equation}
  Since the result holds for $(\Psi_b)_*(M,\nabla)$ then the result holds for $(M,\nabla)$.
 \end{proof}

 \begin{rem}
  If $\RS_i(x,(M,\nabla))=1$, we may ask whether we have $\RS_i(x,(M,\nabla_l))=1$ for $l$ greater
  than some $l_0$. We have the following counter-example, for
  $x=x_{0,1}$, we choose $d=\dT$, $\nabla_l=d+p^{l-1}T^{p^l}$ and $\nabla=d$.
 \end{rem}

 \begin{proof}[Proof of Theorem~\ref{sec:link-betw-spectr}]
Let $N:={\rm deg}(\Psi_b)$. We know that any $F$-analytic curve is a good analytic space
(cf. \cite[Proposition~3.3.7]{Duc}), therefore $\cO_{X,x}$ is dense in
$\Hx$. Let $G\in\cM_n(\Hx)$ be the associated matrix of $\nabla_{d_b}$
in a basis $\mathscr{B}$ of $M$. There exists a family $(G_l)_l\in\cM_n(\cO_{X,x})$
such that $G_l\to G$ in $\cM_n(\Hx)$. For all $l\in\NN$, there exists an affinoid
neighborhood $V_l$ of $x$ in $Y$ such that the coefficients of $G_l$
are analytic functions on $V_l$. Hence, we can define a differential
equation $(\F_l,\nabla_l)$ such that $\F_l(x)=M$ and $G_l$ is the
associated matrix of $\nabla_{l,d_b}$ in $\mathscr{B}$. Hence, we have
$\nabla_{l,d_b}\to \nabla$. By Theorem~\ref{sec:link-betw-spectr-2}
(for $\fra{N}(\Psi_b)_*(\nabla_{l,d_b})$ and $\fra{N}(\Psi_b)_*(\nabla)$),
Lemma~\ref{sec:link-betw-spectr-4} and
Proposition~\ref{sec:link-betw-spectr-1}, the spectrum satisfies
\eqref{eq:56} and \eqref{eq:57}. Therefore, we can obtain the
decomposition \eqref{eq:58} by Theorem~\ref{sec:robb-decomp-spectr},
the remaining statement follows directly.
\end{proof}
\begin{rem}\label{sec:link-betw-spectr-5}
By proving Theorem~\ref{sec:link-betw-spectr}, we deduce that
Theorem~\ref{sec:link-betw-spectr-2} holds for any
differential module $(M,\nabla)$ over $(\Hx,d_b)$, where $x$ is a
point of a quasi-smooth analytic curve and $d_b$ defined as in the
beginning of the section.
\end{rem} 
 
\subsection{Variation of the spectrum, the case of a quasi-smooth
 curve}
Since
\begin{equation}
 \label{eq:59}
 \Sigma_{\nabla_{d_b},F}(\Lk{\F(y)})=\phi(\Sigma_{(\Psi_b)_*\nabla_{d_{(T-c)}},F}(\Lk{(\Psi_b)_*\F(\Psi_b(y))}),
\end{equation}
with $\phi:\AF\to \AF$ $T\mapsto NT$ and $y\in (b\cap C_b)\cup \{x\}$,
Proposition~\ref{sec:vari-spectr-case-4} holds in $(b\cap C_b)\cup
\{x\}$. Therefore, we can define a notion of controlling graph. Indeed, we have the
following results as consequence of Theorem~\ref{sec:link-betw-spectr}
and Proposition~\ref{sec:vari-spectr-case-4}.
\begin{Pro}\label{sec:vari-spectr-case-7}
 Let $\cS$ be a weak triangulation on $X$, $(\F,\nabla)$ be a
 differential equation over $X$ and $D\subset
 X\setminus\Gamma_{\cS}$ be an open disk. Suppose that the radii of
 convergence of $(\F,\nabla)$ are constant in $D$ (respectively
 identically equal to $R_1,\cdots,R_n$ in $D$). Then for all $x\in D$
 not of type (1) or (4),
 and for all $b\in T_xX$ we have
 \begin{equation}
 \Sigma_{\nabla_{d_b},F}(\Lk{\F(x)})=\bigcup_{i=1}^n\{x_{0,\phi_i(x)}\}+\ZZ_p,
  \end{equation}
  with
  \begin{equation}
   \phi_i(x)=
   \begin{cases}
    0& \text{ if } r_{\cS,F}(x)\leq R_i,\\
    \frac{|p|^l\omega \cdot r_{\cS,F}(x)^{p^l}}{R_i^{p^l}} & \text{ if }
                         \omega^{\fra{p^{l-1}}}\cdot
                         r_{\cS,F}(x)\leq
                         R_i \leq
                         \omega^{\fra{p^{l}}}\cdot
                         r_{\cS,F}(x),
                         \text{ with }
                         l\in\NN\setminus\{0\},\\
    \frac{\omega \cdot r_{\cS,F}(x)}{R_i} & R_i\leq \omega \cdot r_{\cS,F}(x),\\
   \end{cases}
  \end{equation}
  (cf. Definition~\ref{sec:shape-quasi-smooth-2}).
 \end{Pro}
 \begin{Defi}
  \begin{proof}
   Consequence of Corollary~\ref{sec:vari-spectr-case-5}.
  \end{proof}
 Let $\cS$ be a weak triangulation on $X$. Let $(\F,\nabla)$ be a
 differential equation over $X$, and let $\Gamma$ be a subgraph of $X$. We say that
 $\Gamma$ is a {\it controlling graph of the spectrum} of $(\F,\nabla)$
 with respect $\cS$, if and only if the
 following conditions are satisfied:
 \begin{enumerate}
 \item the points of $\Gamma$ are of type (2) or (3);
 \item $\Gamma_{\cS}\subset \Gamma$;
 \item $X\setminus\Gamma$ is a disjoint union of open disks;
 \item for each $D\subset X\setminus\Gamma$, let $x\in \Gamma$ with
  $\overline{D}=D\cup \{x\}$, for all $y\in D$ not of type (1) or (4) we have
  \begin{equation}
\Sigma_{\nabla_{d_b},F}(\Lk{\F(y)})=\bigcup_{i=1}^n\{x_{0,\phi_i(y)}\}+\ZZ_p,
  \end{equation}
 with
  \begin{equation}
\phi_i(y)=
   \begin{cases}
    0& \text{ if } r_{\cS,F}(y)\leq\cR_{\cS,i}(x,(\F,\nabla)),\\
    \frac{|p|^l\omega \cdot r_{\cS,F}(y)^{p^l}}{\cR_{\cS,i}(x,(\F,\nabla))^{p^l}} & \text{ if }
                         \omega^{\fra{p^{l-1}}}\cdot
                         r_{\cS,F} (y)\leq
                         \cR_{\cS,i}(x,(\F,\nabla))\leq
                         \omega^{\fra{p^{l}}}\cdot
                         r_{\cS,F} (y),
                         \text{ with }
                         l\in\NN\setminus\{0\},\\
    \frac{\omega \cdot r_{\cS,F} (y)}{\cR_{\cS,i}(x,(\F,\nabla))} & \cR_{\cS,i} (x,(\F,\nabla))\leq \omega\cdot r_{\cS,F} (y),\\
   \end{cases}
  \end{equation}
  (cf. Definition~\ref{sec:shape-quasi-smooth-2}).
 \end{enumerate} 
\end{Defi}
\begin{rem}
 These allow us to deduce the spectrum on
 the whole $X$ only by knowing it on $\Gamma$.
\end{rem}

\begin{Pro}
 The controlling graph of $\Gamma(\cR_\cS,\F)$ of
 $\cR_S(.,(\F,\nabla))$ (cf. Definition~\ref{sec:radii-conv-line-4}), with respect to $\cS$, is the smallest controlling graph of
 the spectrum of $(\F,\nabla)$.
\end{Pro}

\begin{proof}
 Let $\Gamma$ be a controlling graph of the spectrum of $(\F,\nabla)$. By
 the Transfer Theorem we have $X\setminus\Gamma \subset
 X\setminus\Gamma(\cR_\cS)$. Hence, $\Gamma(\cR_\cS,\F)\subset \Gamma$. By
 Proposition~\ref{sec:vari-spectr-case-7}, $\Gamma(\cR_\cS,\F)$ is a
 controlling graph of the spectrum. 
\end{proof}

\section{Finer decomposition with respect to the
 spectrum}\label{sec:refin-decomp-with}
\begin{conv}
 In this section we assume that $F$ is algebraically closed and $\crk=p>0$.
\end{conv}
Let $\cS$ be a weak triangulation of $X$ and $x\in X_{[2,3]}$. Let
$\mathscr{C}$ be a quasi-smooth algebraic curve, such
that there exists an affinoid neighborhood of $x$ in $X$ isomorphic
to an affinoid domain of $\mathscr{C}$, and $x$ lies in the interior
of $\mathscr{C}$. Let $b$, $b'\in T_x\mathscr{C}$, then
there exists $g\in \Hx$ such that $d_{b'}=gd_b$
(cf. \eqref{eq:17}). For $(M,\nabla)$ a differential module over
$(\Hx,d_b)$, we set $\nabla_{b'}=g\cdot\nabla$. Let $(\F,\nabla)$ be a
differential equation over $X$, let $\Gamma(\cR_{\cS,\F})$
be the controlling graph of radii of convergence of $(\F,\nabla)$, with respect to
$\cS$. Then we have $T_x\Gamma(\cR_{\cS,\F})\cup (T_x\mathscr{C}\setminus T_x
X)=\{b_1,\cdots,b_{N_x}\}$, we fix $b\in T_xX\setminus T_x\Gamma(\cR_{\cS,\F})$.

\begin{Theo}\label{sec:finer-decomp-with-1}
 Let $(M,\nabla)$ be a differential module over $(\Hx,d_b)$, with $(M,\nabla)=(\F(x),\nabla_{d_b})$. For each
 $b_i\in \{b_1,\cdots,b_{N_x}\}$, let
 $\Sigma_{b_i}$ be a finite subset of $\AF$ such that for all $ z, z'\in \Sigma_{b_i}$, with
 $z\ne z'$ we have $\{z\}+\ZZ_p\cap \{z'\}+\ZZ_p=\emptyset$ and
 $\Sigma_{\nabla_{b_i}}(\Lk{M})=\Sigma_{b_i}+\ZZ_p$. We set
 $\Sigma:=\prod_{i=1}^{N_x}\Sigma_{b_i}$. For $v\in \Sigma$, we have
 $v=(v_1,\cdots, v_{N_x})$, where $v_i \in \Sigma_{b_i}$ for
 $i=1,\cdots,N_x$. Then there exists a unique (up to an isomorphism) decomposition
 \begin{equation}
  \label{eq:61}
  (M,\nabla)=\bigoplus_{v\in\Sigma}(M_v,\nabla_v),
 \end{equation}
 such that $(M_{v},\nabla_{v})$ are differential modules over
 $(\Hx,d_b)$, if $M_v\ne 0$, then each sub-quotient of $(M_v,(\nabla_v)_{b_i})$
 has spectrum equal to $\{v_i\}+\ZZ_p$.
\end{Theo}

\begin{proof}
 First of all we prove that such a decomposition exists. Let
 $l\in\{1,\cdots,N_x\}$, we set $\Sigma_l=\prod_{i=1}^l\Sigma_{b_i}$
 and for $v\in \Sigma$ we set $v^l=(v_1,\cdots, v_l)$. By induction on $l$, we 
 prove that there exists a decomposition 
 \begin{equation}
  \label{eq:38}
  (M,\nabla)=\bigoplus_{v^l\in\Sigma_l}(M_{v^l},\nabla_{v^l}),
 \end{equation}
 such that $(M_{v^l},\nabla_{v^l})$ are differential modules over
 $(\Hx,d_b)$, and if $M_{v^l}\ne 0$, then for all $i\in\{1,\cdots,l\}$
 each sub-quotient of $(M_{v^l},(\nabla_{v^l})_{b_i})$ has spectrum
 equal to $\{v_i\}+\ZZ_p$. For $l=1$, we use
 Theorem~\ref{sec:link-betw-spectr} to decompose $(M,\nabla_{b_1})$
 as a differential module over $(\Hx,d_{b_1})$. Then it induces a
 decomposition of $(M,\nabla)$ as a differential module $(\Hx,d_b)$,
 which is the following:
 \begin{equation}
  \label{eq:37}
  (M,\nabla)=\bigoplus_{v^1\in\Sigma_1}(M_{v^1},\nabla_{v^1}),
 \end{equation}
  such that $(M_{v^1},\nabla_{v^1})$ are differential modules over
 $(\Hx,d_b)$, and if $M_{v^1}\ne 0$, then each sub-quotient of $(M_{v^1},(\nabla_{v^1})_{b_1})$ has spectrum
 equal to $\{v_1\}+\ZZ_p$. Suppose now that for $l<N_x$ we have a
 decomposition as in \eqref{eq:38}. Then we have a decomposition
 \begin{equation}
  \label{eq:41}
  (M,\nabla)=\bigoplus_{v^{l+1}\in\Sigma_{l+1}}(M_{v^{l+1}},\nabla_{v^{l+1}})
 \end{equation}
 defined as follows. For each $v^l$, if
 $v_{l+1}\in\Sigma_{(\nabla^l)_{b_{l+1}},F}(\Lk{M_{v^l}})$ and
 $M_{v^l}\ne 0$, then $(M_{v^{l+1}},\nabla_{v^{l+1}})$ is the
 differential over $(\Hx,d_b)$ such that
 $(M_{v^{l+1}},(\nabla_{v^{l+1}})_{{b_{l+1}}})$ corresponds to the
 term of the decomposition \eqref{eq:58} of
 $(M_{v^{l}},(\nabla_{v^{l}})_{{b_{l+1}}})$, that satisfies
 $\Sigma_{v_{l+1},F}(\Lk{M_{v_{l+1}}})=\{v_{l+1}\}+\ZZ_p$. In the other case, we simply
 take $M_{v_{l+1}}=0$. In particular we have
 \begin{equation}
  \label{eq:44}
  (M_{v^l},(\nabla_{v^l})_{b_{l+1}})=\bigoplus_{w^{l+1}\in\Sigma_{l+1},w^l=v^l}(M_{w^{l+1}},(\nabla_{w^{l+1}})_{b_{l+1}}).
 \end{equation}
 Hence, we obtain the desired
 decomposition. Consequently, to obtain a decomposition as in
 \eqref{eq:61}, we simply choose a decomposition as \eqref{eq:38}
 with $l=N_x$.

 Now we prove the unicity. Suppose that we have two
 decompositions as in \eqref{eq:61}
 \begin{equation}
  \label{eq:43}
   (M,\nabla)=\bigoplus_{v\in\Sigma}(M_v,\nabla_v)=\bigoplus_{v\in\Sigma}(M'_v,\nabla'_v),
  \end{equation}
  Let $i_v: M_v\to M$ (resp. $i'_v:M'_v\to M$) be the canonical
  injection and $\pi_v:M\to M_v$ (resp. $\pi'_v: M\to M'_v$) be the
  canonical projection. Then for all $v\ne w$, we have $i_v\circ
  \pi'_w=0$ (resp. $i'_v\circ \pi_w=0$). Indeed, in the case where
  $M_v=0$ or $M'_w=0$ (resp. $M'_v=0$ or $M_w=0$) it
  obvious. In the case where $M_v\ne 0$
  and $M'_w\ne 0$ (resp. $M'_v\ne 0$ or $M_w\ne 0$) , if we assume
  that $i_v\circ\pi'_w\ne 0$ (resp. $i'_v\circ \pi_w\ne 0$), for $v_{i_0}\ne
  w_{i_0}$ we get a sub-quotient of
  $(M_v,(\nabla_v)_{b_{i_0}})$ (resp. $(M'_v,(\nabla'_v)_{b_{i_0}})$)
  with spectrum equal to $\{w_{i_0}\}+\ZZ_p$, which is absurd. Therefore,
  $i_v\circ \pi'_v$ and $i'_v\circ \pi_v$ are injective, hence
  $(M_v,\nabla_v)\simeq (M'_v,\nabla'_v)$.   
  
\end{proof}

\begin{rem}
 This decomposition is finer than the one
 provided by the radii of convergence, and even finer than the
 decomposition in Theorem~\ref{sec:link-betw-spectr}, it is shown in
 Example~\ref{sec:introduction-3}.
\end{rem}
\subsection{Dwork-Robba's decomposition} In this part we prove that
we have decomposition with respect to the spectrum analogous to
Dwork-Robba's decomposition. Let $X$ be a quasi-smooth curve and $x\in
X_{[2,3]}$. Let $\mathscr{C}$ be a quasi-smooth algebraic curve, such
that there exists an affinoid neighborhood of $x$ in $X$ isomorphic
to an affinoid domain of $\mathscr{C}$, and $x$ lies in the interior
of $\mathscr{C}$. We fix $b\in T_x\mathscr{C}$ and $d_b$ as in \eqref{eq:17}.

\begin{Theo}\label{sec:dwork-robb-decomp}
 Let $(M_x,\nabla_x)$ be a differential module over
 $(\cO_{X,x},d_b)$. We set
 $(M,\nabla)=(M_x\ct_{\cO_{X,x}}\Hx,\nabla_{x}\ot 1+1\ot d_b)$. Let
 $\Sigma$ be a finite subset of $\AF$ such that for $z,z'\in\Sigma$, with
 $z\ne z'$, we have $\{z\}+\ZZ_p\cap\{z'\}+\ZZ_p$ and
 $\Sigma_{\nabla,F}(\Lk{M})=\Sigma+\ZZ_p$. Then we have a unique decomposition
 \begin{equation}
  \label{eq:45}
  (M_x,\nabla_x)=\bigoplus_{z\in\Sigma}(M_{x,z},\nabla_{x,z}),
 \end{equation}
 such that $(M_{x,z},\nabla_{x,z}\ot 1+1\ot d_b)\simeq
 (M_z,\nabla_z)$ (term of decomposition in \eqref{eq:58}). 
\end{Theo}
\begin{proof}
 Let $\Sigma=\{z_1,\cdots,z_\nu\}$ with $z_i=x_{c_i,r_i}$, $r_i\geq
 0$ and $r_1\leq\cdots\leq r_\nu$. Firstly we prove the existence of
 such decomposition. Let $l\in\{1,\cdots,\nu-1\}$. By
 induction on $l$ we prove that we have the following decomposition
 \begin{equation}
  \label{eq:46}
  (M_x,\nabla_x)=\bigoplus_{i=1}^l(M_{x,z_i},\nabla_{x,z_i})\oplus (M_{x,l},\nabla_{x,l})
 \end{equation}
 such that $(M_{x,z_i},\nabla_{x,z_i}\ot 1+1\ot d_b)\simeq
 (M_{z_i},\nabla_{z_i})$, and $\Sigma_{\nabla_l,F}(\Lk{M_l})=\{z_{l+1},\cdots,z_\nu\}$, where
 $(M_l,\nabla_l)\simeq (M_{x,l}\ct_{\cO_{X,x}}\Hx,\nabla_{x,l}\ot
 1+1\ot d_b)$. For $l=1$, by
 Theorem~\ref{sec:link-betw-spectr} we have
 \begin{equation}
  \label{eq:50}
  (M,\nabla)=\bigoplus_{i=1}^\nu(M_{z_i},\nabla_{z_i}),
 \end{equation}
 with $\Sigma_{\nabla_{z_i}}(\Lk{M_{z_i}})=\{z_i\}+\ZZ_p$. Hence, we
 obtain
 \begin{equation}
  \label{eq:74}
   (M,\nabla-c_1)=\bigoplus_{i=1}^\nu(M_{z_i},\nabla_{z_i}-c_1).
  \end{equation}
  By Theorem~\ref{sec:link-betw-spectr}, we have
  $\RS_1(x,(M_{z_i},\nabla_{z_i}-c_1))=\min
  (\cR(r_i),\cR(\delta(c_i-c_1)))$ (cf. \eqref{eq:94}), in particular we
  have $\RS_1(x,(M_{z_1},\nabla_{z_1}-c_1))=\cR(r_1)$. Since $r_i\geq
  r_1$ and $\{z_i\}+\ZZ_p\cap \{z_1\}+\ZZ_p=\emptyset$, we should
  have $\RS_1(x,(M_{z_i},\nabla_{z_i}-c_1))< \cR(r_1)$ for $i\ne
  1$. By Theorem~\ref{sec:robb-decomp-spectr} we have
  \begin{equation}
   \label{eq:75}
   (M,\nabla-c_1)=\bigoplus_{\rho\in (0,1]}(M_{c_1}^\rho,\nabla_{c_1}^\rho),
  \end{equation}
  such that if
  $\rho\in\{\RS_1(x,(M,\nabla-c_1)),\cdots,\RS_n(x,(M,\nabla-c_1))\}$, then
  $(M_{c_1}^\rho,\nabla_{c_1}^\rho)$ is pure and
  $\RS_1(x,(M_{c_1}^\rho,\nabla_{c_1}^\rho))=\rho$. Otherwise
  $M^\rho_{c_1}=0$. By the unicity of the decomposition we should
  have $(M_{c_1}^{\cR(r_1)},\nabla_{c_1}^{\cR(r_1)})\simeq
  (M_{z_1},\nabla_{z_1}-c_1)$. On the other hand, by
  Theorem~\ref{sec:robb-decomp-spectr-2}, we have
  \begin{equation}
   \label{eq:76}
   (M_x,\nabla_x-c_1)=\bigoplus_{\rho\in (0,1]}(M_{x,c_1}^\rho,\nabla_{x,c_1}^\rho),
  \end{equation}
  with $(M_{x,c_1}^\rho\ct_{\cO_{X,x}}\Hx,\nabla_{x,c_1}^\rho\ot 1+
  1\ot d_b)\simeq (M_{c_1}^\rho,\nabla_{c_1}^\rho)$. We set
  $(M_{x,z_1},\nabla_{x,z_1})=(M_{x,c_1}^{\cR(r_1)},\nabla_{x,c_1}^{\cR(r_1)}+c)$
  and $(M_{x,1},\nabla_{x,1})=\bigoplus_{\rho\in
   (0,1]\setminus\{\cR(r_1)\}}(M_{x,c_1}^\rho,\nabla_{x,c_1}^\rho+c)$. By
  construction the decomposition
  $(M_x,\nabla_x)=(M_{x,1},\nabla_{x,1})\oplus
  (M_{x,1},\nabla_{x,1})$ is as in \eqref{eq:46} with $l=1$. We
  suppose now that we have decomposition \eqref{eq:46} for $l-1$, with
  $l>1$. In order to obtain decomposition \eqref{eq:46} for $l$, we
  decompose $(M_{x,l-1},\nabla_{x,l-1})$ with the same strategy as
  for $(M_x,\nabla_x)$ in
  the case $l=1$. Hence we obtain
  $(M_{x,l-1},\nabla_{x,l-1})=(M_{x,z_l},\nabla_{x,z_l})\oplus
  (M_{x,l},\nabla_{x,l})$, such that
  $(M_x,\nabla_x)=\bigoplus_{i=1}^l(M_{x,z_i},\nabla_{x,z_i})\oplus
  (M_{x,l},\nabla_{x,l})$ is as in \eqref{eq:46}. Hence,
  decomposition \eqref{eq:46} for $l=\nu-1$ is as decomposition
  \eqref{eq:45}.

  We prove now the unicity of the decomposition. Let
  $(M_x,\nabla_x)=\bigoplus_{i=1}^\nu(M'_{x,z_i},\nabla'_{x,z_i})$
  be another decomposition. Let $i_{x,z_i}: M_{x,z_i}\to M_x$
  (resp. $i'_{x,z_i}:M'_{x,z_i}\to M_x$, resp. $i_{z_i}:M_{z_i}\to M$) be the canonical
  injection and $\pi_{x,z_i}:M_x\to M_{x,z_i}$ (resp. $\pi'_{x,z_i}:
  M_x\to M'_{x,z_i}$, resp. $\pi_{z_i}:M\to M_{z_i}$) be the
  canonical projection. Then for all $i\ne j$, we have $i_{z_i}\circ
  \pi'_{z_j}=0$ and $i'_{z_i}\circ \pi_{z_j}=0$. Indeed, if we assume that $i_{x,z_i}\circ\pi'_{x,z_j}\ne 0$
  (resp. $i'_{x,z_i}\circ \pi_{x,z_j}\ne 0$), since $\cO_{X,x}$ is a
  field, after scalar extension by $\Hx$ we obtain $i_{z_i}\circ\pi_{z_j}\ne 0$
  which is absurd (cf. Corollary~\ref{sec:link-betw-spectr-6}). Therefore, $i_{x,z_i}\circ\pi'_{x_i}$ and
  $i'_{x,z_i}\circ\pi_{x_i}$ are injective. Consequently,
  $(M'_{x,z_i},\nabla'_{x,z_i})\simeq (M_{x,z_i},\nabla_{x,z_i})$.  

 \end{proof}
 Let $(\F,\nabla)$ be a
differential equation over $X$, let $\Gamma(\cR_{\cS,\F})$
be the controlling graph of radii of convergence of $(\F,\nabla)$, with respect to
$\cS$. Then we have $T_x\Gamma(\cR_{\cS,\F})\cup (T_x\mathscr{C}\setminus T_x
X)=\{b_1,\cdots,b_{N_x}\}$, we fix $b\in T_xX\setminus T_x\Gamma(\cR_{\cS,\F})$.

\begin{Theo}\label{sec:dwork-robb-decomp-1}
Let $(M_x,\nabla_x)$ be a differential module over $(\cO_{X,x},d_b)$,
with $(M_x,\nabla_x)=(\F_x,\nabla_{d_b})$. For each $b_i\in
\{b_1,\cdots,b_{N_x}\}$, let $\Sigma_{b_i}$ be a finite subset of
$\AF$ such that for all $z, z'\in \Sigma_{b_i}$ with $z\ne z'$, we
have $\{z\}+\ZZ_p\cap \{z'\}+\ZZ_p=\emptyset$ and
$\Sigma_{(\nabla_x){b_i},F}(\Lk{M_x\ct_{\cO_{X,x}}\Hx})=\Sigma_{b_i}+\ZZ_p$. We set $\Sigma:=\prod_{i=1}^{N_x}\Sigma_{b_i}$. For $v\in \Sigma$, we have $v=(v_1,\cdots, v_{N_x})$, where $v_i \in \Sigma_{b_i}$ for $i=1,\cdots,N_x$. Then
\begin{equation}
(M_x,\nabla_x)=\bigoplus_{v\in\Sigma}(M_{x,v},\nabla_{x,v}),
\end{equation}
such that if $M_{x,v}\ne 0$, then each sub-quotient of
$(M_{x,v}\ct_{\cO_{X,x}}\Hx,(\nabla_{x,v})_{b_i})$ has spectrum equal to
$\{v_i\}+\ZZ_p$. 
\end{Theo}

\begin{proof}
Since we have Theorem~\ref{sec:dwork-robb-decomp}, the prove is
analogous to the prove of Theorem~\ref{sec:finer-decomp-with-1}.
\end{proof}

\subsection{Relation with the refined decomposition}\label{sec:relat-with-refin-1}

From the spectrum we can recognize some refined differential
module (a notion that was initial introduced by Xiao in
\cite{Xia12}). However, in general the decomposition provided by the spectrum is
neither coarser nor finer then the
refined decomposition. In order to explain this, we start by giving some definitions.

\begin{Defi}
Let $x\in \AF$ of type (2) and $d:\Hx\to \Hx$ be a nonzero bounded $F$-linear derivation. Let $(M,\nabla)$ be a nonzero differential module over
$(\Hx,d)$. We say that $(M,\nabla)$ is {\it refined} if and only if
\begin{equation}
 \label{eq:63}
 \RS_1(x,(M,\nabla))<\RS_1(x,(M\ot_{\Hx}M^*,\nabla\ot 1+1\ot
 \nabla^*)).
\end{equation}
We say that two refined differential module$(M_1,\nabla_1)$
and $(M_2,\nabla_2)$ over $(\Hx,d)$ are equivalent if and only if
\begin{equation}
 \label{eq:64}
 \RS_1(x,(M_1,\nabla_1)=\RS_1(x,(M_2,\nabla_2))<\RS_1(x,(M_1\ot_{\Hx}{M^*}_2,\nabla_1\ot 1+1\ot
 {\nabla^*}_2)),
\end{equation}
and we write $(M_1,\nabla_1)\sim (M_2,\nabla_2)$. This condition
defines an equivalence relation on refine differential modules (cf.\cite[Lemma~6.2.14]{Ked}).
\end{Defi}

\begin{rem}
 We point out that for any differential module $(M,\nabla)$ over
 $(\Hx,d)$, for $a\in F$ with $|a|$ big enough $(M,\nabla-a)$ is refined. 
\end{rem}

\begin{Theo}[{\cite[Theorem~10.6.7]{Ked}}]\label{sec:relat-with-refin}
 Let $x\in \AF$ of type (2) and $d:\Hx\to \Hx$ be a nonzero bounded $F$-linear
 derivation. Let $(M,\nabla)$ be a differential module over $(\Hx,d)$
with non solvable radii of convergence. Then, for some finite tamely
ramified extension $\Omega$ of $\Hx$, $M\ot_{\Hx}\Omega$ admits a
unique direct sum decomposition
\begin{equation}
 \label{eq:65}
 M\ot_{\Hx}\Omega=\bigoplus_i M_i,
\end{equation}
with each $M_i$ refined, such that for $i\ne j$ we have $M_i\not\sim M_j$. 
\end{Theo}

Let $x=x_{0,\rho}$, and let $(M,\nabla)$ be a differential
module over $(\Hx, T\dT)$. Suppose that
\[\Sigma_{\nabla,F}(\Lk{M})=\{x_{a_1,r_1},\cdots, x_{a_\nu,r_\nu}\}+\ZZ_p,\]
with $\{x_{a_i,r_i}\}+\ZZ_p\cap\{x_{a_j,r_j}\}+\ZZ_p=\emptyset$ for $i\ne j$,
and $r_i\ne 0$ for all $i$. By Theorem~\ref{sec:spectr-theory-sense-1}
we have
\begin{equation}
 \label{eq:66}
 (M,\nabla)=(M_{x_{a_1,r_1}},\nabla_{x_{a_1,r_1}})\oplus\cdots\oplus (M_{x_{a_\nu,r_\nu}},\nabla_{x_{a_\nu,r_\nu}}),
\end{equation}
where
$\Sigma_{\nabla_{x_{a_i,r_i}},F}(\Lk{M_{x_{a_i,r_i}}})=\{x_{a_i,r_i}\}+\ZZ_p$. In the
case where $|a_i|>r_i$, the differential module
$(M_{x_{a_i,r_i}},\nabla_{x_{a_i,r_i}})$ is refined. This means that if a differential module, with
spectrum equal to $\{x\}+\ZZ_p$, is not refined then
$x=x_{0,r}$. In the case where $(M_{x_{a_i,r_i}},\nabla_{x_{a_i,r_i}})$
and $(M_{x_{a_j,r_j}},\nabla_{x_{a_j,r_j}})$ are refined for $i\ne j$, in
general, we do not have $(M_{x_{a_i,r_i}},\nabla_{x_{a_i,r_i}})\not \sim
(M_{x_{a_j,r_j}},\nabla_{x_{a_j,r_j}})$. Indeed, for
example, if $a_i=a_j$ and $r_i<r_j<|a_i|$ we have $(M_{x_{a_i,r_i}},\nabla_{x_{a_i,r_i}})\sim
(M_{x_{a_j,r_j}},\nabla_{x_{a_j,r_j}})$. Moreover, in such situation
$(M_{x_{a_i,r_i}},\nabla_{x_{a_i,r_i}})\oplus
(M_{x_{a_j,r_j}},\nabla_{x_{a_j,r_j}})$ is refined. Hence, in general the decomposition provided by the spectrum is
neither coarser nor finer then the
refined decomposition.

In order to obtain a decomposition that is finer than both the refined decomposition and the decomposition with respect to the spectra corresponding to several directions, we propose the following definition.
\begin{Defi}
Let $x\in \AF$ of type (2) and $d:\Hx\to \Hx$ be a nonzero bounded $F$-linear derivation. Let $(M,\nabla)$ be a nonzero differential module over
$(\Hx,d)$. We say that $(M,\nabla)$ is {\it strongly refined} if and
only if for all differential module of rank 1 $(N,D)$
\begin{equation}
\label{eq:69}
 \RS_1(x,(M\ot N,\nabla\ot 1+ 1\ot D))<\RS_1(x,(M\ot_{\Hx}M^*,\nabla\ot 1+1\ot
 \nabla^*)),
\end{equation}
in other words, $(M\ot N,\nabla\ot 1+ 1\ot D)$ is refined.

We say that two strongly refined differential modules $(M_1,\nabla_1)$
and $(M_2,\nabla_2)$ over $(\Hx,d)$ are strongly equivalent if and
only if for all differential module of rank 1 $(N,D)$
\begin{equation}
 \RS_1(x,M_1\ot_{\Hx}N)=\RS_1(x,M_2\ot_{\Hx}N
 )<\RS_1(x,M_1\ot_{\Hx}{M_2}^*),
\end{equation}
in other word, $M_1\ot N\sim M_2\ot N$.
This condition
defines an equivalence relation on strongly refined differential modules. 
\end{Defi}

\begin{rem}
 The tensor product of a solvable differential module with a rank one
 differential module cannot be strongly refined.
\end{rem}
\begin{nota}
 Let $(\Omega,d)$ be a differential field. We set
 $\DD_{\Omega,d}:=\bigoplus_{i\in\NN}\Omega. d^i$ to be the ring of
 differential polynomials on $d$ with coefficients in $\Omega$. If
 there is no confusion about the derivation we simply denote it
 by $\DD_\Omega$. 
\end{nota}
Using this definition, we can reformulate Theorem~\ref{sec:relat-with-refin} in a strong version.
\begin{Theo}
 Let $x\in \AF$ of type (2) and $d:\Hx\to \Hx$ be a nonzero bounded
 $F$-linear derivation. Let $(M,\nabla)$ be a differential module over $(\Hx,d)$. Then, for some finite tamely
ramified extension $\Omega$ of $\Hx$, $M\ot_{\Hx}\Omega$ admits a
unique direct sum decomposition
\begin{equation}
 \label{eq:70}
 M\ot_{\Hx}\Omega=M_0\oplus \bigoplus_{i=1}^{\nu} M_i,
\end{equation}
where 
\begin{equation}
 \label{eq:71}
 M_0=\bigoplus_{i=1}^\mu N_i\ot \DD_\Omega/\DD_\Omega\cdot (d+h_i),
\end{equation}
with each $N_i$ is a solvable differential module and $h_i\in \Omega$, each $M_i$ is strongly refined, and for $i\ne j$ the differential modules
$M_i$ and $M_j$ are not strongly equivalent. 
\end{Theo}
We suspect that by computing the spectrum with respect to all the
derivation of the form $d_{c,l}=(T-c)^l\dT$ with
$l\in \ZZ$, and
proving a decomposition theorem as in
Theorem~\ref{sec:link-betw-spectr}, the resulting decomposition will not
depend on the choice of the coordinates, and it will coincide with the strong refined decomposition if
it exists.

\printbibliography
\textsc{Tinhinane Amina, Azzouz}
\end{document}